\documentclass[11pt, twoside, leqno]{amsart}  
\usepackage{soul}
\usepackage{lipsum}
\usepackage{amsfonts}
\usepackage{graphicx}
\usepackage{epstopdf}
\usepackage{calligra}
\usepackage{amsfonts,amsmath,amsthm,amssymb}
\usepackage{mathtools}
\usepackage{hyperref}
\usepackage{autonum}
\usepackage{hhline}
\usepackage{array}
\usepackage{diagbox}
\usepackage{tcolorbox}
\usepackage{mdframed}
\usepackage{multicol}
\usepackage{graphicx}
\usepackage{subcaption}
\usepackage{moreverb}
\usepackage{bbm}
\usepackage[margin=1.3in]{geometry}
\usepackage{todonotes}
\usepackage{scalerel,amssymb}
\allowdisplaybreaks
\usepackage{mathrsfs}  
\usepackage{lineno}
\usepackage{todonotes}
\usepackage{tikz}
\usepackage{pgfplots}
\usetikzlibrary{arrows.meta}
\usepackage[numbers,sort&compress]{natbib}
\definecolor{mygreen}{HTML}{43a047}
\usepackage{subcaption}
\usepackage{doi}
\usepackage{alphalph}
\usepackage{booktabs}
\usepackage{makecell}
\usepackage{adjustbox}
\usepackage{appendix}
\usepackage{makecell}
\usepackage{array}
\newcolumntype{H}{>{\setbox0=\hbox\bgroup}c<{\egroup}@{}}

\def\calN{\mathcal{N}}
\def\calB{\mathcal{B}}
\def\calT{\mathcal{T}}
\def\calF{\mathcal{F}}

\def\ulaaa{\underline{\mathfrak{a}}}
\def\olaaa{\overline{\mathfrak{a}}}
\def\olbbb{\overline{\mathfrak{b}}}

\newcommand{\Om}{\Omega}
\newcommand{\D}{\Delta}

\newcommand{\Dt}{\textup{D}_t}

\def\aaa{\mathfrak{a}}
\def\bbb{\mathfrak{b}}

\def \xin{\xi^{(n)}}
\def \xit{\xi^{(n)}_t}
\def \xitt{\xi^{(n)}_{tt}}
\def \xittt{\xi^{(n)}_{ttt}}
\def \bxin{\boldsymbol{\xi}}
\def \bxit{\boldsymbol{\xi_t}}
\def \bxitt{\boldsymbol{\xi_{tt}}}
\def \bxittt{\boldsymbol{\xi_{ttt}}}
\def \ut{u_t}
\def \utt{u_{tt}}
\def \uttt{u_{ttt}}

\def \psitt{\psi_{tt}}
\def \psittt{\psi_{ttt}}

\def \taua {\tau^a}

\def\ppsi{u}
\def\ppsit{u_t}
\def\ppsitt{u_{tt}}
\def\ppsittt{u_{ttt}}
\def\ppsin{u^{(n)}}
\def\ppsin{u^{(n)}}
\def\ppsint{u_t^{(n)}}
\def\ppsintt{u_{tt}^{(n)}}
\def\ppsinttt{u_{ttt}^{(n)}}
\newcommand\ttbdelta{\delta}

\newcommand{\Dal}{{\textup{D}}_t^\alpha}
\newcommand{\Doal}{{\textup{D}}_t^{1-\alpha}}
\newcommand{\bmu}{\boldsymbol{\mu}}
\newcommand{\bxi}{\boldsymbol{\xi}}
\newcommand{\dt}{\, \textup{d} t}
\newcommand{\ds}{\, \textup{d} s }
\newcommand{\dx}{\, \textup{d} x}

\newcommand{\dxs}{\, \textup{d}x\textup{d}s}

\newcommand{\intTO}{\int_0^T \int_{\Omega}}
\newcommand{\inttO}{\int_0^t \int_{\Omega}}
\newcommand{\intt}{\int_0^t}
\newcommand{\intT}{\int_0^T}
\newcommand{\intO}{\int_{\Omega}}
\newcommand{\nLtwo}[1]{\|#1\|_{L^2}}

\newcommand{\nLtwoLtwo}[1]{\|#1\|_{L^2 (L^2)}}

\def\calX{\mathcal{X}}
\def\XWest{\mathcal{X}^{\textup{WB}}}
\def\calXBW{\mathcal{X}^{\textup{WB}}}
\def\calXBK{\mathcal{X}^{\textup{KB}}}
\def\calBBW{\mathcal{B}^{\textup{WB}}}
\def\calBBK{\mathcal{B}^{\textup{KB}}}

\newcommand{\R}{\mathbb{R}} 
\newcommand{\N}{\mathbb{N}} 
\newcommand{\Ltwo}{L^2(\Omega)}
\newcommand{\Hone}{H^1(\Omega)}
\newcommand{\Hneg}{H^{-1}(\Omega)}
\newcommand{\Honezero}{H_0^1(\Omega)}

\newcommand{\Honetwo}{{H_\diamondsuit^2(\Omega)}}
\newcommand{\Honethree}{{H_\diamondsuit^3(\Omega)}}

\newtheorem{theorem}{Theorem}
\newtheorem{lemma}{Lemma}
\newtheorem{proposition}{Proposition}

\newtheorem*{assumption*}{Assumptions}

\newtheorem{remark}{Remark}
\numberwithin{lemma}{section}
\numberwithin{proposition}{section}
\numberwithin{theorem}{section}
\numberwithin{equation}{section}
\newcommand{\bfq}{\boldsymbol{q}}
\newcommand{\bfmu}{\boldsymbol{\mu}}
\newcommand{\bfxi}{\boldsymbol{\xi}}
\newcommand{\calK}{\mathcal{K}}
\newcommand{\frakK}{\mathfrak{K}}
\newcommand{\tfrakK}{\tilde{\mathfrak{K}}}
\newcommand{\frakKone}{\mathfrak{K}_1}
\newcommand{\frakKtwo}{\mathfrak{K}_2}
\newcommand{\tfrakKone}{\tilde{\mathfrak{K}}_1}
\newcommand\Lconv{\ast}

\newcommand{\rt}{\tau_\theta}
\newcommand{\ttb}{\tau_\theta^b}

\definecolor{grey}{rgb}{0.5,0.5,0.5}
 
\usepackage{stackengine,scalerel}

\newcommand\T{\rule{0pt}{4ex}}       
\newcommand\B{\rule[-1.2ex]{0pt}{0pt}} 
\definecolor{darkgreen}{rgb}{0,0.5,0}

\makeatletter
\@namedef{subjclassname@2020}{\textup{2020} Mathematics Subject Classification}
\makeatother
\makeatletter
\newcommand{\leqnomode}{\tagsleft@true}
\newcommand{\reqnomode}{\tagsleft@false}
\makeatother
\makeatother  
\title[Multi-term nonlocal Jordan--Moore--Gibson--Thompson equations]{The vanishing relaxation time behavior of multi-term nonlocal Jordan--Moore--Gibson--Thompson equations}

    
\author[B. Kaltenbacher and V. Nikoli\'{c}]{\small Barbara Kaltenbacher and Vanja Nikoli\'{c}}
\address{  \small
	Department of Mathematics, 
	Alpen-Adria-Universit\"at Klagenfurt 
	\\ Universit\"atsstra\ss e 65--67, A-9020 Klagenfurt, Austria}
\email{barbara.kaltenbacher@aau.at}
\address{ 
	Department of Mathematics \\ 
	Radboud University   \\ 
	Heyendaalseweg 135,
	6525 AJ Nijmegen, The Netherlands}
\email{vanja.nikolic@ru.nl}
\begin{document}
\vspace*{8mm}
\begin{abstract}
	The family of Jordan--Moore--Gibson--Thompson (JMGT) equations arises  in nonlinear acoustics when a relaxed version of the heat flux law is employed within the system of governing equations of sound motion. Motivated by the propagation of sound waves in complex media with anomalous diffusion, we consider here a generalized class of such equations involving two (weakly) singular memory kernels in the principal and non-leading terms. To relate them to the second-order wave equations, we investigate their vanishing relaxation time behavior. The key component of this singular limit analysis are the uniform bounds for the solutions of these nonlinear equations of fractional type with respect to the relaxation time. Their availability turns out to depend not only on the regularity and coercivity properties of the two kernels, but also on their behavior relative to each other and the type of nonlinearity present in the equations.
\end{abstract}

\vspace*{-7mm}
\maketitle           
\section{Introduction} \label{Sec:Introduction}
The family of Jordan--Moore--Gibson--Thompson (JMGT) equations~\cite{jordan2014second} arises in nonlinear acoustics when the classical Fourier heat flux law is replaced by the Maxwell--Cattaneo law~\cite{cattaneo1958forme} within the system of governing equations of sound propagation. The latter introduces thermal relaxation with the parameter $\tau>0$ thereby avoiding the so-called paradox of infinite speed of propagation. The resulting acoustic equations are then third-order in time:
\begin{equation} \label{general_JMGT_equation} \tag{\text{JMGT}}
\begin{aligned}
\begin{multlined}[t]
\tau \ppsittt+\aaa(\ppsi, \ppsit) \ppsitt
-c^2 \bbb(\ppsi, \ppsit) \Delta \ppsi - \tau c^2  \Delta\ppsi
- \ttbdelta  \Delta \ppsit +\calN(\ppsit, \nabla \ppsi, \nabla \ppsit)=0.
\end{multlined}
\end{aligned}
\end{equation}
The function $u$ stands for either the acoustic pressure or acoustic velocity potential. The functions $\aaa$, $\bbb$, and $\calN$ dictate the type of nonlinearity present in the equation; we will discuss them in depth together with modeling in Section~\ref{Sec:Modeling}.  The JMGT equations and their linearizations (known as Moore--Gibson--Thompson (MGT) equation) have been extensively studied in the recent mathematical literature; we refer to, e.g.,~\cite{bucci2020regularity,dell2017moore, chen2019nonexistence, kaltenbacher2011wellposedness, racke2020global} for a selection of the results on their well-posedness, regularity of solutions, and long-term behavior.  Also the stability of (J)MGT equations with additional memory terms has been extensively researched; see, e.g.,~\cite{dell2016moore, dell2020note,  lasiecka2017global, lasiecka2016moore} and the references provided therein. \\
\indent  Recently nonlocal generalizations of these equations of higher order have been put forward in~\cite{kaltenbacher2022time} based on using the Compte--Metzler fractional interpolations~\cite{compte1997generalized} of the Fourier and Maxwell--Cattaneo flux laws valid in media with anomalous diffusion, such as biological tissues. This type of modeling has significantly gained in importance with the rise of ultrasound imaging applications~\cite{szabo2004diagnostic}.  Motivated by such sound propagation, we consider here the following family of nonlocal generalizations of the \ref{general_JMGT_equation}  equation, given by 
\begin{equation} \label{general_wave_equation}
	\begin{aligned}
\begin{multlined}[t]
	\taua \frakKone *\ppsittt+\aaa \ppsitt
-c^2 \bbb\Delta \ppsi - \taua c^2  \frakKone*\Delta\ppsit
- \ttbdelta   \frakKtwo*\Delta \ppsitt  +\calN=0,
\end{multlined}
	\end{aligned}
\end{equation}
where $\Lconv$ denotes the Laplace convolution in time.  
The power $a$ is dependent on the kernel $\frakKone$ and included to ensure dimensional homogeneity. When $\frakKone$ is the Dirac delta distribution $\delta_0$ (with $a=1$) and $\frakKtwo=1$,  \eqref{general_wave_equation} formally reduces to the \ref{general_JMGT_equation} equation, up to modifying the right-hand side. However, the presence of the two kernels allows us to treat a much richer family of equations here than \eqref{general_JMGT_equation}. For example, equation \eqref{general_wave_equation} with suitable Abel kernels covers the time-fractional equations introduced and analyzed in~\cite{kaltenbacher2022time} under the name fractional Jordan--Moore--Gibson--Thompson (fJMGT) equations that correspond to the four fractional flux laws of Compte and Metzler~\cite{compte1997generalized}; we refer to upcoming Section~\ref{Sec:Modeling} for details.  \\
\indent As the relaxation parameter $\tau>0$ is small, it is of high interest to determine the behavior of solutions to \eqref{general_wave_equation} as it vanishes. This is the main goal of the present work. By formally setting the relaxation parameter to zero, one arrives at equations with the leading term of second order:
\begin{equation} \label{general_limiting_wave_equation}
\begin{aligned}
\begin{multlined}[t]
\aaa \ppsitt
-c^2 \bbb \Delta \ppsi 
- \ttbdelta   \frakKtwo*\Delta \ppsitt  +\calN=0.
\end{multlined}
\end{aligned}
\end{equation}
Our singular limit analysis is based on proving well-posedness of \eqref{general_wave_equation}, which we consider with homogeneous Dirichlet data on bounded domains and three initial conditions, uniformly in $\tau$. As one might expect, whether one can obtain the uniform bounds on the solutions is heavily influenced by the properties of the two kernels and their interplay with the nonlinearities present in the equations (that is, the properties of the functions $\aaa$, $\bbb$, and $\calN$). This will necessarily lead to delicate case distinctions. More precisely, we will consider two sets of assumptions on the kernels and then develop the corresponding theories; the details can be found in respective Sections~\ref{Sec:GFEIII} and \ref{Sec:GFEI}. The first set of assumptions will allow for a more standard testing strategy, using $(-\Delta)^\nu \ppsitt$ 
with $\nu\in\{0,1,2\}$ as test functions; the second set of assumptions will involve $\frakKone*(-\Delta)^\nu \ppsitt$ as a test function. \\
\indent The analysis will cover the following fractional JMGT equations introduced in~\cite{kaltenbacher2022time}:
\begin{equation} \label{fJMGTI} \tag{\text{fJMGT I}}
\begin{aligned}
	\tau^\alpha \textup{D}_t^\alpha \ppsitt+\aaa \ppsitt
-c^2 \bbb(\ppsi, \ppsit) \Delta \ppsi - \tau^\alpha c^2  \textup{D}_t^{\alpha} \Delta\ppsi
- \ttbdelta   \textup{D}_t^{1-\alpha}\Delta \ppsit  +\calN=0,
\end{aligned}
\end{equation}
\begin{equation} \label{fJMGTIII} \tag{\text{fJMGT III}}
\begin{aligned}
	\tau \ppsittt+\aaa \ppsitt
-c^2 \bbb \Delta \ppsi - \tau c^2 \Delta\ppsi
- \ttbdelta   \textup{D}_t^{1-\alpha} \Delta \ppsit  +\calN=0,
\end{aligned}
\end{equation} 
\begin{equation} \label{fJMGT} \tag{\text{fJMGT}}
\begin{aligned}
\tau^\alpha \textup{D}_t^\alpha \ppsitt+\aaa \ppsitt
-c^2 \bbb \Delta \ppsi - \tau^\alpha c^2  \textup{D}_t^{\alpha} \Delta\ppsi
- \ttbdelta   \Delta \ppsit  +\calN=0,
\end{aligned}
\end{equation}
 under various (different) restrictions in terms of the order of differentiation $\alpha$ and the involved nonlinearities (where we distinguish the so-called Westervelt--Blackstock and Kuznetsov--Blackstock type). Here numbers I and III indicate that the equations stem from the Compte--Metzler fractional flux laws introduced under the same numbers in~\cite{compte1997generalized}; the last one is unnumbered in \cite{compte1997generalized}. \\
\indent In terms of closely relevant works, we point out the singular limit analysis of third-order equations \eqref{general_JMGT_equation} on bounded domains in \cite{bongarti2020vanishing, kaltenbacher2019jordan, kaltenbacher2020vanishing}. In particular, our analysis follows in the spirit of~\cite{kaltenbacher2019jordan} by employing an energy method on a linearized problem in combination with a fixed-point strategy. We also point out two works which consider \eqref{general_wave_equation} in simplified settings that allow for optimizing the theory. The first one is~\cite{nikolic2023} with a tailored treatment of \eqref{general_wave_equation} in the case $\frakKtwo=1$ (leading to the \ref{fJMGT} equation) which allows for a different testing strategy compared to ours here and less restrictive assumptions on $\frakKone$. The second is \cite{meliani2023} which investigates linear versions of \eqref{general_wave_equation} allowing for a broader family of kernels and the treatment of equations based on the second Compte--Metzler law, among others. 
 \begin{table}[h]
	\captionsetup{width=0.9\linewidth}
	\begin{center} \small
		\begin{tabular}{|m{1.6cm}||m{8.68cm}||m{3.3cm}|} \hline&&\\ 
			Equation          &  $\tau$-uniform well-posedness or existence  & $\tau$ weak limits \\[4mm]
			\Xhline{2\arrayrulewidth}  \hline&&\\
			\ref{general_JMGT_equation}  &  Theorems~\ref{Thm:WellP_GFEIII_BWest} and \ref{Thm:WellP_GFEIII_BKuznetsov}  & Theorems~\ref{Thm:WeakLim_GFEIII_BWest} and~\ref{Thm:WeakLim_GFEIII_BKuznetsov} \\[4mm]
			\Xhline{0.1\arrayrulewidth} &&\\
			\ref{fJMGTI}  & \vspace*{-2mm}\makecell{\hspace*{-2.9cm} with  $\alpha <1/2$, $\ppsitt \vert_{t=0}=0$, Theorem~\ref{Thm:Wellp_testingK1_caseI}  \\[2mm] with  $\alpha \geq 1/2$, $\aaa \equiv 1$, $\ppsitt\vert_{t=0}=0$, Theorem~\ref{Thm:Wellp_testingK1_caseII} (existence, \\ \hspace*{-3.1cm} Westervelt--Blackstock nonlinearities)} & Theorem~\ref{Thm:WeakLim_testingK1} \\[4mm]
			\Xhline{0.1\arrayrulewidth} &&\\
			\ref{fJMGTIII}  & with  $\alpha>1/2$, Theorems~\ref{Thm:WellP_GFEIII_BWest} and \ref{Thm:WellP_GFEIII_BKuznetsov} &   Theorems~\ref{Thm:WeakLim_GFEIII_BWest} and~\ref{Thm:WeakLim_GFEIII_BKuznetsov} \\[4mm]
			\Xhline{0.1\arrayrulewidth} &&\\
			\ref{fJMGT}  & with  $\alpha \geq 1/2$, $\aaa \equiv 1$, $\ppsitt \vert_{t=0}=0$, Theorem~\ref{Thm:Wellp_testingK1_caseII}  (existence, Westervelt--Blackstock nonlinearities) & Theorem~\ref{Thm:WeakLim_testingK1} \\[4mm]
			\hline
		\end{tabular}
	\end{center}
	\caption{\small Main results of this work for the nonlinear JMGT and fJMGT equations}  \label{tab:recap}
\end{table} 

\indent We organize the rest of the exposition as follows. Section~\ref{Sec:Modeling} first gives the necessary background details on the modeling. We then split the analysis into two parts, corresponding to two sets of assumptions on the kernels. Section~\ref{Sec:GFEIII} considers kernels and equations amenable to testing with $(-\Delta)^\nu \ppsitt$ with $\nu \in \{1, 2\}$; this is, for example, the fJMGT III equation. Section~\ref{Sec:GFEI} considers kernels and equations amenable to testing with $\ppsitt$ but also $- \Delta \frakKone* \ppsitt$; this will be the case for the fJMGT I equation, for instance. Before proceeding, we summarize in Table~\ref{tab:recap} the uniform well-posedness and existence results of this work for the fractional JMGT equations as particular cases. 

\section{Acoustic equations based on nonlocal flux laws} \label{Sec:Modeling}
In this section, we discuss acoustic modeling to motivate the wave equations considered in this work in the context of nonlinear sound propagation in tissue-like media. 
\subsection{A generalized nonlocal heat flux law of the Maxwell--Cattaneo type} Classical acoustic equations that describe nonlinear sound motion in thermoviscous fluids are derived as approximations of the Navier--Stokes--Fourier system  based on the following Fourier heat flux law:
\begin{equation} \label{Fourier_flux_law}
\bfq= - \kappa \nabla \theta;
\end{equation}
 see, for example~\cite{blackstock1963approximate, jordan2016survey, crighton1979model}. Here $\bfq$ is the heat flux, $\theta$ the absolute temperature, and $\kappa>0$ denotes the thermal conductivity.\\
 \indent The acoustic equations studied in this work are based on employing a nonlocal generalization of \eqref{Fourier_flux_law} of the Maxwell--Cattaneo type which incorporates thermal relaxation as follows:
\begin{equation} \label{MC_flux_law_nonlocal}
	\bfq(t)+ \taua \int_0^t	\mathfrak{K}_1(t-s)\bfq_t(s)\ds= - \ttb \kappa \int_0^t \mathfrak{K}_2(t-s)\nabla \theta_t(s)\ds.
\end{equation}
Here $\tau>0$ stands for the thermal relaxation time and the constant $\rt$ as well as the powers $a$ and $b$ are there to ensure the dimensional homogeneity of the equation. 
 The introduced memory kernels $\frakKone$ and $\frakKtwo$ are assumed to be independent of $\tau$.  \\
\indent  This relation generalizes many flux laws in the literature. Fourier law \eqref{Fourier_flux_law} follows by setting $\tau=0$, $\frakKtwo=1$, and $\ttb=1$, where we assume that $\bfq(0)=0$ and $\nabla\theta(0)=0$. The well-known Maxwell--Cattaneo law~\cite{cattaneo1958forme}:
\begin{equation} \label{MC_law}
	\begin{aligned}
		\bfq+\tau \bfq_t= - \kappa \nabla \theta
	\end{aligned}
\end{equation}
follows by setting $a=1$, $\frakKone=\delta_0$ (the Dirac delta distribution), $\ttb=1$, and $\frakKtwo=1$. Importantly, \eqref{MC_flux_law_nonlocal} unifies (and generalizes) the Compte--Metzler fractional laws ~\cite{compte1997generalized}:
\begin{alignat}{3}
	\hspace*{-1.5cm}\text{\small(GFE I)}\hphantom{II}&& \qquad \qquad(1+\tau^\alpha \Dal)\boldsymbol{q}(t) =&&\, -\kappa {\rt^{1-\alpha}}\Doal \nabla \theta;\\[1mm]
	\hspace*{-1.5cm}\text{\small(GFE II)}\, \hphantom{I}&&\qquad \qquad (1+\tau^\alpha \Dal)\boldsymbol{q}(t) =&&\, -\kappa {\rt^{\alpha-1}}\Dt^{\alpha-1} \nabla \theta;\\[1mm]
	\hspace*{-1.5cm}\text{\small(GFE III)}\,\, && \qquad \qquad (1+\tau \partial_t)\boldsymbol{q}(t) =&&\, -\kappa {\rt^{1-\alpha}}\Doal \nabla \theta; \\[1mm]
	\hspace*{-1.5cm}\text{\small (GFE)}\hphantom{III}&&\qquad \qquad  (1+\tau^\alpha \Dal)\boldsymbol{q}(t) =&&-\kappa \nabla \theta. \hphantom{{\rt^{1-\alpha}}\Doal }
\end{alignat}
Although in~\cite{compte1997generalized} the fractional derivative is understood in the Riemann--Liouville sense, in this work $\Dt^{\eta}$ denotes the Caputo--Djrbashian fractional derivative:
\[
\Dt^{\eta}w(t)=\frac{1}{\Gamma(1-\eta)}\int_0^t (t-s)^{-\eta}\Dt^{\lceil\eta\rceil} w(s) \, \textup{d} s, \quad -1<\eta <1;
\]
see, for example,~\cite[\S 1]{kubica2020time} and~\cite[\S 2.4.1]{podlubny1998fractional} for the definition.
Here $n=\lceil\eta\rceil$, $n\in\{0,1\}$ is the integer obtained by rounding up $\eta$ and $\Dt^n$ is the zeroth or first derivative operator.  We may do this exchange at this point since it is assumed that $\bfq(0)=0$ and $\nabla\theta(0)=0$.\\
\indent Looking at the Compte--Metzler laws, we can see that GFE II has a different form compared to others since $\textup{D}_t^{\alpha-1}$ is an integral operator rather than a derivative for $\alpha\in(0,1)$. It thus does not fit properly into the framework of \eqref{MC_flux_law_nonlocal} and we do not consider it going forward. We refer to~\cite{meliani2023} for the treatment of acoustic waves based on this law in a linear setting. \\
\indent	Before proceeding we collect in Table~\ref{Table:Kernels_Compte--Metzler_laws} the kernels $\frakKone$ and $\frakKtwo$ for the three Compte--Metzler laws relevant for this work (and the powers $a$ and $b$)  which allow writing them in form of \eqref{MC_flux_law_nonlocal}. Here we use the short-hand notation for the Abel kernel:
\[
g_\alpha(t):= \frac{1}{\Gamma(\alpha)} t^{\alpha-1}.
\]
\begin{table}[h]
	\centering	
	\captionsetup{width=13cm}
	\begin{adjustbox}{max width=\textwidth}
		\begin{tabular}{|c||c|c|c|c|c|c|}
			\hline
			Flux law  & $\frakKone$ & $\frakKtwo$  &$a$  &$b$	\\
			\hline\hline
			GFE I\rule{0pt}{3ex} & $g_{1-\alpha}$ 	&$g_{\alpha}$	&$\alpha$ &$1-\alpha$  \\
			GFE III\rule{0pt}{3ex} &$\delta_0$			&$g_{\alpha}$	&$1$ &$1-\alpha$ 
			\\
			GFE \rule{0pt}{3ex} & $g_{1-\alpha}$	&$1   $		&$\alpha$ & - 	\\
			\hline
		\end{tabular}
	\end{adjustbox}
	~\\[2mm]
	\caption{\small Kernels for the  Compte--Metzler fractional laws} \label{Table:Kernels_Compte--Metzler_laws}
\end{table}
\begin{remark}[On the Gurtin--Pipkin approach]
An alternative approach of generalizing the heat flux laws is given by the Gurtin--Pipkin flux law~\cite{gurtin1968general}:
\begin{equation} \label{GP_flux_law}
\bfq(t)= - \kappa \int_0^t \calK_\tau(t-s)\nabla \theta(s) \ds.
\end{equation}
The resulting acoustic equations are then second order in time with damping of fractional type but the kernel $\calK_\tau$ may depend on $\tau$; we refer to \cite{kaltenbacher2022limiting} for their derivation and limiting analysis. 
\end{remark}
\subsection{Acoustic modeling with the generalized heat flux law} By closely following the steps of the derivation in \cite{kaltenbacher2022time} only now with generalized heat flux law instead of the fractional ones, the following nonlinear acoustic wave equation can be derived:
\begin{equation}\label{general_eq}
	\begin{aligned}
\begin{multlined}[t]\taua \frakKone*\psi_{ttt}+\aaa(\psi_t)\psi_{tt}
	-c^2\bbb(\psi_t) \Delta \psi -\taua c^2 \Delta \frakKone*\psi_t 
	-\ttb \delta \frakKtwo*\Delta \psi_{tt}\\+ \calN(\nabla \psi, \nabla \psi_t)=0,
	\end{multlined}
	\end{aligned}
\end{equation}
where either
\begin{equation} \label{Blackstock_nonlinearity}
	\aaa = 1, \quad \bbb(\psi_t)=1-2\tilde{k}\psi_t, \quad \calN(\nabla \psi, \nabla \psi_t)=\tilde{\ell} \partial_t(|\nabla \psi|)^2=2 \tilde{\ell} \nabla \psi \cdot \nabla \psi_t
\end{equation}
or
\begin{equation} \label{Kuznetsov_nonlinearity}
	\aaa(\psi_t) = 1+2\tilde{k}\psi_t, \quad \bbb=1, \quad \calN(\nabla \psi, \nabla \psi_t)=\tilde{\ell} \partial_t(|\nabla \psi|)^2=2 \tilde{\ell} \nabla \psi \cdot \nabla \psi_t.
\end{equation}
Equation \eqref{general_eq} can be understood as a generalization of the time-fractional models considered in \cite{kaltenbacher2022time}, with kernels $\frakKone$ and $\frakKtwo$ generalizing those in Table~\ref{Table:Kernels_Compte--Metzler_laws}. \\
\indent In the limiting case $\tau=\delta=0$, having nonlinearities \eqref{Blackstock_nonlinearity} corresponds to the classical Blackstock equation~\cite{blackstock1963approximate} in nonlinear acoustics: 
\begin{equation} \label{Blackstock_equation}
	\begin{aligned}
\psi_{tt}
		-c^2(1+2 \tilde{k}\psi_t) \Delta \psi + 2 \tilde{\ell} \nabla \psi \cdot \nabla \psi_t=0,
	\end{aligned}
\end{equation}
and \eqref{Kuznetsov_nonlinearity} to the Kuznetsov equation~\cite{kuznetsov1971equations}:
\begin{equation} \label{Kuznetsov_equation}
	\begin{aligned}
		(1-2 \tilde{k}\psi_t)\psi_{tt}
		-c^2 \Delta \psi + 2 \tilde{\ell} \nabla \psi \cdot \nabla \psi_t=0.
	\end{aligned}
\end{equation}
For the Kuznetsov equation above, it is usual to employ the approximation
\begin{equation}
	|\nabla \psi|^2 \approx c^{-2}\psi_t^2,
\end{equation}
when cumulative nonlinear effects dominate the local ones, and in this manner simplify it by the Westervelt equation~\cite{westervelt1963parametric}. Using this approximation in \eqref{general_eq} results in
\begin{equation} \label{West_type_potential}
	\begin{aligned}
		\begin{multlined}[t]
\taua \frakKone*\psi_{ttt}+
{\aaa(\psi_t)}
\psi_{tt}
-c^2\bbb(\psi_t) \Delta \psi -\taua c^2 \frakKone*\Delta \psi_t
-\ttb \delta \frakKtwo*\Delta \psi_{tt}=0
		\end{multlined}
	\end{aligned}
\end{equation}
with $k=k+c^{-2}\tilde{\ell} $.  \\
\indent It is also common to express the Westervelt equation in terms of the acoustic pressure $p$. Formally taking the time derivative of \eqref{West_type_potential} and employing the pressure-potential relation \[p=\varrho \psi_t\]
as well as the approximation $\Delta \psi\approx c^{-2}\psi_{tt}$ justified by Blackstock's scheme \cite{hamilton1998nonlinear,lighthill1956viscosity} that allows to use a lower order approximation within a higher order term (here it is a first linear approximation used in a quadratic term)
leads to the pressure form 
\begin{equation} \label{JMGT_West_nonlocal_pressure}
	\begin{aligned}
		\begin{multlined}[t]
			\taua \frakKone*(p_{tt}-c^2 \Delta p)_t+
{\aaa(\tfrac{p}{\rho})}
p_{tt}
			- c^2
{\bbb(\tfrac{p}{\rho})}
\Delta p	-	\ttb \delta  \frakKtwo*\Delta p_{tt} 
{+\tfrac{1}{\varrho}(\aaa'(\tfrac{p}{\varrho})-\bbb'(\tfrac{p}{\varrho}))}
p_t^2\\= {f}
		\end{multlined}
	\end{aligned}
\end{equation}
with the right-hand side
\begin{equation} \label{source_Westervelt_type}
	{f}(t) = -\taua \frakKone(t) p_{tt}(0)+\taua c^2 \frakKone(t) \Delta p(0) + \ttb\delta \frakKtwo(t)\Delta p_{t}(0).
\end{equation}
To be able to treat all these different acoustic equation, we unify them in one model. \\[1mm]
\noindent \emph{Equations considered in this work.} We assume throughout that $\Omega \subset \R^d$ with $d \in \{1, 2, 3\}$ is a bounded and sufficiently smooth domain. $T>0$ denotes the final propagation time. Motivated by the modeling discussed above, we study the following general equation:
\begin{equation}\label{general_eq_ppsi}
	\begin{aligned}
		\taua \frakKone*\ppsittt+\aaa \psitt
		-c^2\bbb \Delta \ppsi -\taua c^2 \frakKone* \Delta \ppsit
		-\ttbdelta \frakKtwo*\Delta \ppsitt+ \calN =f,
	\end{aligned}
\end{equation}
 coupled with initial data
\begin{equation} \label{IC}
	(\ppsi, \ppsit, \ppsitt)\vert_{t=0}=(\ppsi_0, \ppsi_1, \ppsi_2),
\end{equation}
and homogeneous Dirichlet boundary conditions
\begin{equation} \label{BC}
	\ppsi \vert_{\partial \Omega} =0.
\end{equation}
We distinguish the following two cases that require different regularity assumptions on the initial data:
\begin{itemize}
	\item Equations of Westervelt--Blackstock-type with
	\begin{equation}
			\begin{aligned}
	\aaa=\aaa(\ppsi)=1+2k_1 \ppsi, \quad \bbb=1-2k_2 \ppsi, \quad
	 \calN=\calN(\ppsit) = 2k_3 \ppsit^2; 
	\end{aligned}
	\end{equation}
\item Equations of Kuznetsov--Blackstock type with
	\begin{equation}
		\begin{aligned}
	\aaa=\aaa(\ppsit)=1+2k_1 \ppsit, \ \bbb=\bbb(\ppsit)=1-2k_2 \ppsit, \
	 \calN=\calN(\nabla \ppsi, \nabla \ppsit) = 2k_3 \nabla \ppsi \cdot \nabla \ppsit,
	\end{aligned}
\end{equation}
\end{itemize}
where we assume $k_{1,2,3} \in \R$. Note that to relax the notation we have merged the constant $\ttb \delta$ into the coefficient $\ttbdelta>0$ in \eqref{general_eq_ppsi}, which therefore no longer has the usual dimension of sound diffusivity. To be able to take limits as $\tau \searrow 0$ we will rely on the damping term containing $\frakKtwo$ and therefore assume $\ttbdelta>0$ to be fixed. As we are interested in the vanishing behavior $\tau \searrow 0$, we assume throughout that $\tau \in (0, \bar{\tau}]$ for some fixed $\bar{\tau}>0$. \\
\indent The Westervelt--Blackstock-type equation incorporates the nonlinearities that arise in pressure form \eqref{JMGT_West_nonlocal_pressure}. We point out that this pressure form is more general than the equation considered in \cite{kaltenbacher2022time, nikolic2023} where $\bbb \equiv 1$ and which was termed of Westervelt-type. Here we also allow for $\bbb=\bbb(u)$ and thus call it of Westervelt--Blackstock type. \\
\indent The Kuznetsov--Blackstock-type equation encompasses \eqref{Blackstock_nonlinearity}, \eqref{Kuznetsov_nonlinearity}  and \eqref{West_type_potential}. It  involves $\ut$ instead of $u$ in the nonlinearities tied to the leading and second terms in the limiting equation ($\tau=0$) and also contains a quadratic gradient term. For this reason, it requires stronger regularity assumptions on the data in the well-posedness analysis compared to the Westervelt--Blackstock case.

\subsection{General strategy in the singular limit analysis} 
Our strategy in the singular limit analysis is based on first deriving $\tau$-uniform bounds for a linearization of \eqref{general_eq_ppsi} with a source term, given by
\begin{equation} \label{JMGT_Kuzn_nonlocal_lin}
\begin{aligned}
\begin{multlined}[t]
\taua \frakKone\Lconv \ppsi_{ttt} +\aaa(x,t)\ppsi_{tt}
-c^2\bbb(x,t)\Delta \ppsi 
- \taua c^2 \frakKone\Lconv \Delta \ppsit	-\ttbdelta  \frakKtwo\Lconv \Delta \ppsi_{tt}=f(x,t),
\end{multlined}
\end{aligned}
\end{equation}
 and later using a suitable fixed-point theorem on the mapping
 \[
 \calT : u^* \mapsto u
 \] to transfer the result to the nonlinear problem. To this end,  the two types of nonlinearities will necessitate different assumptions on the smoothness of data and coefficients $\aaa$ and $\bbb$ (and in some cases prevent certain arguments):
 \begin{itemize}
 	\item Westervelt--Blackstock type with
 	\begin{equation}
 		\aaa=1+2k_1 \ppsi^*, \quad \bbb=1-2k_2 \ppsi^*, \qquad f=-2k_3 (\ppsit^*)^2; 
 	\end{equation}
 	\item Kuznetsov--Blackstock-type with
 	\begin{equation}
 		\aaa=1+2k_1 \ppsit^*, \quad \bbb=1-2k_2 \ppsit^*, \quad f =- 2k_3 \nabla \ppsi^* \cdot \nabla \ppsit^*,
 	\end{equation}
\end{itemize}
The fixed-point of the mapping $\ppsi^*=\ppsi$ will give us the solution of the nonlinear problem. The key component of this analysis are the uniform bounds for the solutions of \eqref{JMGT_Kuzn_nonlocal_lin} with respect to $\tau$.  To obtain them, we will  employ the following two testing strategies
\begin{itemize}
	\item Section~\ref{Sec:GFEIII}: Testing with $-\Delta \ppsitt$ and $\Delta^2 \ppsitt$;
	\item Section~\ref{Sec:GFEI}: Testing with $\ppsitt$ and $- \Delta \frakKone* \ppsitt$,
\end{itemize}
which can be rigorously justified through a Faedo--Galerkin procedure; cf.\ Appendix~\ref{Appendix:Galerkin}.  To make the testing procedure work we will introduce two sets of regularity and coercivity assumptions on the kernels in the corresponding sections.  As a consequence, among the equations obtained from the Compte--Metzler laws, the first testing strategy will work for the \ref{fJMGTIII} equation. The second testing strategy will turn out to work for proving uniform solvability of the \ref{fJMGT} and \ref{fJMGTI} equations but if $\alpha>1/2$, we will have the restriction of Westervelt--Blackstock type on the allowed nonlinearities. 

\subsection{Notation} 
	Below we often use the notation $x\lesssim y$ for $x\leq C\, y$ with a constant $C>0$ that does not depend on the thermal relaxation time $\tau$.  We use $\lesssim_T$ to emphasize that $C=C(T)$ tends to $\infty$ as $T \rightarrow \infty$. \\
\indent We use $(\cdot, \cdot)_{L^2}$ to denote the scalar product on $L^2(\Omega)$ and $\sim$ to denote equivalence of norms. 
 We often omit the spatial and temporal domain when writing norms; for example, $\|\cdot\|_{L^p(L^q)}$ denotes the norm on the Bochner space $L^p(0,T; L^q(\Omega))$. We use $\|\cdot\|_{L^p_t(L^q)}$ to denote the  norm on $L^p(0,t; L^q(\Omega))$ for $t \in (0,T)$.
\subsection{Regularity of the kernels} 
As already mentioned above, several coercivity assumptions will have to be made on the kernels below.
However, concerning their regularity, throughout this paper we will only assume that
\begin{equation}\label{reg_kernels}
\frakKone, \, \frakKtwo \, \in \{\delta_0\} \cup L^1(0,T).
\end{equation}
In fact, the analysis below would also apply to measures $\frakKone\Lconv, \, \frakKtwo\Lconv \, \in \mathcal{M}(0,t)=C_b([0,T])^*$, however, at the cost of somewhat increased technicality. We will be reminded of this possibility when using the following norm below: 
\[
\|\frakKone\|_{\mathcal{M}(0,T)}=\begin{cases}
1 &\text{ if }\frakKone=\delta_0,\\
\|\frakKone\|_{L^1(0,T)}&\text{ if }\frakKone\in L^1(0,T).
\end{cases}
\]

\subsection{Auxiliary theoretical results} We recall from \cite{meliani2023} a result that will allow us to extract appropriately converging subsequences later on in the existence proofs; see, for example, Proposition~\ref{Prop:Wellp_GFEIII_BWest_lin}.   
\begin{lemma}[Sequential compactness with the Caputo--Djrbashian derivative, see~ \cite{meliani2023}] \label{Lemma:Caputo_seq_compact}
	Let $1\leq p\leq \infty$, and let $\frakK \in L^1(0,T)$ be such that there exists $\tfrakK \in L^{p'}(0,T)$ for which $\tfrakK \Lconv \frakK=1$ with $p' =  \frac{p}{p-1}$. Consider the space
	\begin{equation} \label{def_Xfp}
		X_\frakK^p = \{u \in L^p(0,T) \ |\ \frakK \Lconv u_t \in L^p(0,T)\},
	\end{equation}
	with the norm
	\[ \|\cdot\|_{X_\frakK^p} = \big(\|u\|_{L^p}^p + \|(\frakK\Lconv u_t)\|_{L^p}^p \big)^{1/p},\] and the usual modification for $p=\infty$. The following statements hold true:
	\begin{itemize}
		\item 	The space $X_\frakK^p$ is reflexive for $1<p<\infty$ and separable for $1 \leq p<\infty$.
		\item The unit ball $B_X^p$ of $X_\frakK^p$  is weakly sequentially compact for $1<p<\infty$, and $B_X^\infty$ is weak-$*$ sequentially compact. 
		\item The space $X_\frakK^p$ continuously embeds into $C[0,T]$.
	\end{itemize}
\end{lemma}
We also state and prove a result which will be helpful in verifying a coercivity assumption on the two kernels in Section~\ref{Sec:GFEIII}.
\begin{lemma} \label{Lemma_Hneg_alphahalf}
	Let  $y\in W^{1,1}(0,t)$ with $y'\in H^{-\alpha/2}(0,t)$ for $\alpha \in (0,1)$. Then
	\[
	\|y-y(0)\|_{L^\infty(0,t)}\lesssim \|y-y(0)\|_{H^{1-\alpha/2}(0,t)}\sim 
	\|I_t^{\alpha/2}y'\|_{L^2(0,t)}
	\sim\|y'\|_{H^{-\alpha/2}(0,t)}.
	\]
\end{lemma}
\begin{proof} 
	We introduce the time-flip operator as \[\overline{w}^{t}(s)=w(t-s).\]
	By \cite[Theorem 2.2, Corollary 2.1]{kubica2020time}, 
	we have
	\[\|\overline{w}^{t}\|_{H^{\alpha/2}(0,t)}\sim \|(\textup{I}_t^{\alpha/2})^{-1}\overline{w}^{t}\|_{L^2(0,t)},\]
	where
	\[
	\textup{I}^\eta_t y(t)= \frac{1}{\Gamma(\eta)} \int_0^t (t-s)^{\eta-1} y(s) \ds, \ \eta>0.
	\] 
	Additionally employing the identity $\langle \overline{a}^{t}, b\rangle_{L^2(0,t)} = (a*b)(t)= (b*a)(t)$,
	leads to 
	\begin{equation}
	\begin{aligned}
	\|y'\|_{H^{-\alpha/2}(0,t)}
	=&\, \sup_{w\in C_0^\infty(0,t)}\frac{|\langle y',w\rangle_{L^2(0,t)}|}{\|w\|_{H^{\alpha/2}(0,t)}}
	=\sup_{w\in C_0^\infty(0,t)}\frac{|\langle \overline{y'}^t,\overline{w}^t\rangle_{L^2(0,t)}|}{\|\overline{w}^t\|_{H^{\alpha/2}(0,t)}}\\
	\sim&\,\sup_{\tilde{w}\in C_0^\infty(0,t)}\frac{|\langle \overline{y'}^t, \textup{I}_t^{\alpha/2}\tilde{w}\rangle_{L^2(0,t)}|}{\|\tilde{w}\|_{L^2(0,t)}}
	=\sup_{\tilde{w}\in C_0^\infty(0,t)}\frac{|(y'*g_{\alpha/2}*\tilde{w})(t)|}{\|\tilde{w}\|_{H^{\alpha/2}(0,t)}}\\
	=&\,\sup_{\tilde{w}\in C_0^\infty(0,t)}\frac{|\langle \textup{I}_t^{\alpha/2} y',\overline{\tilde{w}}^t\rangle_{L^2(0,t)}}{\|\tilde{w}\|_{L^2(0,t)}}
	=\|I^{\alpha/2} y'\|_{L^2(0,t)},
	\end{aligned}
	\end{equation}
	where we have set $\tilde{w}=(\textup{I}_t^{\alpha/2})^{-1}\overline{w}^{t}$ and used $\|\overline{\tilde{w}}^t\|_{L^2(0,t)}=\|\tilde{w}\|_{L^2(0,t)}$.
	Now due to Sobolev's embedding and $1-\frac{\alpha}{2}>\frac12$, as well as 
	\cite[Theorems 2.4, 2.5, p.\ 22]{kubica2020time}, we have
	\[
	\|y-y(0)\|_{L^\infty(0,t)}\lesssim \|y-y(0)\|_{H^{1-\alpha/2}(0,t)}\sim 
	\|I_t^{\alpha/2}y'\|_{L^2(0,t)}
	\sim\|y'\|_{H^{-\alpha/2}(0,t)},
	\]
	as claimed.
\end{proof}

\section{Testing with canonical test functions} \label{Sec:GFEIII} 
In this section, we perform the analysis of equations \eqref{general_eq_ppsi} amenable to testing with $(-\Delta)^\nu \ppsi_{tt}$ where $\nu \geq 0$. These are the canonical test functions for the third-order  (Jordan--)Moore--Gibson--Thompson  equations  and thus make an obvious choice for an attempt at a uniform analysis here.  Unsurprisingly, among the equations based on the Compte--Metzler laws, this testing will suffice for the acoustic equation \eqref{fJMGTIII} with the integer-order leading term, obtained using the third Compte--Metzler law. \\[2mm]
\noindent \emph{Assumptions on the two kernels in this section.} 
Recall that throughout the paper, we assume the regularity of the two kernels given in \eqref{reg_kernels}. In this section we make the following additional assumption on the resolvent of the leading kernel:
\begin{equation}\label{Athree_GFEIII} 
	\text{there exists} \  \tfrakKone \in L^2(0,T), \text{ such that }\frakKone * \tfrakKone =1.   \tag{$\mathcal{A}_1$}
\end{equation}
In case of the Abel kernel, this is equivalent to asking that the fractional order of differentiation $\alpha$ is larger than $ 1/2$. We next need the coercivity assumptions; one on the leading kernel and one on the two kernels combined. We assume that there exist constants $C$, $\underline{c}$, $\overline{C}>0$, independent of $\tau$, such that 
\begin{equation}\label{Afour_GFEIII} 
	\int_0^{t'} \left(\frakKone* y' \right)(t) \,y(t)\dt \geq - C|y(0)|^2, \quad y\in X^2_{\frakKone}(0,t'),   \tag{$\mathcal{A}_2$}
\end{equation}
and for all $\tau \geq 0$
\begin{equation}\label{Afive_GFEIII}    \tag{$\mathcal{A}_3$}
	\begin{aligned}	 
		&\int_0^{t'} \left(\taua c^2  \frakKone*y +\ttbdelta \frakKtwo*y'\right)(t) y'(t) \dt  
			\geq& \underline{c}\|y\|_{L^\infty(0,t')}^2-\overline{C}|y(0)|^2  
		, \ y\in W^{1,1}(0,t');   
	\end{aligned}
\end{equation}	
recall that the space $X^2_{\frakKone}(0,t')$ is defined in \eqref{def_Xfp} with $p=2$. These assumptions are suited to the analysis of the wave equation \eqref{general_wave_equation} that relies on an energy method based on testing with $(-\Delta)^\nu \utt$ for $\nu \geq 0$. In this case, the leading term $\taua\frakKone * \uttt$ will invoke assumption  \eqref{Afour_GFEIII}, whereas the combination of the other two nonlocal terms, $- \taua c^2 \frakKone* \Delta \ut - \delta \frakKtwo*\Delta \utt$ will invoke \eqref{Afive_GFEIII}. Note that we assume \eqref{Afive_GFEIII} to hold also for $\tau=0$, thereby having the coercivity of the kernel $\frakKtwo$ alone as well. 
\subsection*{How to verify the coercivity assumptions} 
Assumption \eqref{Afour_GFEIII} clearly holds for $\frakKone=\delta_0$. For kernels in $L^p(0,T)$, sufficient conditions under which \eqref{Afour_GFEIII} holds for $y \in W^{1,1}(0,t')$ can be found, for example, in Lemma B.1 in~\cite{kaltenbacher2021determining}; see also~\cite[Lemma 3.1]{oparnica2020well}.  These are as follows:
\begin{equation} \label{assumptions_LemmaB1}
\begin{aligned}
&\frakKone \in L^p(0,T) \text{ for some } p>1, \quad  (\forall \varepsilon>0)\quad  \frakKone \in W^{1,1}(\varepsilon, T),\\
&\frakKone \geq 0 \ \text{ a.e., } \quad \frakKone'\vert_{[\varepsilon, T]} \leq 0 \ \text{ a.e}.
\end{aligned}
\end{equation}
As observed in~\cite{kaltenbacher2022limiting}, the condition on the sign of the kernel can be relaxed to $\frakKone\geq 0$ in a neighborhood of $0$. By a density argument, \eqref{Afour_GFEIII} then also holds for any $y \in X^2_{\frakKone}(0,t')$ under assumptions \eqref{assumptions_LemmaB1}. \\
\indent Assumption \eqref{Afive_GFEIII} is fulfilled for $\frakKone=\delta_0$ and $\frakKtwo=1$; that is, for the third-order \ref{general_JMGT_equation} equation. 
\begin{center}
	\begin{table}[h]
		\centering
		\begin{tabular}[h]{|l||l|}
			\hline
			&\\[-2ex]
			Flux law &$(\Re\mathcal{F}\overline{(\tau^a c^2 1*\mathfrak{K}_1+ \ttbdelta \mathfrak{K}_2)}^\infty)(\imath\omega)$
			\\
			\hline\hline   
			GFE I&$\tau^a c^2 \cos((2-\alpha)\pi/2)\omega^{\alpha-2}+\ttbdelta \cos(\alpha\pi/2)\omega^{-\alpha}$\\[2mm]
			GFE III& $\ttbdelta \cos(\alpha\pi/2)\omega^{-\alpha}$\\[2mm]
			GFE& $\tau^a c^2 \cos((2-\alpha)\pi/2)\omega^{\alpha-2}$\\
			\hline
		\end{tabular}
		\vspace*{4mm}
		\caption{\small Real part of Fourier transforms of the Compte--Metzler laws} \label{Table:Fourier_transf_combined_kernel}
	\end{table}	
\end{center}
\vspace*{-4mm}
For other kernels, one can verify \eqref{Afive_GFEIII} by applying the Fourier analysis technique from \cite[Lemma 2.3]{Eggermont1987} to the combined kernel \[\tilde{\mathcal{K}}:=\taua c^2 1*\frakKone + \ttbdelta\frakKtwo.\] 
see Table~\ref{Table:Fourier_transf_combined_kernel}, where $\calF$ denotes the Fourier transform, $\Re$ the real part, and $\overline{f}^t$ the extension of a function $f$ by zero outside $\mathbb{R}\setminus(0,t)$. Among the fractional Compte--Metzler laws, the combined uniform lower bound \eqref{Afive_GFEIII} only holds for the kernels in the GFE III law. Indeed, in case of the equation obtained using the third Compte--Metzler law:
\begin{equation} 
\begin{aligned}
\begin{multlined}[t]\tau  \ppsi_{ttt}+\aaa \ppsi_{tt}
-c^2\bbb\Delta \ppsi
-   \tau c^2  \Delta \ppsi_t -\ttbdelta \frakKtwo* \Delta \ppsi_{tt}+\calN=0 \end{multlined}
\end{aligned}
\end{equation}
we have $\frakKone=\delta_0$ and $a=1$. Then for $\delta>0$, condition \eqref{Afive_GFEIII} holds (uniformly with respect to $\alpha\in[0,1]$) due to the estimate 
	\begin{equation} \label{coercivity_GFEIII}
		\begin{aligned}
	\int_0^{t'} (\taua c^2 y +\ttbdelta g_\alpha\Lconv y')(t) y'(t)\dt
	 \geq&\, 
	\frac12 \taua c^2 (y(t')^2-y(0)^2)+\ttbdelta\cos(\alpha\pi/2) \|y'\|_{H^{-\alpha/2}(0,t')}^2\\
	\geq&\, \ttbdelta\cos(\alpha\pi/2) \|y'\|_{H^{-\alpha/2}(0,t')}^2
	\end{aligned}
	\end{equation}
and Lemma~\ref{Lemma_Hneg_alphahalf}, which provides a lower bound for $\|y'\|_{H^{-\alpha/2}(0,t')}^2$.
\subsection{Well-posedness of a linearized Westervelt--Blackstock problem} \label{Sec:Lower_order_West_type}
Under assumptions \eqref{reg_kernels} and \eqref{Athree_GFEIII}--\eqref{Afive_GFEIII} on the two kernels, we next discuss the well-posedness of \eqref{general_eq_ppsi} with Westervelt--Blackstock nonlinearities, uniformly in $\tau$. Recall that this means we consider equation 
\begin{equation} \tag{\ref{general_eq_ppsi}}
\begin{aligned}
\begin{multlined}[t]
\taua \frakKone\Lconv \ppsittt +\aaa\ppsitt
-c^2\bbb\Delta \ppsi 
- \taua c^2 \frakKone\Lconv \Delta \ppsit	-	\ttbdelta  \frakKtwo\Lconv \Delta \ppsitt=f,
\end{multlined}
\end{aligned}
\end{equation}
where we have in mind 
\begin{itemize}
	\item $\aaa=1+2{k}_1\ppsi$, \ $\bbb=1-2k_2\ppsi$, and $f= -k_3 \ppsi_t^2$
\end{itemize}
and that, among the discussed equations, the assumptions on the kernels in this section are verified by the third-order \ref{general_JMGT_equation} and \ref{fJMGTIII} equations. \\
\indent As announced, we first analyze a linearization 
where $\aaa=\aaa(x,t) $ and $\bbb=\bbb(x,t)$. We assume that these coefficients are sufficiently regular in the following sense:
\begin{equation} \label{reg_assumptions_aaa_bbb_lowersetting}
\begin{aligned}
\aaa \in&\, \left \{\aaa \in C([0,T]; L^\infty(\Om)): \nabla \aaa \in L^\infty(0,T;L^4(\Omega))\right \}, \\
\bbb \in&\,  W^{1,1}(0,T; L^\infty(\Omega)) \hookrightarrow C([0,T]; L^\infty(\Om)).
\end{aligned}
\end{equation}
Furthermore, the coefficient $\aaa$ should not degenerate: we assume that there exist $\ulaaa$, $\olaaa>0$, independent of $\tau$, such that
\begin{equation} \label{nondeg_aaa_GFEIII}
\ulaaa < \aaa(x,t) < \olaaa \quad \text{a.e.\ in } \Omega \times (0,T).
\end{equation}
To state the well-posedness result, we introduce the following space:
\begin{equation}
\begin{aligned}
\calXBW= \{\ppsi \in L^\infty(0,T; \Honetwo):&\, \ppsit \in L^\infty(0,T; \Honetwo), \\ &\, \ppsitt \in L^2(0,T; H_0^1(\Om)) \}.
\end{aligned}
\end{equation}
Note that if $u \in \calXBW$, we have the following (weak) continuity in time
\begin{equation}
\begin{aligned}
\ppsi \in \XWest \implies \ppsi \in C([0,T]; \Honetwo), \ \ppsit \in C_{w}(0,T; \Honetwo);
\end{aligned}
\end{equation}
see~\cite[Ch.\ 2, Lemma 3.3]{temam2012infinite}.
\begin{proposition}\label{Prop:Wellp_GFEIII_BWest_lin}
	Let $T>0$ and $\tau \in (0, \bar{\tau}]$. Let assumptions \eqref{reg_kernels}, \eqref{Athree_GFEIII}--\eqref{Afive_GFEIII} on the kernels hold. Let the coefficients $\aaa$ and $\bbb$ satisfy regularity assumption \eqref{reg_assumptions_aaa_bbb_lowersetting}. Assume that $\aaa$ does not degenerate so that \eqref{nondeg_aaa_GFEIII} holds. Let $f \in W^{1,1}(0,T; L^2(\Omega))$ and
	\begin{equation}
	\begin{aligned}
	(\ppsi, \ppsit, \ppsitt)\vert_{t=0} =(\ppsi_0, \ppsi_1, \ppsi_2) \in \Honetwo \times \Honetwo \times \Honezero.
	\end{aligned}
	\end{equation}
	Then there exists $m>0$, independent of $\tau$, such that  if
	\begin{equation}
	\begin{aligned}
	\|\nabla \aaa\|_{L^\infty(L^4)} \leq m,
	\end{aligned}
	\end{equation}
	then there is a solution 
	\[ \ppsi \in \XWest, \quad \taua \frakKone* \ppsittt \in L^2(0,T; \Hneg) \] of the initial boundary-value problem 
	\begin{equation} \label{IBVP_Westtype_lin}
	\left \{	\begin{aligned}
	&\begin{multlined}[t]	\taua  \frakKone * \ppsittt+\aaa(x,t) \ppsitt
	-c^2 \bbb(x,t) \Delta \ppsi
	-   \taua c^2 \frakKone* \Delta \ppsit	\\ \hspace*{6cm}-\ttbdelta \frakKtwo* \Delta \ppsitt=f(x,t)\ \textup{ in } \Omega \times (0,T), \end{multlined}\\
	&\ppsi\vert_{\partial \Omega}=0, \\
	&	(\ppsi, \ppsit, \ppsitt)\vert_{t=0} =(\ppsi_0, \ppsi_1, \ppsi_2).
	\end{aligned} \right. 
	\end{equation}
	The solution satisfies the following estimate: 
	\begin{equation}
	\begin{aligned}
	\|\ppsi\|_{\XWest}^2 \lesssim_T \|\ppsi_0\|^2_{H^2}+\|\ppsi_1\|^2_{H^2} + \taua \|\ppsi_2\|^2_{H^1}+ \|f\|_{W^{1,1}(L^2)}^2,
	\end{aligned}
	\end{equation}
	where the hidden constant does not depend on $\tau$. If additionally $\bbb \in W^{1,1}(0,T; W^{1,4}(\Omega))$, the solution is unique.
\end{proposition}
\begin{proof}
	We conduct the proof by using a standard Faedo--Galerkin procedure where we first construct an approximate solution $\ppsin \in W^{2, \infty}(0,T; V_n)$ with $V_n$ being a finite-dimensional subspace of $\Honetwo$; the details can be found in Appendix~\ref{Appendix:Galerkin}. \\
	\indent The next step in the proof is to obtain a bound that is uniform in $n$ (and also $\tau$, having in mind the later singular limit analysis). By testing the semi-discrete problem with $-\D \ppsintt$ and using the assumptions we have made on the kernels, we find that
	\begin{equation}
	\begin{aligned}
	&		\int_0^t \nLtwo{\sqrt{\aaa}\nabla \ppsintt}^2 + \underline{c} \|\Delta \ppsint\|^2_{L^\infty(L^2)}\\
	\lesssim &\,  \begin{multlined}[t]
	\int_0^t\left\{
	-(\ppsintt \nabla \aaa , \nabla \ppsintt )_{L^2}-	c^2(\bbb\D \ppsi, \D \ppsintt )_{L^2}
	+({f},-\D \ppsintt )_{L^2}
	\right\}\ds+ \taua\|\nabla \ppsin_2\|^2_{L^2}\\ +\|\Delta \ppsin_1 \|^2_{L^2}. \end{multlined}
	\end{aligned}
	\end{equation}
	Integrating by parts in the right-hand side terms yields
	\begin{equation}
	\begin{aligned}
	&		\int_0^t \nLtwo{\sqrt{\aaa}\nabla \ppsintt}^2 + \underline{c} \|\Delta \ppsint\|^2_{L^\infty(L^2)}\\
	\lesssim 	&\,  \begin{multlined}[t] - \intt (\ppsintt \nabla \aaa , \nabla \ppsintt)_{L^2} \ds+c^2\left(-	\bbb \D \ppsin(s), \D \ppsint (s)\right)_{L^2}\Big \vert_0^t +({f}(s),-\D \ppsint(s))_{L^2} \Big \vert_0^t \\	+\int_0^t\left\{
	c^2(\bbb \D \ppsint+\bbb_t \D \ppsin, \D \ppsint)_{L^2}
	+({f}_t,\D \ppsint)_{L^2}
	\right\}\ds\\+ \taua \|\nabla \ppsin_2\|^2_{L^2}+\|\Delta \ppsin_{1}\|^2_{L^2}. \end{multlined}
	\end{aligned}
	\end{equation}
	From here using H\"older's inequality and the embedding $H^1(\Omega) \hookrightarrow L^4(\Omega)$, we have
	\begin{equation} \label{lowerorder_est_1}
	\begin{aligned}
	&
	\ulaaa\int_0^t \nLtwo{\nabla \ppsintt}^2 + \underline{c} \|\Delta \ppsint\|^2_{L^\infty(L^2)}\\
	\lesssim	&\,  \begin{multlined}[t]
	\|\nabla \aaa\|_{L^\infty(L^4)}\nLtwoLtwo{\nabla \ppsintt}^2 + \olbbb^2\|\Delta \ppsin\|^2_{L^\infty(L^2)}+\varepsilon\|\Delta \ppsint\|^2_{L^\infty(L^2)}\\+\|\bbb(0)\|_{L^\infty}^2\|\Delta \ppsin_0\|_{L^2}^2+\|\Delta \ppsin_{1}\|^2_{L^2}+\|f\|^2_{W^{1,1}(L^2)}\\
	+\olbbb \|\D \ppsint\|^2_{L^2(L^2)}+\|\bbb_t\|_{L^1(L^\infty)}^2\|\Delta \ppsin\|^2_{L^\infty(L^2)}+ \taua \|\nabla \ppsin_2\|^2_{L^2} . \end{multlined}
	\end{aligned}
	\end{equation}
	We can further bound $\|\Delta \ppsin\|_{L^\infty(L^2)}$ by using the estimate
	\[
	\|\Delta \ppsin\|_{L^\infty(L^2)} \leq \sqrt{T} \|\Delta \ppsint\|_{L^2(L^2)} +\|\Delta \ppsin_0\|_{L^2}.
	\]
	If $C$ is a hidden constant in \eqref{lowerorder_est_1}, then provided we take $\aaa$ small enough so that
	\begin{equation}
	\begin{aligned}
	\|\nabla \aaa\|_{L^\infty(L^4)} \leq \frac12 C \ulaaa,
	\end{aligned}
	\end{equation}
	and also choose $\varepsilon>0$ small enough, we have via Gronwall's inequality
	\begin{equation} \label{lowerorder_est_2}
	\begin{aligned}
	&
	\int_0^t \nLtwo{\nabla \ppsintt }^2 +  \|\Delta \ppsint\|^2_{L^\infty(L^2)}+ \|\Delta \ppsin\|^2_{L^\infty(L^2)}\\
	\lesssim_T	&\,  \begin{multlined}[t]
	\|\Delta \ppsin_0\|_{L^2}^2+ \|\Delta \ppsin_{1}\|^2_{L^2}+ \taua  \|\nabla \ppsin_2\|^2_{L^2}+\|f\|^2_{W^{1,1}(L^2)}, \end{multlined}
	\end{aligned}
	\end{equation}
	where the hidden constant has the form
	\begin{equation} \label{const_Gronwall}
	C= C_1 \exp \left(C_2T(1+\|\bbb_t\|^2_{L^1(L^\infty)})\right).
	\end{equation}
	Additionally, from the PDE and Young's convolution inequality, we have
	\begin{equation}\label{bootstrap_Hneg}
	\begin{aligned}
	&\taua\| \frakKone * \ppsinttt\|_{L^2(H^{-1})} \\
	\lesssim&\, \begin{multlined}[t]\olaaa \|\ppsintt\|_{L^2(L^2)}+ \olbbb \|\Delta \ppsin \|_{L^2(L^2)}+\taua \|\frakKone\|_{{\mathcal{M}}}\|\D \ppsint\|_{L^2(L^2)}\\+\|\frakKtwo\|_{{\mathcal{M}}}\|\nabla \ppsintt\|_{L^2(L^2)} + \|f\|_{L^2(H^{-1})},
	\end{multlined}
	\end{aligned}
	\end{equation}
	which, taken together with \eqref{lowerorder_est_2}, provides us with a uniform bound on $\taua\| \frakKone * \ppsinttt\|_{L^2(H^{-1})}$ as well. {Since we have assumed that $\tfrakKone \in L^2(0,T)$}, then from the  bound on 
	\begin{equation}
	\begin{aligned}
	\taua \frakKone* \ppsinttt := \tilde{f}^{(n)}
	\end{aligned}
	\end{equation}
	in $L^2(0,T; H^{-1}(\Omega))$, we also obtain a uniform bound on
	\begin{equation} \label{Linf_bound_untt}
	\begin{aligned}
	\taua\|\ppsintt\|_{L^\infty(H^{-1})}=&\, \taua\|\ppsin_2 + \tfrakKone*\tilde{f}^{(n)}\|_{L^\infty(H^{-1})} \\
	\lesssim&\,  \|\ppsin_2\|_{H^{-1}} + \|\tfrakKone\|_{L^2}\|\tilde{f}^{(n)}\|_{L^2(H^{-1})}.
	\end{aligned}
	\end{equation}
	~\\
From the above analysis, we conclude that there is a subsequence (not relabeled), such that 
	\begin{equation} \label{weak_limits_1}
	\begin{alignedat}{4} 
	\ppsin &\relbar\joinrel\rightharpoonup \ppsi&& \quad \text{ weakly-$\star$}  &&\text{ in } &&L^\infty(0,T; \Honetwo),  \\
	\ppsint &\relbar\joinrel\rightharpoonup \ppsit && \quad \text{ weakly-$\star$}  &&\text{ in } &&L^\infty(0,T; \Honetwo),\\
	\ppsintt &\relbar\joinrel\rightharpoonup \ppsitt && \quad \text{ weakly}  &&\text{ in } &&L^2(0,T; \Honezero),
	\end{alignedat} 
	\end{equation} 
	as $n\rightarrow \infty$. From \eqref{weak_limits_1}, by \cite[Theorem 3.1.1]{zheng2004nonlinear}, there is a subsequence (again not relabeled), such that 
	\begin{equation} \label{weak_limits_1_cor}
	\begin{alignedat}{4} 
	\ppsin &\longrightarrow \ppsi && \quad \text{ strongly}  &&\text{ in } &&L^2(0,T; \Honezero),  \\
	\ppsint &\longrightarrow \ppsit && \quad \text{ strongly}  &&\text{ in } &&L^2(0,T; \Honezero).
	\end{alignedat} 
	\end{equation} 
	Furthermore, by Young's convolution inequality 
	\begin{equation}
	\begin{aligned}
	&\|\frakKone* \Delta \ppsint\|_{L^\infty(L^2)} \leq \|\frakKone\|_{{\mathcal{M}}}\|\Delta \ppsint\|_{L^\infty(L^2)} , \\
	& \|\frakKtwo* \Delta \ppsintt\|_{L^2(H^{-1})} \leq \|\frakKtwo\|_{{\mathcal{M}}}\|\Delta \ppsintt\|_{L^2(H^{-1})},
	\end{aligned}
	\end{equation}
	so we conclude that (up to a subsequence) 
	\begin{equation} \label{weak_limits_2}
	\begin{alignedat}{4} 
	\frakKone* \Delta \ppsint &\relbar\joinrel\rightharpoonup \frakKone* \Delta \ppsit\ && \quad\text{ weakly-$\star$}   &&\text{ in } &&L^\infty(0,T; L^2(\Omega)), \\
	\frakKtwo* \Delta \ppsintt &\relbar\joinrel\rightharpoonup \frakKtwo* \Delta \ppsitt && \quad \text{ weakly}  &&\text{ in } &&L^2(0,T; H^{-1}(\Omega)).
	\end{alignedat} 
	\end{equation}
	By \eqref{bootstrap_Hneg} and Lemma~\ref{Lemma:Caputo_seq_compact}, we also have
	\begin{equation} \label{weak_limits_4}
	\begin{alignedat}{4} 
	\taua \frakKone * \ppsinttt&\relbar\joinrel\rightharpoonup 	\taua \frakKone * \ppsittt && \quad\text{ weakly}    &&\text{ in } &&L^2(0,T; \Hneg).
	\end{alignedat} 
	\end{equation}
	This allows us to pass to the limit in the semi-discrete equation. From \eqref{weak_limits_1_cor}, and uniqueness of limits, we conclude that $(\ppsi, \ppsit) \vert_{t=0}=(\ppsi_0, \ppsi_1)$. It remains to interpret how $\ppsi_2$ is attained. Following~\cite[Ch.\ 7]{evans2010partial}, let $v \in C^1([0,T]; \Honezero)$ with $v(T)=v_t(T)=0$. We have
	\begin{equation}\label{Galerkin_weak}
	\begin{aligned}
	\begin{multlined}[t]  	-\taua \intTO \frakKone* \ppsitt v_t \dxs - \taua \intT (\frakKone \ppsitt(0), v)_{L^2} \ds \\
	+\intT ( \aaa \ppsitt-c^2  \bbb \Delta u 
	-   {\tau^a} c^2 \frakKone* \Delta \ppsit , v)_{L^2}\ds 
	+ \intT \ttbdelta  ( \frakKtwo* \nabla \ppsitt, \nabla v)_{L^2}\ds= \intT (f, v)\ds.
	\end{multlined}
	\end{aligned}
	\end{equation}
	For the Galerkin approximation, we similarly have
	\begin{equation} \label{Galerkin_weak_un}
	\begin{aligned}
	\begin{multlined}[t]  -	\taua \intTO \frakKone* \ppsintt v_t \dxs - \taua \intT (\frakKone \ppsintt(0), v)_{L^2} \ds \\
	+\intT ( \aaa \ppsintt-c^2  \bbb \Delta \ppsin
	-   {\tau^a} c^2 \frakKone* \Delta \ppsint , v)_{L^2}\ds 
	+ \intT \ttbdelta  ( \frakKtwo* \nabla \ppsintt, \nabla v)_{L^2}\ds\\= \intT (f, v)\ds.
	\end{multlined}
	\end{aligned}
	\end{equation} 
	Note that with $L^\infty$ regularity in time for $\ppsintt$ obtained in \eqref{Linf_bound_untt}, we have $\taua(\frakKone* \ppsintt)(0) =0$.
	We also have that $\taua(\frakKone* \ppsitt)(0)=0$ since due to \eqref{weak_limits_4} and \eqref{Linf_bound_untt}
	\begin{equation} \label{Linf_bound_utt}
	\begin{aligned}
	\taua\|\ppsitt\|_{L^\infty(H^{-1})} \leq &\, \taua \liminf_{n \rightarrow \infty}	\|\ppsintt\|_{L^\infty(H^{-1})} \leq C.
	\end{aligned}
	\end{equation}
	The uniform bound on $\ppsintt$ taken together with \eqref{weak_limits_1}--\eqref{weak_limits_4}, allows us to pass to the limit (in possibly a subsequence) in \eqref{Galerkin_weak_un}. Comparing the resulting identities gives
	\[
	\taua \intT (\frakKone \ppsitt(0), v)_{L^2} \ds = \taua \intT (\frakKone \ppsi_2, v)_{L^2} \ds,
	\]
	from which we conclude (since $\frakKone\neq0$) that $\ppsitt(0)=\ppsi_2$.  \\[2mm]
	\noindent \emph{Uniqueness.} Uniqueness of solutions in $\calXBW$ follows by proving that the only solution of the homogeneous problem
	with $f=0$ and zero data is $\ppsi=0$.  Since $- \Delta \ppsitt$ is not a valid test function in this setting, we test this problem with $\ppsitt \in L^2(0,T; \Honezero)$ instead. Similarly to before, and using the additional smoothness assumption on $\bbb$, we have
	\begin{equation} \label{est_uniqeness_lowerorder}
	\begin{aligned}
	&	\int_0^t \nLtwo{\sqrt{\aaa} \ppsitt}^2 + \underline{c} \|\nabla \ppsit\|^2_{L^\infty(L^2)}\\
	\leq&\, 	\int_0^t\left\{
	-c^2(\bbb\nabla \ppsi,  \nabla \ppsitt)_{L^2} -c^2(\ppsi \nabla \bbb,  \nabla \ppsitt)_{L^2}
	\right\}\ds\\
	=&\, \begin{multlined}[t] 
	-c^2(\bbb(t)\nabla \ppsi(t),  \nabla \ppsit(t))_{L^2} -c^2(\ppsi(t) \nabla \bbb(t),  \nabla \ppsit(t))_{L^2}
	\\ -\int_0^t\left\{
	-c^2(\bbb_t\nabla \ppsi+\bbb\nabla \ppsit,  \nabla \ppsit)_{L^2} -c^2(\ppsit \nabla \bbb+\ppsi \nabla \bbb_t,  \nabla \ppsit)_{L^2}\right\}\ds. \end{multlined}
	\end{aligned}
	\end{equation}
	We can further bound the right-hand side using H\"older's inequality:
	\begin{equation}
	\begin{aligned}
	&\left| 	-c^2(\bbb(t)\nabla \ppsi(t),  \nabla \ppsit(t))_{L^2} -c^2(\ppsi(t) \nabla \bbb(t),  \nabla \ppsit(t))_{L^2} \right| \\
	\lesssim&\, \olbbb \|\nabla \ppsi(t)\|_{L^2} \|\nabla \ppsit(t)\|_{L^2}+\|\ppsi(t)\|_{L^4}\|\nabla \bbb(t)\|_{L^4}\|\nabla \ppsit(t)\|_{L^2}
	\end{aligned}
	\end{equation}
	a.e.\ in time. Similarly,
	\begin{equation}
	\begin{aligned}
	&\left| -\int_0^t\left\{
	-c^2(\bbb_t\nabla \ppsi+\bbb\nabla \ppsit,  \nabla \ppsit)_{L^2} -c^2(\ppsit \nabla \bbb+\ppsi \nabla \bbb_t,  \nabla \ppsit)_{L^2}\right\}\ds \right| \\
	\lesssim&\,\begin{multlined}[t] \|\bbb_t\|_{L^1(L^\infty)}\|\nabla \ppsi\|_{L^\infty(L^2)}\|\nabla \ppsit\|_{L^\infty(L^2)}+ \olbbb \|\nabla \ppsit\|^2_{L^2(L^2)} \\
	+\|\nabla \bbb\|_{L^2(L^4)}\|\ppsit\|_{L^2(L^4)}\|\nabla \ppsit\|_{L^\infty(L^2)}+ \|\nabla \bbb_t\|_{L^1(L^4)}\|\ppsi\|_{L^\infty(L^4)}\|\nabla \ppsit\|_{L^\infty(L^2)}.
	\end{multlined}
	\end{aligned}
	\end{equation}
	Employing these bounds in \eqref{est_uniqeness_lowerorder} leads to
	\[
	\int_0^t \nLtwo{\ppsitt}^2 +  \|\nabla \ppsit\|^2_{L^\infty(L^2)} \leq 0,
	\]
	from which (combined with the zero data) it follows that $\ppsi=0$. This concludes the proof.
\end{proof}
\subsection{Uniform well-posedness with Westervelt--Blackstock nonlinearities}
To treat equations \eqref{general_eq_ppsi} with Westervelt--Blackstock nonlinearities under assumptions \eqref{reg_kernels} and \eqref{Athree_GFEIII}--\eqref{Afive_GFEIII} on the two kernels, we next set up a fixed-point mapping $\calT: \calBBW \ni \ppsi^* \mapsto \ppsi$, 
where $\ppsi$ solves \eqref{IBVP_Westtype_lin} with
\begin{equation}
\begin{aligned}
\aaa(\ppsi^*) =\,1+2k_1 \ppsi^*, \quad \bbb(\ppsi^*) = 1-2k_2 \ppsi^*, \quad
f=\, -\calN(\ppsi^*_t)= -2k_3 (\ppsit^*)^2,
\end{aligned}
\end{equation}
and $\ppsi^*$ is taken from the ball
\begin{equation}
\begin{aligned}
\calBBW =\left \{ \ppsi^* \in \XWest: \right.&\, \, \|\ppsi^*\|_{\XWest} \leq R, \quad 
\left	(\ppsi^*, \ppsit^*, \ppsitt^*)\vert_{t=0} =(\ppsi_0, \ppsi_1, \ppsi_2) \}. \right.
\end{aligned}
\end{equation}
The radius $R>0$ is independent of $\tau$ and will be determined by the upcoming proof. 
\begin{theorem} \label{Thm:WellP_GFEIII_BWest}
	Let $T>0$ and $\tau \in (0, \bar{\tau}]$. Let $\delta>0$ and $k_{1,2,3} \in \R$. Let assumptions \eqref{reg_kernels}, \eqref{Athree_GFEIII}--\eqref{Afive_GFEIII} on the kernels hold. There exists a size of data $r=r(T)>0$, independent of $\tau$, such that if
	\begin{equation}
	\|\ppsi_0\|^2_{H^2}+\|\ppsi_1\|^2_{H^2} + \bar{\tau}^a \|\ppsi_2\|^2_{H^1} \leq {r}^2,
	\end{equation}
	then there is a unique solution $\ppsi \in \calBBW$ of 
	\begin{equation} \label{IBVP_Westtype_nonlin}
	\begin{aligned}
	\begin{multlined}[t]	\taua  \frakKone * \ppsittt+(1+2k_1 \ppsi) \ppsitt
	-c^2 (1-2k_2 \ppsi) \Delta \ppsi 
	-   \taua c^2 \frakKone* \Delta \ppsit- \ttbdelta  \frakKtwo* \Delta \ppsitt\\+2k_3 \ppsit^2=0 ,\end{multlined}
	\end{aligned}
	\end{equation}
with initial \eqref{IC} and boundary \eqref{BC} conditions.	The solution satisfies the following estimate: 
\begin{equation}
	\begin{aligned}
		\|\ppsi\|_{\XWest}^2 \lesssim_T \|\ppsi_0\|^2_{H^2}+\|\ppsi_1\|^2_{H^2} + \taua \|\ppsi_2\|^2_{H^1},
	\end{aligned}
\end{equation}
where the hidden constant does not depend on $\tau$.
\end{theorem}
\begin{proof}
	We check that the conditions of the Banach fixed-point theorem are satisfied for the introduced mapping $\calT$. We note that the set $\calB^{\textup{W}}$ is non-empty, as the solution of the linear problem with $k_1=k_2=k_3=0$  belongs to it for small enough (with respect to $R$) initial data.  \\
	\indent To show that $\calT(\calBBW) \subset \calBBW$, take $\ppsi^* \in \calB^{\textup{W}} \subset \XWest$. Then the smoothness assumptions on $\aaa$ and $\bbb$ in Proposition~\ref{Prop:Wellp_GFEIII_BWest_lin} are fulfilled and the smallness assumption on $\aaa$ is fulfilled by reducing $R>0$. The non-degeneracy condition on $\aaa$ is fulfilled for small enough $R$. To see this, note that due to the embedding $H^2(\Omega) \hookrightarrow L^\infty(\Omega)$, we have
	\[
	\|2k_1 \ppsi^*\|_{L^\infty(L^\infty)} \leq C(\Omega, T)|k_1|\|\ppsi^*\|_{L^\infty(H^2)} \leq C(\Omega, T)|k_1| R.
	\]
	Thus $R>0$ should be chosen so that
	\[
	1-C(\Omega, T)|k_1|R \geq \ulaaa >0.
	\]
	Furthermore, we have 
	\begin{equation}
	\begin{aligned}
	\|\calN(\ppsi^*_t)\|_{W^{1,1}(L^2)}
	\lesssim&\, \begin{multlined}[t]  \|\ppsit^*\|^2_{L^2(L^4)} + \|\ppsit^*\|_{L^\infty(L^4)} \|\ppsitt^*\|_{L^1(L^4)} 
	\end{multlined}
	\leq C(\Omega, T) R^2,
	\end{aligned}
	\end{equation}
	where we have relied on the embedding $\Hone \hookrightarrow L^4(\Om)$. By employing Proposition~\ref{Prop:Wellp_GFEIII_BWest_lin}, we have 
	\begin{equation}
	\begin{aligned}
	\|\ppsi\|_{\XWest} 
	\leq\,
	C_1 \exp \left(C_2T(1+\|\bbb_t\|^2_{L^1(L^\infty)})\right)(\|\ppsi_0\|^2_{H^2}+\|\ppsi_1\|^2_{H^2} + \taua \|\nabla \ppsi_2\|^2_{L^2}+ \|f\|_{W^{1,1}(L^2)}^2).
	\end{aligned}
	\end{equation}
	Since
	\begin{equation}
	\begin{aligned}
	\|\bbb_t\|_{L^1(L^\infty)} \lesssim&\,  	\|\ppsi^*_t\|_{L^1(L^\infty)} \lesssim T \|\ppsi^*_t\|_{L^\infty(H^2)}  \lesssim TR,
	\end{aligned}
	\end{equation}
	we have
	\begin{equation}
	\begin{aligned}
	\|\ppsi\|_{\XWest} 
	\leq&\, C_1 \exp \left(C_2T(1+T^2R^2)\right)(r^2+CR^4).
	\end{aligned}
	\end{equation}
	Therefore, $\ppsi \in \calB^{\textup{W}}$ for sufficiently small radius $R$ and data size $r$. \\
	
	\noindent To prove strict contractivity, let $\calT \ppsi^{*} =\ppsi$ and $\calT v^*=v$. Denote $\phi=\ppsi-v$ and $\phi^*= \ppsi^*-v^*$. Then $\phi$ solves 
	\begin{equation}
	\begin{aligned}
	&\begin{multlined}[t]	\taua \frakKone * \phi_{ttt}+\aaa(\ppsi^*) \phi_{tt}
	-c^2 \bbb(\ppsi^*) \Delta \phi
	-   \taua c^2 \frakKone* \Delta \phi_t -\ttbdelta  \frakKtwo* \Delta \phi_{tt}\end{multlined}\\
	=&\, -2k_1 \phi^* v_{tt} -2k_2 \phi^*\Delta v-2k_3\phi^*_t (\ppsit^*+v_t^*):=f.
	\end{aligned}
	\end{equation}
	Note that we cannot prove contractivity with respect to the $\|\cdot\|_{\XWest}$ norm by exploiting the linear bound in Proposition~\ref{Prop:Wellp_GFEIII_BWest_lin}, as the right-hand side of this equation does not belong to $W^{1,1}(0,T; L^2(\Omega))$ due to the first term on the right-hand side in the last line (we do not have control over the time derivative of $v_{tt}$). Similarly to the proof of uniqueness for the linear problem, we can test this equation with $\phi_{tt}$ to obtain  
	\begin{equation}
	\begin{aligned}
	&\int_0^t \nLtwo{\phi_{tt}}^2\ds +  \|\nabla \phi_t\|^2_{L^\infty(L^2)}\\
	\lesssim&\, \begin{multlined}[t]\|f\|^2_{L^2(L^2)}+\left| 	-c^2(\bbb(\ppsi^*)\nabla \phi(t),  \nabla \phi_t(t))_{L^2} +2k_2c^2(\phi(t) \nabla \ppsi^*,  \nabla \phi_t(t))_{L^2} \right|\\
	+\left| -\int_0^t\left\{
	2k_2c^2(\ppsit^*\nabla \phi+\bbb(\ppsi^*)\nabla \phi_t,  \nabla \phi_t)_{L^2} +2k_2c^2(\ppsi_t \nabla \ppsi^*+\phi \nabla \ppsit^*,  \nabla \phi_t)_{L^2}\right\}\ds \right|.
	\end{multlined}
	\end{aligned}
	\end{equation}
	The terms on the right-hand side not containing $f$ can be further estimated as in the proof of uniqueness for the linear problem. We additionally have
	\begin{equation}
	\begin{aligned}
	\|f\|_{L^2(L^2)}
	\lesssim&\, \begin{multlined}[t]\|\phi^*\|_{L^\infty(L^4)}\|v_{tt}\|_{L^2(L^4)}+\|\phi^*\|_{L^\infty(L^4)}\|\Delta v\|_{L^\infty(L^2)}\\
	+\|\phi_t^*\|_{L^2(L^4)}\|\ppsit^*+v_t^*\|_{L^\infty(L^4)}.
	\end{multlined}
	\end{aligned}
	\end{equation}
	Altogether, it follows that
	\begin{equation}
	\begin{aligned}
	\int_0^t \nLtwo{\phi_{tt}}^2\ds +  \|\nabla \phi\|^2_{W^{1,\infty}(L^2)}
	\lesssim\, R^2 \left(\int_0^t \nLtwo{\phi^*_{tt}}^2\ds +  \|\nabla \phi^*\|^2_{W^{1,\infty}(L^2)} \right).
	\end{aligned}
	\end{equation}
	Therefore, we obtain strict contractivity of the mapping $\calT$ with respect to the norm of $W^{1,\infty}(0,T; \Honezero) \cap H^2(0,T; L^2(\Omega))$ by reducing $R$. The closedness of $\calBBW$ with respect to this norm can be argued similarly to, e.g.,~\cite[Theorem 4.1]{kaltenbacher2022parabolic} to conclude the proof..
\end{proof}
{We note that this uniform well-posedness result generalizes~\cite[Theorem 4.1]{kaltenbacher2019jordan}, where the \ref{fJMGTIII} equation with Westervelt nonlinearities is considered, to equation \eqref{IBVP_Westtype_nonlin} with Westervelt--Blackstock nonlinearities and general kernels satisfying the assumptions of this section.}
\subsection{Weak singular limit with Westervelt--Blackstock nonlinearities}\label{Subsec:Weaklim_tau_GFEIII}
Equipped with the previous uniform analysis, we are now ready to discuss the limiting behavior of these equations as $\tau \searrow 0$. Let $\tau \in (0, \bar{\tau}]$. Under the assumptions of Theorem~\ref{Thm:WellP_GFEIII_BWest} with the uniform smallness condition
\begin{equation}
\|\ppsi^\tau_0\|^2_{H^2}+\|\ppsi^\tau_1\|^2_{H^2} + \bar{\tau}^a \|\ppsi^\tau_2\|^2_{H^1} \leq r^2,
\end{equation}
let $\ppsi^{\tau}$ be the solution of
\begin{equation} \label{IBVP_tau_GFEIII_West}
\left \{	\begin{aligned}
&\begin{multlined}[t] \taua  \frakKone * \ppsittt^\tau+\aaa(\ppsi^\tau) \ppsitt^\tau
-c^2 \bbb(\ppsi^\tau) \Delta \ppsi^\tau 
-   {\tau^a} c^2 \frakKone* \Delta \ppsit^\tau \\ \hspace*{4cm}-\ttbdelta  \frakKtwo* \Delta \ppsitt^\tau+\calN(\ppsit^\tau)=0 \ \textup{ in } \Omega \times (0,T),\end{multlined}\\
&\ppsi^\tau\vert_{\partial \Omega}=0, \\
&	(\ppsi^\tau, \ppsit^\tau, \ppsitt^\tau)\vert_{t=0} = (\ppsi^\tau_0, \ppsi^\tau_1, \ppsi^\tau_2).
\end{aligned} \right. 
\end{equation}
Based on the previous analysis and the obtained uniform bounds with respect to the thermal relaxation time $\tau$, there exists a subsequence, not relabeled, such that 
\begin{equation} \label{weak_limits_1_tau_GFEIII_West}
\begin{alignedat}{4} 
\ppsi^\tau&\relbar\joinrel\rightharpoonup \ppsi && \quad \text{ weakly-$\star$}  &&\text{ in } &&L^\infty(0,T; \Honetwo),  \\
\ppsit^\tau &\relbar\joinrel\rightharpoonup \ppsit && \quad \text{ weakly-$\star$}  &&\text{ in } &&L^\infty(0,T; \Honetwo),\\
\ppsitt^\tau &\relbar\joinrel\rightharpoonup \ppsitt && \quad \text{ weakly}  &&\text{ in } &&L^2(0,T; \Honezero).
\end{alignedat} 
\end{equation} 
By the Aubin--Lions--Simon lemma (see~\cite[Corollary 4]{simon1986compact}), this further implies that
\begin{equation} \label{weak_limits_2_tau_GFEIII_West}
\begin{alignedat}{4} 
\ppsi^\tau&\longrightarrow \ppsi && \quad \text{ strongly}  &&\text{ in } &&C([0,T]; \Honezero),  \\
\ppsit^\tau &\longrightarrow \ppsit && \quad \text{ strongly}  &&\text{ in } &&C([0,T]; \Honezero).
\end{alignedat} 
\end{equation} 
Therefore, we know that
\begin{equation} \label{limits_tau_initial_BW}
\begin{alignedat}{4} 
\ppsi^\tau_0=\ppsi^\tau(0)&\longrightarrow \ppsi(0):=\ppsi_0 && \quad \text{ strongly}  &&\text{ in } && \Honezero,  \\
\ppsi^\tau_1=\ppsit^\tau(0) &\longrightarrow \ppsit(0):=\ppsi_1 && \quad \text{ strongly}  &&\text{ in } && \Honezero.
\end{alignedat} 
\end{equation} 
On top of this, by \eqref{weak_limits_1_tau_GFEIII_West}, we have
\begin{equation} \label{weak_limits_3_tau}
\begin{alignedat}{4} 
\frakKtwo * \nabla \ppsitt^\tau &\relbar\joinrel\rightharpoonup 	\frakKtwo * \nabla \ppsitt && \quad \text{ weakly}  &&\text{ in } &&L^2(0,T; \Ltwo).
\end{alignedat} 
\end{equation} 
We next prove that $u$ satisfies the limiting problem. Let $v \in C_0^\infty([0,T]; C_0^\infty(\Omega))$. We have with $\bar{\ppsi}=\ppsi-\ppsi^\tau$:
\begin{equation}
\begin{aligned}
&\begin{multlined}[t] \intTO \aaa(\ppsi) \ppsitt v \dxs-c^2 \intTO \bbb(\ppsi) \Delta u v \dxs +\rt^b \delta \intTO  \frakKtwo* \nabla \ppsitt \cdot \nabla v \dxs\\
+\intTO \calN(\ppsit) v \dxs = \textup{rhs} \end{multlined}
\end{aligned}
\end{equation}
with the right-hand side
\begin{equation}
	\begin{aligned}
\textup{rhs}:=&	\begin{multlined}[t]  \intTO  \aaa(\ppsi)\bar{\ppsi}_{tt} v \dxs- c^2 \intTO \bbb(\ppsi) \Delta \bar{\ppsi} v \dxs+\rt^b \delta \intTO  \frakKtwo* \nabla \bar{\ppsi}_{tt} \cdot \nabla v \dxs \\
- \intTO  \taua \frakKone * \ppsittt^\tau v \dxs + \taua c^2 \intTO \frakKone* \Delta \ppsit^\tau v \dxs \\
-\intTO (\aaa(\ppsi^\tau)-\aaa(\ppsi)) \ppsitt^\tau v \dxs + c^2\intTO (\bbb(\ppsi^\tau)-\bbb(u))\Delta \ppsi^\tau v \dxs\\ -\intTO (\calN(\ppsit^\tau) -\calN(\ppsit))v \dxs.  \end{multlined}
\end{aligned}
\end{equation}
We wish to prove that $\textup{rhs}$ tends to zero as $\tau \searrow 0$. To this end, we rely on the established weak convergence. We first discuss the terms involving the kernels. Note that
\begin{equation}
\begin{aligned}
&\intTO  \taua \frakKone * \ppsittt^\tau v \dxs\\
=&\,  - \taua \intTO   \frakKone * \ppsitt^\tau \, v_t \dxs- \taua \intTO  \frakKone(s) u_2 \, v \dxs	\rightarrow  0 \quad \text{as } \tau \searrow 0.
\end{aligned}
\end{equation}
Above we have relied on the uniform bound on
\[
\left |\intTO   \frakKone * \ppsitt^\tau \, v_t \dxs \right| \leq \|\frakKone\|_{\mathcal{M}} \|\ppsitt^\tau\|_{L^2(L^2)}\|v_t\|_{L^2(L^2)}.
\]
Similarly, we  have
\begin{equation}
\taua c^2 \intTO \frakKone* \Delta \ppsit^\tau v \dxs\rightarrow  0 \quad \text{as } \tau \searrow 0.
\end{equation}
By the limit in \eqref{weak_limits_3_tau}, it also follows that
\begin{equation}
\rt^b \delta \intTO  \frakKtwo* \nabla \bar{\ppsi}_{tt} \cdot \nabla v \dxs \rightarrow  0 \quad \text{as } \tau \searrow 0.
\end{equation}
By relying on \eqref{weak_limits_1_tau_GFEIII_West} and the equivalence of the norms $\|{\aaa(\ppsi) v}\|_{L^2}$, $\|v\|_{L^2}$, and $\|\bbb(\ppsi) v\|_{L^2}$ (under the assumptions of Theorem~\ref{Thm:WellP_GFEIII_BWest}), we can conclude that 
\begin{equation} 
\begin{alignedat}{4} 
\aaa(\ppsi) \ppsitt^\tau &\relbar\joinrel\rightharpoonup 	\aaa(\ppsi) \ppsitt && \quad \text{ weakly}  &&\text{ in } &&L^2(0,T; \Ltwo), \\
\bbb(\ppsi) \Delta \ppsi^\tau &\relbar\joinrel\rightharpoonup 		\bbb(\ppsi) \Delta \ppsi && \quad \text{ weakly}  &&\text{ in } &&L^2(0,T; \Ltwo), 
\end{alignedat} 
\end{equation} 
and thus 
\begin{equation}
\begin{aligned}
\intTO  \aaa(\ppsi)\bar{\ppsi}_{tt} v \dxs- c^2 \intTO \bbb(\ppsi) \Delta \bar{\ppsi} v \dxs  \rightarrow  0 \quad \text{as } \tau \searrow 0.
\end{aligned}
\end{equation}
Additionally, we have
\begin{equation}
\begin{aligned}
&-\intTO (\aaa(\ppsi^\tau)-\aaa(\ppsi) \ppsitt^\tau v \dxs + c^2\intTO (\bbb(\ppsi^\tau)-\bbb(\ppsi))\Delta \ppsi^\tau v \dxs \\
=& \, -2\intTO k_1(\ppsi^\tau-\ppsi) \ppsitt^\tau v \dxs - 2c^2 k_2\intTO (\ppsi^\tau-\ppsi)\Delta \ppsi^\tau v \dxs,	
\end{aligned}
\end{equation}
which tends to zero as well thanks to \eqref{weak_limits_1_tau_GFEIII_West} and \eqref{weak_limits_2_tau_GFEIII_West}. Finally,
\begin{equation}
\begin{aligned}
\intTO (\calN(\ppsit^\tau) -\calN(\ppsit))v \dxs =\, 2k_3\intTO \bar{\ppsi}_t (\ppsit^\tau+\ppsit)v \dxs, 
\end{aligned}
\end{equation}
which also tends to zero on account of again \eqref{weak_limits_1_tau_GFEIII_West} and the uniform bounds in \eqref{weak_limits_1_tau_GFEIII_West}. \\
\indent We have thus proven that there is a subsequence of $\{\ppsi^\tau\}_{\tau \in (0, \bar{\tau}]}$ that converges to a solution $\ppsi \in \XWest$ of the following problem: 
\begin{equation} \label{IBVP_limit_GFEIII_West}
\left \{	\begin{aligned}
&\begin{multlined}[t] (1+2k_1\ppsi) \ppsitt
-c^2 (1-2k_2 \ppsi) \Delta \ppsi -\ttbdelta  \frakKtwo* \Delta \ppsitt+2k_3 \ppsit^2=0 ,\end{multlined}\\
&\ppsi\vert_{\partial \Omega}=0, \\
&	(\ppsi, \ppsit)\vert_{t=0} =(\ppsi_0, \ppsi_1).
\end{aligned} \right. 
\end{equation}
The initial conditions $(\ppsi_0, \ppsi_1)$ are obtained in the limit of $(\ppsi_0^\tau, \ppsi^\tau_1)$ as $\tau \searrow 0$ in the sense of \eqref{limits_tau_initial_BW}.  The uniqueness of this solution in $\XWest$ can be shown by noting that the difference $\bar{\ppsi}=\ppsi^{(1)}-\ppsi^{(2)}$ of two solutions would have to satisfy
\begin{equation} 
\begin{aligned}
&(1+2k_1\ppsi^{(1)}) \bar{\ppsi}_{tt}
-c^2 (1-2k_2 \ppsi^{(1)}) \Delta \bar{\ppsi} -\rt^b\delta  \frakKtwo* \Delta \bar{\ppsi}_{tt} \\
=&\,-2k_1 \bar{\ppsi}\ppsi_{tt}^{(2)}-2k_2 c^2 \bar{\ppsi} \Delta \ppsi-2k_3 \bar{\ppsi}_t(\ppsit^{(1)}+\ppsit^{(2)})
\end{aligned}  
\end{equation}
with zero initial data and then testing this problem by $\bar{\ppsi}_{tt}$, similarly to the proof of contractivity in Theorem~\ref{Thm:WellP_GFEIII_BWest}. Thus by a subsequence-subsequence argument, the whole sequence $\{\ppsi^\tau\}_{\tau \in (0, \bar{\tau}]}$ converges to $\ppsi$. Altogether, we arrive at the following result.
\begin{theorem} \label{Thm:WeakLim_GFEIII_BWest}
	Let the assumptions of Theorem~\ref{Thm:WellP_GFEIII_BWest} hold for \eqref{IBVP_tau_GFEIII_West}. Then the family $\{\ppsi^\tau\}_{\tau \in (0, \bar{\tau}]}$ of solutions to
	\eqref{IBVP_tau_GFEIII_West}
	converges as $\tau \searrow 0$ in the sense of \eqref{weak_limits_1_tau_GFEIII_West}, \eqref{weak_limits_2_tau_GFEIII_West} to the solution $\ppsi \in \XWest$ of \eqref{IBVP_limit_GFEIII_West}.
\end{theorem}
{Note that as a by-product of this analysis, we obtain the unique solvability of the limiting problem given in \eqref{IBVP_limit_GFEIII_West} for small data in $\Honetwo \times \Honetwo$.} 
\subsection{Well-posedness of a linearized Kuznetsov--Blackstock problem}  We next show how the previous arguments can be adapted to the equations with Kuznetsov--Blackstock nonlinearities. That is, we consider again the equation 
\begin{equation} \tag{\ref{JMGT_Kuzn_nonlocal_lin}}
\begin{aligned}
\begin{multlined}[t]
\taua \frakKone\Lconv \ppsi_{ttt} +\aaa\ppsi_{tt}
-c^2\bbb\Delta \ppsi 
- \taua c^2 \frakKone\Lconv \Delta \ppsi_t	-	\ttbdelta \frakKtwo\Lconv \Delta \ppsi_{tt}=f,
\end{multlined}
\end{aligned}
\end{equation}
but now we have in mind that
\begin{itemize}
	\item $\aaa=1+2{k}_1\ppsit$, \ $\bbb=1-2k_2\ppsit$, and $f= -2 k_3 \nabla \ppsi \cdot \nabla \ppsit$.
\end{itemize}
As mentioned before, we need more regularity of the variable coefficients and data in this setting. In particular, we need 
\begin{equation} \label{smoothness_a_b}
	\aaa \in C([0,T]; \Honetwo), \quad \bbb \in  H^1(0,T; \Honetwo).
\end{equation}
We still assume that $\aaa$ does not degenerate so that there exist $\ulaaa$, $\olaaa>0$, independent of $\tau$, such that
\eqref{nondeg_aaa_GFEIII} holds.  Furthermore, we need smallness in the following sense:  there exists as small enough $m>0$, independent of $\tau$, such that  
\begin{equation} \label{smallness_a}
	\begin{aligned}
		\|\D\aaa\|_{L^\infty(L^2)} \leq m.
	\end{aligned}
\end{equation} Let also $f \in W^{1,1}(0,T; H^1(\Omega))$. Now the $\tau$-independent solution space is
\begin{equation}
\begin{aligned}
	\calXBK= \{\ppsi \in W^{1,\infty}(0,T; \Honethree):\, \ppsitt \in L^2(0,T; \Honetwo)\}.
\end{aligned} 
\end{equation}
The analysis of this linearized problem can be conducted as before through a Faedo--Galerkin procedure based on the smooth eigenfunctions of the Dirichlet--Laplace operator; cf.\ Appendix~\ref{Appendix:Galerkin}. In this way, we can rely on $\Delta \ppsi^{(n)}= \Delta \ppsit^{(n)}=\Delta \ppsitt^{(n)}=$ on $\partial \Omega$. In what follows, we only discuss the testing procedure as the other details follow analogously to the Westervelt--Blackstock case. For ease of notation, we drop the superscript $(n)$. By testing the (semi-discretized) problem with $\D^2\ppsi_{tt}$ and employing coercivity assumptions \eqref{Afour_GFEIII} and \eqref{Afive_GFEIII} on the kernels, we obtain
\begin{equation}\label{est_1_GFEIII_BK}
\begin{aligned}
&
\int_0^t \nLtwo{\sqrt{\aaa}\D\ppsi_{tt}}^2
+\underline{c} {\|\nabla\D\ppsi_{t}\|^2_{L^\infty(L^2)}}\\
\leq&\,  \begin{multlined}[t]
-c^2 (\nabla[\bbb\D\ppsi],\nabla\D\ppsi_t)_{L^2}\Big \vert_0^t
+\overline{C}\nLtwo{\nabla\D\ppsi_1}^2+C \taua \|\Delta \ppsi_2\|^2_{L^2}
\\
+\int_0^t\Bigl(
-(\ppsi_{tt}\, \D\aaa+\nabla\ppsi_{tt}\cdot\nabla\aaa,\D\ppsi_{tt})_{L^2}
+c^2(\nabla[\bbb\D\ppsi]_t,\nabla\D\ppsi_t)_{L^2}
\Bigr)\ds \\
+\int_0^t (f,\D^2\ppsi_{tt})_{L^2}\ds,\end{multlined}
\end{aligned}
\end{equation}
where we have used the identity
$\aaa\D\ppsi_{tt}-\D[\aaa\ppsi_{tt}]=-\ppsi_{tt}\, \D\aaa-\nabla\ppsi_{tt}\cdot\nabla\aaa$,
and integrated by parts to obtain
\[
\int_0^t (\bbb\D\ppsi,\D^2\ppsi_{tt})_{L^2}\ds
=\int_0^t (\nabla[\bbb\D\ppsi]_t,\nabla\D\ppsi_t)_{L^2}\ds - (\nabla[\bbb\D\ppsi],\nabla\D\ppsi_t)_{L^2} \Big\vert_0^t.
\]
The terms on the right hand side of \eqref{est_1_GFEIII_BK} can be further estimated as follows. First, using the embeddings $H^1(\Omega) \hookrightarrow L^4(\Omega)$ and $H^2(\Omega) \hookrightarrow L^\infty(\Omega)$, we have
\[
\begin{aligned}
-(\nabla[\bbb(t)\D\ppsi(t)],\nabla\D\ppsi_t(t))_{L^2}
\leq&\,  (\|\nabla\bbb\D\ppsi\|_{L^\infty(L^2)} + 
\|\bbb\nabla\D\ppsi\|_{L^\infty(L^2)}) \|\nabla\D\ppsi_t\|_{L^\infty(L^2)}\\
\lesssim&\, \left(
\|\D\bbb\|_{L^\infty(L^2)}+1\right) \|\nabla\D\ppsi\|_{L^\infty(L^2)}\|\nabla\D\ppsi_t\|_{L^\infty(L^2)}, 
\end{aligned}
\]
where we have also used the fact that $\bbb-1$ vanishes on the boundary and that $\D(\bbb-1)=\D\bbb$. We can employ
\[
\|\nabla\D\ppsi\|_{L^\infty(L^2)} \leq \sqrt{T} \|\nabla\D\ppsit\|_{L^2(L^2)}+\|\nabla \D \ppsi_0\|_{L^2}
\]
to further bound the $\ppsi$ term.  We also have 
\[
\begin{aligned}
\int_0^t-(\ppsi_{tt}\, \D\aaa+\nabla\ppsi_{tt}\cdot\nabla\aaa,\D\ppsi_{tt})_{L^2}\ds
\leq&\, \|\ppsi_{tt}\, \D\aaa+\nabla\ppsi_{tt}\cdot\nabla\aaa\|_{L^2(L^2)}
\|\D\ppsi_{tt}\|_{L^2(L^2)}\\
\lesssim&\, \|\D\aaa\|_{L^\infty(L^2)}
\|\D\ppsi_{tt}\|_{L^2(L^2)}^2\lesssim\,m
\|\D\ppsi_{tt}\|_{L^2(L^2)}^2.
\end{aligned}
\]
Furthermore,
\[
\begin{aligned}
&\int_0^t(\nabla[\bbb\D\ppsi]_t,\nabla\D\ppsi_t)_{L^2}\ds\\
\lesssim&\, \begin{multlined}[t]
 \|\D\bbb_t\|_{L^2(L^2)}
\|\nabla\D\ppsi\|_{L^\infty(L^2)}\|\nabla\D\ppsi_t\|_{L^2(L^2)}
+
\left( \|\D\bbb\|_{L^\infty(L^2)}+1\right)
\|\nabla\D\ppsi_t\|_{L^2(L^2)}^2 \end{multlined}\\
\lesssim&\, \begin{multlined}[t]
\|\D\bbb_t\|_{L^2(L^2)}
(\sqrt{T}\|\nabla\D\ppsi_t\|_{L^2(L^2)}+\|\nabla\D\ppsi_0\|_{L^2})\|\nabla\D\ppsi_t\|_{L^2(L^2)}
+
\left( \|\D\bbb\|_{L^\infty(L^2)}+1\right)
\|\nabla\D\ppsi_t\|_{L^2(L^2)}^2. \end{multlined}
\end{aligned}
\]
We can estimate the term containing $f$ using integration by parts as follows:
\begin{equation}
	\begin{aligned}
		\int_0^t (f,\D^2\ppsi_{tt})_{L^2}\ds
		 =&\, 	 (f,\D^2\ppsi_{t})_{L^2}\Big \vert_0^t-		\int_0^t (f_t,\D^2\ppsi_{t})_{L^2}\ds\\
		=&\, 	\begin{multlined}[t] -(\nabla f, \nabla \D \ppsi_{t})_{L^2}\Big \vert_0^t+\int_{\partial\Omega} f \nabla \D \ppsi_{t} \cdot n \, \textup{d}\Gamma \Big \vert_0^t-		\int_0^t (\nabla f_t, \nabla \D\ppsi_{t})_{L^2}\ds\\+	\int_0^t  \int_{\partial \Omega} f_t \nabla \D\ppsi_{t} \cdot n \, \textup{d}\Gamma \ds 	\end{multlined} \\
	\lesssim&\, \|f\|^2_{W^{1,1}(H^1)} + \varepsilon \| \Delta \ppsi_t\|^2_{L^\infty(H^1)}+\| \Delta \ppsi_1\|^2_{H^1}.
	\end{aligned}
\end{equation}
Employing the above bounds to further estimate the right-hand side terms in \eqref{est_1_GFEIII_BK} and assuming that $\|\D\aaa\|_{L^\infty(L^2)}$ is small enough (independently of $\tau$) thus leads to
\begin{equation}
	\begin{aligned}
		\int_0^t \nLtwo{\D\ppsi_{tt}}^2
+{\|\nabla\D\ppsi_{t}\|^2_{L^\infty(L^2)}} 
		\lesssim_T&\,  \begin{multlined}[t]	\nLtwo{\nabla\D\ppsi_0}^2+
			\| \Delta \ppsi_1\|^2_{H^1}+ \taua \|\Delta \ppsi_2\|^2_{L^2}+
			\|f\|^2_{W^{1,1}(H^1)}.\end{multlined}
	\end{aligned}
\end{equation}
We can further use
\[
\|\nabla\D\ppsi\|_{L^\infty(L^2)} \leq T \|\nabla\D\ppsit\|_{L^\infty(L^2)}+\|\nabla \D \ppsi_0\|_{L^2}
\]
and
\[
\|\Delta \ppsi\|_{L^\infty(L^2)} \lesssim_T \|\Delta \ppsitt\|_{L^2(L^2)} + \|\Delta \ppsi_0\|_{L^2}+ \|\Delta \ppsi_1\|_{L^2}
\]
to obtain a bound on $\|\ppsi\|_{\calXBK}$:
\begin{equation}\label{est_2_GFEIII_BK}
	\begin{aligned}
		\|\ppsi\|_{\calXBK}
		\lesssim_T&\,  \begin{multlined}[t]\|\ppsi_0\|^2_{H^3}+
			\|  \ppsi_1\|^2_{H^3}+ \taua \| \ppsi_2\|^2_{H^2}+
			\|f\|^2_{W^{1,1}(H^1)},\end{multlined}
	\end{aligned}
\end{equation}
where the hidden constant has the form
\begin{equation}
	C= C_1 \exp \left((1+ \|\D\bbb_t\|_{L^2(L^2)}+\|\D\bbb\|_{L^\infty(L^2)})C_2T\right).
\end{equation}
  This bound for the linearized problem forms the foundation for the well-posedness and limiting analysis in the Kuznetsov--Blackstock case.  
\subsection{Uniform well-posedness with Kuznetsov--Blackstock-type nonlinearities} \label{Sec:Testing_psi_tt}
The next step is as before to set up a fixed-point mapping $\calT: \calBBK \ni \ppsi^* \mapsto \ppsi$, 
where $\ppsi$ solves the linear equation in \eqref{IBVP_Westtype_lin} with
\begin{equation}
\begin{aligned}
\aaa(\ppsi_t^*) =\,1+2k_1 \ppsi^*_t, \quad \bbb(\ppsi_t^*) = 1-2k_2 \ppsi_t^*, \
f=\, -\calN(\nabla \ppsi^*, \ppsi^*_t)= -2 k_3 \nabla \ppsi^*\cdot \nabla \ppsi_t^*,
\end{aligned}
\end{equation}
and $\ppsi^*$ is taken from the ball
\begin{equation} \label{def_BBK}
\begin{aligned}
\calBBK =\left \{ \ppsi^* \in \calXBK:\right.&\, \, \|\ppsi^*\|_{\calXBK} \leq R,\ 
\left	(\ppsi^*, \ppsit^*, \ppsitt^*)\vert_{t=0} =(\ppsi_0, \ppsi_1, \ppsi_2) \}. \right.
\end{aligned}
\end{equation}
\begin{theorem} \label{Thm:WellP_GFEIII_BKuznetsov}
	Let $T>0$ and $\tau \in (0, \bar{\tau}]$. Let assumptions \eqref{reg_kernels} and \eqref{Athree_GFEIII}--\eqref{Afive_GFEIII} on the kernels $\frakKone$ and $\frakKtwo$ hold. Let 
		\begin{equation}
	\begin{aligned}
	(\ppsi_0, \ppsi_1, \ppsi_2) \in \Honethree \times \Honethree \times \Honetwo.
	\end{aligned}
	\end{equation}
	There exists $r=r(T)>0$, independent of $\tau$, such that if
	\begin{equation}
	\|\ppsi_0\|^2_{H^3}+\|\ppsi_1\|^2_{H^3} + \bar{\tau}^a \|\ppsi_2\|^2_{H^2} \leq {r}^2,
	\end{equation}
	then there is a unique solution $
	\ppsi \in \calBBK$, such that $\taua \frakKone*\ppsittt \in L^2(0,T; L^2(\Omega))$,
	 of 
	\begin{equation} \label{IBVP_GFEIII_BlackstockKuznetsov}
	\begin{aligned}
		\taua  \frakKone * \ppsittt+(1+2k_1 \ppsi_t) \ppsitt
	-c^2 (1-2k_2 \ppsi_t) \Delta \ppsi 
	-   \taua c^2 \frakKone* \Delta \ppsit -\ttbdelta  \frakKtwo* \Delta \ppsitt \\+ 2 k_3 \nabla \ppsi \cdot \nabla \ppsi_t=0
	\end{aligned} 
	\end{equation}
with initial \eqref{IC} and boundary \eqref{BC} data. The solution satisfies
\begin{equation}
	\|\ppsi\|^2_{\calXBK}
\lesssim_T\,  \|\ppsi_0\|^2_{H^3}+
	\|  \ppsi_1\|^2_{H^3}+ \taua \| \ppsi_2\|^2_{H^2},
\end{equation}
where the hidden constant does not depend on $\tau$.
\end{theorem}
\begin{proof}
{The proof follows by employing the Banach fixed-point theorem on $\calT$, analogously to the proof of Theorem~\ref{Thm:WellP_GFEIII_BWest}. We provide the details in Appendix~\ref{Appendix:WellPproofKB}.}
\end{proof}
{This uniform well-posedness result generalizes~\cite[Theorem 6.1]{kaltenbacher2019jordan}, where the \ref{fJMGTIII} equation with Kuznetsov nonlinearities is considered, to equation \eqref{IBVP_GFEIII_BlackstockKuznetsov} with Kuznetsov--Blackstock nonlinearities and kernels satisfying the assumptions of this section.}
\subsection{Weak singular limit with Kuznetsov--Blackstock type nonlinearities} 
We next discuss the limiting behavior of the equations with Kuznetsov--Blackstock nonlinearities. Let $\tau \in (0, \bar{\tau}]$. Under the assumptions of Theorem~\ref{Thm:WellP_GFEIII_BKuznetsov}, with
	\begin{equation}
	\|\ppsi^\tau_0\|^2_{H^3}+\|\ppsi^\tau_1\|^2_{H^3} + \bar{\tau}^a \|\ppsi^\tau_2\|^2_{H^2} \leq r^2,
\end{equation}
consider 
\begin{equation} \label{IBVP_tau_GFEIII_BK}
	\left \{	\begin{aligned}
&\begin{multlined}[t]	\taua  \frakKone * \ppsittt^\tau+(1+2k_1 \ppsi^\tau_t) \ppsitt^\tau
-c^2 (1-2k_2 \ppsi^\tau_t) \Delta \ppsi^\tau
-   \taua c^2 \frakKone* \Delta \ppsit^\tau \\-\ttbdelta \frakKtwo* \Delta \ppsitt^\tau + 2 k_3 \nabla \ppsi^\tau \cdot \nabla \ppsi^\tau_t=0,\end{multlined}\\
&\ppsi^\tau\vert_{\partial \Omega}=0, \\
&	(\ppsi^\tau, \ppsit^\tau, \ppsitt^\tau)\vert_{t=0} =(\ppsi^\tau_0, \ppsi^\tau_1, \ppsi^\tau_2) .
\end{aligned} \right. 
\end{equation}
From the previous analysis we knowc that a unique solution of this problem exists  in $ \calBBK$. Therefore, there exists a subsequence, not relabeled, such that 
\begin{equation} \label{weak_limits_1_tau_GFEIII_KB}
\begin{alignedat}{4} 
\ppsi^\tau&\relbar\joinrel\rightharpoonup \ppsi && \quad \text{ weakly-$\star$}  &&\text{ in } &&L^\infty(0,T; \Honethree),  \\
\ppsit^\tau &\relbar\joinrel\rightharpoonup \ppsit && \quad \text{ weakly-$\star$}  &&\text{ in } &&L^\infty(0,T; \Honethree),\\
\ppsitt^\tau &\relbar\joinrel\rightharpoonup \ppsitt && \quad \text{ weakly}  &&\text{ in } &&L^2(0,T; \Honetwo).
\end{alignedat} 
\end{equation} 
Similarly to before, using the Aubin--Lions--Simon lemma, this further implies
\begin{equation} \label{weak_limits_2_tau_GFEIII_KB}
\begin{alignedat}{4} 
\ppsi^\tau&\longrightarrow \ppsi && \quad \text{ strongly}  &&\text{ in } &&C([0,T]; \Honetwo),  \\
\ppsit^\tau &\longrightarrow \ppsit && \quad \text{ strongly}  &&\text{ in } &&C([0,T]; \Honetwo),
\end{alignedat} 
\end{equation} 
and therefore,
\begin{equation} \label{limits_2_tau_initial_GFEIII_BK}
	\begin{alignedat}{4} 
		\ppsi_0^\tau=\ppsi^\tau(0)&\longrightarrow \ppsi(0):=\ppsi_0 && \quad \text{ strongly}  &&\text{ in } && \Honetwo,  \\
		\ppsi_1^\tau=\ppsit^\tau(0) &\longrightarrow \ppsit(0):= \ppsi_1&& \quad \text{ strongly}  &&\text{ in } && \Honetwo.
	\end{alignedat} 
\end{equation} 
It remains to prove that $u$ solves the limiting problem. Let $v \in C_0^\infty([0,T]; C_0^\infty(\Omega))$. We have 
\begin{equation}
\begin{aligned}
&\begin{multlined}[t] \intTO \aaa(\ppsi_t) \ppsitt v \dxs- c^2 \intTO \bbb(\ppsi_t) \Delta \ppsi v \dxs +\rt^b \delta \intTO  \frakKtwo* \nabla \ppsitt \cdot \nabla v \dxs\\
+\intTO \calN(\nabla \ppsi, \nabla \ppsit) v \dxs = \textup{rhs},\end{multlined}
\end{aligned}
\end{equation}
where, with $\bar{\ppsi}=\ppsi-\ppsi^\tau$, the right-hand side is
\begin{equation}
	\begin{aligned}
		\textup{rhs}:
		=&	\begin{multlined}[t]  \intTO  \aaa(\ppsit)\bar{\ppsi}_{tt} v \dxs- c^2 \intTO \bbb(\ppsit) \Delta \bar{\ppsi} v \dxs+\rt^b \delta \intTO  \frakKtwo* \nabla \bar{\ppsi}_{tt} \cdot \nabla v \dxs \\
			- \intTO  \taua \frakKone * \ppsittt^\tau v \dxs + \taua c^2 \intTO \frakKone* \Delta \ppsit^\tau v \dxs \\
			-\intTO (\aaa(\ppsit^\tau)-\aaa(\ppsit)) \ppsitt^\tau v \dxs + c^2\intTO (\bbb(\ppsit^\tau)-\bbb(\ppsit))\Delta \ppsi^\tau v \dxs-\\ -\intTO (\calN(\nabla \ppsi^\tau,  \nabla \ppsit^\tau) -\calN(\nabla \ppsi,  \nabla \ppsit))v \dxs.  \end{multlined}
	\end{aligned}
\end{equation}
We should prove that $\textup{rhs}$ tends to zero as $\tau \searrow 0$. We only comment here on how to tackle the $\aaa$, $\bbb$ and $\calN$ terms; the other terms can be treated as in Section~\ref{Subsec:Weaklim_tau_GFEIII}.
By relying on the equivalence of the norms $\|{\aaa(\ppsit) v}\|_{L^2}$, $\|v\|_{L^2}$, and $\|\bbb(\ppsi_t) v\|_{L^2}$ (under the assumptions of Theorem~\ref{Thm:WellP_GFEIII_BKuznetsov}), we can conclude that 
\begin{equation}
\begin{aligned}
\intTO  \aaa(\ppsit)\bar{\ppsi}_{tt} v \dxs+ c^2 \intTO \bbb(\ppsit) \Delta \bar{\ppsi} v \dxs  \rightarrow  0 \quad \text{as } \tau \searrow 0.
\end{aligned}
\end{equation}
We also have
\begin{equation}
\begin{aligned}
&-\intTO (\aaa(\ppsit^\tau)-\aaa(\ppsit)) \utt^\tau v \dxs + c^2\intTO (\bbb(\ppsit^\tau)-\bbb(\ppsit))\Delta \ppsi^\tau v \dxs \\
=& \, -2k_1\intTO \bar{\ppsi}_t \ppsitt^\tau v \dxs - 2 k_2 c^2 \intTO \bar{\ppsi}_t \Delta \ppsi^\tau v \dxs,	
\end{aligned}
\end{equation}
which tends to zero as well thanks to \eqref{weak_limits_1_tau_GFEIII_KB} and \eqref{weak_limits_2_tau_GFEIII_KB}. Finally,
\begin{equation}
\begin{aligned}
\intTO (\calN(\nabla \ppsi^\tau, \ppsit^\tau) -\calN(\nabla \ppsi,  \nabla \ppsit))v \dxs 
=\, 2k_3 \intTO ( \nabla \bar{\ppsi} \cdot \nabla \ppsi_t^\tau + \nabla \ppsi \cdot \nabla \bar{\ppsi}_t )  v \dxs, 
\end{aligned}
\end{equation}
which tends to zero on account of again \eqref{weak_limits_2_tau_GFEIII_KB} and the uniform bounds obtained in Theorem~\ref{Thm:WellP_GFEIII_BKuznetsov}. 
 Uniqueness of solutions for the limiting problem follows by testing the equation solved by the difference $\bar{\ppsi}$ of two solutions by $\bar{\ppsi}_{tt}$, similarly to the proof of contractivity in Theorem~\ref{Thm:WellP_GFEIII_BKuznetsov} given in Appendix~\ref{Appendix:WellPproofKB}. This allows us to employ a subsequence-subsequence argument on $\{\ppsi^\tau\}_{\tau \in (0, \bar{\tau}]}$. Altogether, we arrive at the following result.
\begin{theorem} \label{Thm:WeakLim_GFEIII_BKuznetsov}
	Let the assumptions of Theorem~\ref{Thm:WellP_GFEIII_BKuznetsov} hold for problem \eqref{IBVP_tau_GFEIII_BK}. Then the family $\{\ppsi^\tau\}_{\tau \in (0, \bar{\tau}]}$ of solutions to
	\eqref{IBVP_tau_GFEIII_BK}
	converges in the sense of \eqref{weak_limits_1_tau_GFEIII_KB}, \eqref{limits_2_tau_initial_GFEIII_BK} as $\tau \searrow 0$  to the solution $\ppsi \in \XWest$ of 
	\begin{equation} \label{IBVP_limit_GFEIII_BK}
	\left \{	\begin{aligned}
	&\begin{multlined}[t](1+2k_1 \ppsit) \ppsitt
	-c^2 (1-2k_2 \ppsit) \Delta \ppsi
 -\ttbdelta \frakKtwo* \Delta \ppsitt\\ \hspace*{6cm}+ 2 k_3 \nabla \ppsi \cdot \nabla \ppsit=0 \ \textup{ in } \Omega \times (0,T), \end{multlined}\\
	&\ppsi\vert_{\partial \Omega}=0, \\
	&	(\ppsi, \ppsit)\vert_{t=0} =(\ppsi_0, \ppsi_1).
	\end{aligned} \right. 
	\end{equation}
\end{theorem}
As a by-product of the previous analysis, we obtain unique solvability of \eqref{IBVP_limit_GFEIII_BK} for small data in $\Honethree \times \Honethree$.
\newcommand\pp{u} 
\newcommand\kthree{k_3}
\section{Testing with $\ppsi_{tt}$ and $-\D\frakKone\Lconv\ppsi_{tt}$ } \label{Sec:GFEI}
We have seen in Section~\ref{Sec:GFEIII} that among the Compte--Metzler laws, testing with $(-\Delta)^\nu\ppsi_{tt}$ leads to a $\tau$-uniform bound only in the \ref{fJMGTIII} case. Therefore in this section, we will investigate an alternative way of testing by somewhat weakening the time derivative in the multiplier: instead of $-\Delta\ppsi_{tt}$ we will use $-\frakKone\Lconv \Delta\ppsi_{tt}$. Complementary to this, we will also test with $\ppsi_{tt}$. \\
\indent This testing strategy turns out to be applicable to all three Compte--Metzler laws of interest, with caveats. The price to pay is that we have to restrict ourselves to  Westervelt--Blackstock nonlinearities, 
 unless dealing with the \ref{fJMGTI} equation with $\alpha\leq1/2$ (Case I below). 
 In this case we can prove existence of solutions but not uniqueness. \\
 \indent Besides, {when $\frakKone \in L^1(0,T)$}, for the analysis in this section to go through we need to rewrite the leading term in \eqref{general_wave_equation} and consider
 \begin{equation} \label{general_wave_equation_rewritten}
 \begin{aligned}
 	\begin{multlined}[t]
 		\taua (\frakKone *\ppsitt)_t+\aaa \ppsitt
 		-c^2 \bbb\Delta \ppsi - \taua c^2  \frakKone*\Delta\ppsit
 		- \ttbdelta   \frakKtwo*\Delta \ppsitt  +\calN= \taua \frakKone(t) \utt(0),
 	\end{multlined}
 \end{aligned}
\end{equation}
which will force us to assume $u_2=0$ (and thus have that the right-hand side above is zero). \\[2mm]
\noindent \emph{\bf Assumptions on the two kernels in this section.} 
In this section, we assume that 
\begin{equation}\label{Athree_general} 
	\text{there exists} \  \tfrakKone \in L^1(0,T), \text{ such that }\frakKone * \tfrakKone =1.   \tag{$\mathcal{H}_1$}
\end{equation}
{Note that unlike in Section~\ref{Sec:GFEIII}, here the resolvent is only required to be $L^1$ regular}. Furthermore, we need the following coercivity property of the kernel in the leading term:
 \begin{equation}\label{eqn:Alikhanov_1_testK1}
 	\int_0^t (\frakKone\Lconv y)'(s) \,y(s)\ds \geq - C_{\frakKone} |y(0)|^2, \quad y\in   \tag{${\bf \mathcal{H}}_2$}
X^2_{\frakKone}(0,t), 
 \end{equation}
for some (possibly large, but possibly also vanishing) constant $C_{\frakKone}$. We also assume that
\begin{equation}\label{eqn:coercivity_K1}
 	\int_0^t (\frakKone\Lconv y)(s) \,y(s)\ds \geq \Psi[y](t), \quad y\in L^2(0,t),   \tag{${\bf \mathcal{H}}_3$}
 \end{equation}
for some functional $\Psi$ to be specified more precisely later on. 
In particular, we will assume the energy contribution due to $\Psi$ to be quantified by an estimate in some Sobolev norm 
\begin{equation}\label{eqn:Psi_Sobolevnorms}   \tag{${\bf \mathcal{H}}_3 \ \Psi$}
\int_0^t\Psi[\nabla\ppsi_{tt}](s)\ds \sim \|\nabla\ppsi\|_{W^{\sigma,\rho}(0,t;L^2(\Omega))}^2
\end{equation}
for some $\sigma\in[0,2]$ and $\rho\in[2,\infty]$; see Table~\ref{table:GFElaws_pq_etc} for the Compte--Metzler laws.\\
\indent Furthermore,  we impose a coercivity condition on $\frakKtwo$: 
\begin{equation}\label{eqn:coercivity_K2}
	\int_0^t (\frakKtwo\Lconv y)(s)y(s)\ds \geq 0
, \quad y\in L^2(0,t) .  \tag{${\bf \mathcal{H}}_4$}
\end{equation}
\indent How the two kernels behave relative to each other will also be important. More precisely, besides the above assumptions, the key coercivity property that we assume in this section is the following:
\begin{equation}\label{eqn:coercivity_2_testK1}
	\int_0^t (\frakKtwo\Lconv y)(s)(\frakKone\Lconv y)(s)\ds 
\geq \Phi[y](t)
, \quad y\in L^2(0,t),  \tag{${\mathcal{H}}_5$}
\end{equation}
where $\Phi$ is a nonnegative functional satisfying 
\begin{equation}\label{eqn:coercivity_2_bothCases}
\sup_{t'\in(0,t)}\Phi[y](t') \geq
\underline{c}_2\max\Bigl\{\|\frakKone\Lconv 1\Lconv y\|_{L^2(0,t)}^2 ,\,\|1\Lconv 1\Lconv y\|_{L^2(0,t)}^2\Bigr\}
, \quad y\in L^2(0,t)   
\end{equation}
for some (possibly small) $\underline{c}_2>0$. Later on, we will assume the energy contribution due to $\Phi$ to be quantified by an estimate in some Sobolev norm 
\begin{equation}\label{eqn:PhiPsi_Sobolevnorms}  \tag{${\bf \mathcal{H}}_5 \ \Phi$}
\sup_{t'\in(0,t)}\Phi[\Delta\ppsi_{tt}](t')\sim \|\Delta\ppsi\|_{W^{s,\mathtt{r}}(0,t;L^2(\Omega))}^2,
\end{equation}
for some $s\in[0,2]$ and $\mathtt{r}\in[2,\infty]$; see Table~\ref{table:GFElaws_pq_etc}. Concerning the further behavior of $\Phi$, we distinguish between two cases. 

\subsubsection*{\textbf{Case I}: ``Kernel $\frakKtwo$ is more singular than $\frakKone$''} Here we assume that
\begin{equation}\label{eqn:coercivity_2_CaseI}  \tag{${\bf \mathcal{H}}_5 \ \textup{I}$}
\sup_{t'\in(0,t)}\Phi[y](t') \geq
\max\Bigl\{
\underline{c}_3 \|\frakKone\Lconv y\|_{L^2(0,t)}^2 ,\,
\overline{C}_3 \|1\Lconv1\Lconv y\|_{L^2(0,t)}^2\Bigr\}
, \quad y\in L^2(0,t)   
\end{equation}
for some positive constants $\underline{c}_3$ and $\overline{C}_3$ (the latter being assumed sufficiently large). 

\subsubsection*{\textbf{Case II}: ``Kernel $\frakKtwo$ 
is 
less singular than $\frakKone$''}  In this case we assume that
\begin{equation}\label{eqn:coercivity_2_CaseII} \tag{${\bf \mathcal{H}}_5 \ \textup{II}$}
\begin{aligned}
&\sup_{t'\in(0,t)}\Phi[y](t') \\
\geq&\,\overline{C}_3\max\Bigl\{ 
\|1\Lconv y\|_{L^p(0,t)}^2, \, 
\|1\Lconv 1\Lconv y\|_{L^\infty(0,t)}^2, \, 
\|\frakKone\Lconv 1\Lconv y\|_{L^\infty(0,t)}^2, 
\Bigr\},
\quad y\in L^2(0,t)   
\end{aligned}
\end{equation}
for some $p\in[1,\infty]$ and $\overline{C}_3>0$ sufficiently large.  Further,
\begin{equation}\label{eqn:cond_ab_CaseII}   \tag{${\bf \mathcal{H}}_5 \ \textup{II} \ \aaa, \bbb$}
\nabla\aaa=0, \quad \aaa_t=0, \quad
\bbb_t\in L^{\tilde{q}}(0,T;L^\infty(\Omega)),  
\end{equation}
where
\begin{equation}\label{eqn:cond_pq_CaseII}  \tag{${\bf \mathcal{H}}_5 \ \textup{II} \ \frakKone$}
\begin{aligned}
&\text{either }\frakKone=\delta_0\mbox{ and }\tilde{q}= p'=\frac{p}{p-1}\\
&\text{ or } \quad
\frakKone\in L^q(0,T), \quad q\leq p', \quad \tilde{q}=\frac{p'q}{p'q+q-p'}=\frac{pq}{2pq-q-p}
\end{aligned}
\end{equation}
with $p$ as in \eqref{eqn:coercivity_2_CaseII}. \\[1mm]
\noindent \textbf{Assumptions on the variable coefficients.}  We will need  the following smoothness and non-degeneracy assumptions:
\begin{equation}\label{eqn:cond_ab}
\aaa,\, \frac{1}{\aaa}, \, \bbb \, \in L^\infty(0,T;L^\infty(\Omega)), \quad
\|\aaa-\aaa_0\|_{L^\infty(L^\infty)}\leq \overline{c}_4
\end{equation}
for some constants $\aaa_0>0$ and $\overline{c}_4$, where the latter will be assumed to be small enough (independently of $\tau$); cf.\ \eqref{eqn:cond_a-a0_CaseI} for Case I. In Case II we will even need $\aaa-\aaa_0=0$.
\subsection{How to verify assumptions on the kernels}\label{sec:Howto_K1psitt}
As noted before, assumption \eqref{eqn:Alikhanov_1_testK1} is satisfied for all Compte--Metzler laws
considered due to \cite[Lemma B.1]{kaltenbacher2021determining}. Conditions \eqref{eqn:coercivity_K1} and \eqref{eqn:coercivity_K2} hold for the three laws considered in this work due to \cite[Lemma 2.3]{Eggermont1987} with the choice
\[
\frakKone=g_{\alpha_1}=\frac{1}{\Gamma(\alpha_1)}t^{\alpha_1-1}, \qquad \Psi[y]=\cos(\alpha_1\pi/2)\|y\|_{H^{-\alpha_1/2}(0,t)}, \qquad \frakKtwo=g_{\alpha_2}
\]  
with $\alpha_1$, $\alpha_2\, \in[0,1]$.  Also assumption \eqref{eqn:coercivity_2_testK1} which relates the two kernels is satisfied for all three Compte--Metzler laws of interest.
More precisely, we have
\[
\Phi[y] = 
\tilde{c}(\alpha_1,\alpha_2)\,  \cos((\alpha_2-\alpha_1)\pi/2) \|y\|_{H^{-(\alpha_2+\alpha_1)/2}(0,T)}.
\]

Condition \eqref{eqn:coercivity_2_CaseI} means that the damping by $\frakKtwo$ is at least as strong (in the sense of higher order of differentiation) as the one by $\frakKone$. However, among the Compte--Metzler laws it is only satisfied for the GFE I kernel with $\alpha\leq\frac12$.  

Condition \eqref{eqn:coercivity_2_CaseII} is satisfied for the GFE I kernel with $\alpha\geq\frac12$, as well as for the GFE III and GFE kernels. Largeness of the constant $\overline{C}_3$ in assumptions \eqref{eqn:coercivity_2_CaseI} and \eqref{eqn:coercivity_2_CaseII} can be achieved by using H\"older's inequality and making $T$ small enough.

Assumption \eqref{eqn:cond_pq_CaseII} holds for the GFE III kernel and, with any $q\in[1,\frac{1}{1-\alpha})$ for the GFE I and GFE kernels. 

The three Compte--Metzler laws and relevant further information are summarized in Table~\ref{table:GFElaws_pq_etc}, where we have also listed $s$, $\mathtt{r}$, $\sigma$, and $\rho$, such that \eqref{eqn:PhiPsi_Sobolevnorms} holds.

	\begin{table}[h]	
	\captionsetup{width=13cm}
	\centering
	\begin{adjustbox}{max width=\textwidth}
		\begin{tabular}{|c||c|c|c|c|c|c|c|c|c|c||}
			\hline
			Flux law  & $\alpha_1$ & $\alpha_2$  &$p$  &$p'$  &$q$ &$\tilde{q}$ &$s$ &$\mathtt{r}$ &$\sigma$ &$\rho$	\\
			\hline\hline
			GFE I\rule{0pt}{3ex} &${1-\alpha}$ &${\alpha}$ &
$<\infty$ &$1+\tilde{\iota}$ &$\frac{1}{\alpha}-\tilde{\tilde{\iota}}$ & $1$
&$\frac32$ &$2$ &$\frac{3+\alpha}{2}$ &$2$  \\
			GFE III\rule{0pt}{3ex} &$0$			&${\alpha}$	&
$\infty$ &$1$ &$1$ &$1$ &$2-\frac{\alpha}{2}$ &$2$ &$2$ &$2$ \\
			GFE \rule{0pt}{3ex} & ${1-\alpha}$	&$1$		&
$\frac{2}{1-\alpha}$ &$\frac{2}{1+\alpha}$ &$\frac{1}{\alpha}-\tilde{\tilde{\iota}}$ &$1$
&$1+\frac{\alpha}{2}$ &$2$ &$\frac{3+\alpha}{2}$ &$2$ \\
			\hline
		\end{tabular}
	\end{adjustbox}
	~\\[2mm]
	\caption{\small Kernels $\frakKone(t)=g_{\alpha_1}(t)$ and $\frakKtwo=g_{\alpha_2}$ for different examples of generalized flux laws; values of relevant quantities in the assumptions on kernels; here $\tilde{\iota},\, \tilde{\tilde{\iota}}>0$ can be arbitrarily small; in particular, they can be chosen such that $q=p'$ and thus $\tilde{q}=1$ in all cases} \label{table:GFElaws_pq_etc}
\end{table}

\subsection{Well-posedness of the linearized problem}

Again, we first present the main ideas of establishing well-posedness of the linearized equation
\begin{equation} \label{pde_testingK1_lin}
	\begin{aligned}
		\begin{multlined}[t]
			\taua (\frakKone\Lconv \ppsi_{tt})_t+\aaa\ppsi_{tt}
			-c^2\bbb\Delta \ppsi 
			- \taua c^2 \frakKone\Lconv \Delta \ppsi_t	-	\ttbdelta  \frakKtwo\Lconv \Delta \ppsi_{tt}=f,
		\end{multlined}
	\end{aligned}
\end{equation}
in particular, the energy estimates that will be needed for this purpose in each of the above-mentioned Cases I and II.

\smallskip

\noindent \emph{Testing with $\ppsitt$.} We start with the low-order energy estimate that is the same for both cases, obtained by testing with $\ppsi_{tt}$.
To this end, we interpret equation \eqref{pde_testingK1_lin} as 
\begin{equation}\label{eqn:pde-step2}			
\begin{aligned}
& \taua(\frakKone\Lconv \ppsi_{tt})_t
+\aaa\ppsi_{tt} 
- \ttbdelta  \frakKtwo\Lconv \Delta \ppsi_{tt}
= \bar{r}
\end{aligned}
\end{equation}
with the right-hand side
\[
\bar{r}=f+c^2\bbb\Delta \ppsi+ {\tau^a} c^2\frakKone\Lconv \Delta \ppsi_t
\]
and test it with $\ppsi_{tt}$.
By assumptions \eqref{eqn:Alikhanov_1_testK1} and \eqref{eqn:coercivity_K2}, we obtain
\[
\begin{aligned}
\int_0^t|\sqrt{\aaa}\ppsi_{tt}|_{L^2}^2\ds
\leq \left\|\frac{1}{\sqrt{\aaa}}r \right\|_{L^2_t(L^2)} \|\sqrt{\aaa}\ppsi_{tt}\|_{L^2_t(L^2)}
+C_{\frakKone}\|\ppsi_{tt}(0)\|_{L^2}^2.
\end{aligned}
\]
The $r$ term can be estimated as follows:
\[
\left\|\frac{1}{\sqrt{\aaa}}\bar{r} \right\|_{L^2_t(L^2)} \leq \left\|\frac{1}{\sqrt{\aaa}} \right\|_{L^\infty(L^\infty)}    \|\bar{r}\|_{L^2_t(L^2)} 
\]
with
\begin{equation}\label{eqn:est_r2}
\begin{aligned}
\|\bar{r}\|_{L^2_t(L^2)} 
\leq&\, \|f\|_{L^2_t(L^2)}
+c^2 \|\bbb\|_{L^\infty(L^\infty)} \|\Delta \ppsi\|_{L^2_t(L^2)}
+{\tau^a} c^2\|\frakKone\Lconv \Delta \ppsi_t\|_{L^2_t(L^2)}\\
\leq&\, \begin{multlined}[t] \|f\|_{L^2_t(L^2)} 
 +c^2 \|\bbb\|_{L^\infty(L^\infty)} (t^{1/2}\|\Delta \ppsi(0)\|_{L^2}+ t^{3/2} \|\Delta \ppsi_t(0)\|_{L^2}) \\+ c^2\tau^a t^{1/2} \|\frakKone\|_{\mathcal{M}(0,t)} \|\Delta \ppsi_t(0)\|_{L^2}
+c^2 (\|\bbb\|_{L^\infty(L^\infty)}+ {\tau^a})/\sqrt{\underline{c}_2}\ \sqrt{\Phi[\Delta\ppsi_{tt}]}, \end{multlined}
\end{aligned}
\end{equation}
where we have relied on \eqref{eqn:coercivity_2_bothCases}. Altogether,
\begin{equation}\label{eqn:est-step2}
\begin{aligned}
\|\ppsi_{tt}\|_{L^2_t(L^2)}^2
\lesssim& \, \begin{multlined}[t]
\|f\|_{L^2(L^2)}^2 + \Phi[\Delta\ppsi_{tt}]
+C_{\frakKone}\|\ppsi_{tt}(0)\|_{L^2}^2
 + t\|\Delta \ppsi(0)\|_{L^2}^2\\+ t^3 \|\Delta \ppsi_t(0)\|_{L^2}^2. \end{multlined}
\end{aligned}
\end{equation}
We keep track of $C_{\frakKone}$ here because it comes with high-order initial data but is nonzero only in the exceptional case $\frakKone=\delta_0$. Later on we will see that for a different reason we will anyway have to assume $\ppsi_{tt}(0)=0$.
The hidden constant in \eqref{eqn:est-step2} is independent of both $T$ and $\tau$.

\smallskip

\noindent \emph{Testing with $- \frakKone*\Delta \ppsitt$.} The higher-order estimate is obtained by rearranging equation \eqref{pde_testingK1_lin} as 
\begin{equation}\label{eqn:pde-step1}			
\begin{aligned}
&{\tau^a}(\frakKone\Lconv \ppsi_{tt})_t
+\aaa_0\ppsi_{tt} - {\tau^a} c^2\frakKone\Lconv \Delta \ppsi_t
- \ttbdelta  \frakKtwo\Lconv \Delta \ppsi_{tt}
= \bar{r}
\end{aligned}
\end{equation}
for $\aaa_0>0$ with now
\[
\bar{r}=f-(\aaa-\aaa_0)\ppsi_{tt}+c^2\bbb\Delta \ppsi
\]
and testing it with \[-\frakKone\Lconv \Delta \ppsi_{tt}=-(\frakKone\Lconv \Delta \ppsi_t)_t+\frakKone\cdot\Delta \ppsi_t(0).\] Using assumption \eqref{eqn:coercivity_K1} with $y=\nabla \ppsi_{tt}(x)$ as well as \eqref{eqn:coercivity_2_testK1} with $y=-\Delta \ppsi_{tt}(x)$ yields
\begin{equation}\label{eqn:pde_step1}
\begin{aligned}
&\frac12 \tau^a\Bigl[\|\frakKone\Lconv \nabla\ppsi_{tt}\|_{L^2}^2+c^2\|\frakKone\Lconv \Delta\ppsi_t\|_{L^2}^2\Bigr]_0^t 
+ \aaa_0 \int_0^t \Psi[\nabla\ppsi_{tt}](s)\ds
+ \ttbdelta \Phi[\Delta \ppsi_{tt}](t)\\
\leq&\, 
\text{rhs} + \tau^a c^2 \int_0^t\int_\Omega \frakKone\,\Delta\ppsi_t(0)\,\frakKone\Lconv \Delta \ppsi_t\dxs, 
\end{aligned}
\end{equation}
where
\begin{align} \label{rhs}
 \text{rhs}:= \int_0^t\int_\Omega (f-(\aaa-\aaa_0)\ppsi_{tt}+c^2\bbb\Delta\ppsi)(-\frakKone\Lconv \Delta \ppsi_{tt})\dxs.
\end{align}
We can estimate the second term on the right in \eqref{eqn:pde-step1} as follows: 
\[
\begin{aligned}
\Bigl|\int_0^t\int_\Omega \frakKone\,\Delta\ppsi_t(0)\,\frakKone\Lconv \Delta \ppsi_t\dxs\Bigr|
\leq \|\frakKone\|_{{\mathcal{M}}(0,t)} \|\Delta\ppsi_t(0)\|_{L^2} \|\frakKone\Lconv \Delta \ppsi_{t}\|_{L^\infty_t(L^2)}.
\end{aligned}
\]
We distinguish between two cases in the further treatment of the right-hand side.

\medskip
 
\noindent
\textbf{Case I}:
Due to \eqref{eqn:coercivity_2_CaseI}, assumptions \eqref{eqn:cond_ab} on the variable coefficients, and Young's inequality, as well as the elementary estimate
\[
\|\frakKone\Lconv\Delta\ppsi_t \|_{L^\infty_t(L^2)}\leq
\| (\frakKone\Lconv\Delta\ppsi_t)(0)\|_{L^2}+ \sqrt{t}\|(\frakKone\Lconv \Delta \ppsi_t)_t\|_{L^2_t(L^2)},
\]
all terms containing $\ppsi$ on the right-hand side of \eqref{eqn:pde-step1} can be estimated by means of $\Phi(\Delta\ppsi_{tt})$ and the already obtained $L^2(0,t; L^2(\Omega))$ estimate of $\ppsi_{tt}$. Indeed, 
\begin{equation}\label{eqn:est_step1}
\begin{aligned}
\text{rhs}\leq& \, \begin{multlined}[t] \epsilon \|\frakKone\Lconv \Delta \ppsi_{tt}\|_{L^2_t(L^2)}^2
+\frac{1}{\epsilon} \|f\|_{L^2_t(L^2)}^2
+\frac{1}{\epsilon} \|\aaa-\aaa_0\|_{L^\infty_t(L^\infty)}^2 \|\ppsi_{tt}\|_{L^2_t(L^2)}^2\\
+\frac{1}{\epsilon} c^2\|\bbb\|_{L^\infty_t(L^\infty)}^2 \|\Delta\ppsi \|_{L^2_t(L^2)}^2
+\mu \|\frakKone\Lconv\Delta\ppsi_t \|_{L^\infty_t(L^2)}^2 \end{multlined} \\ 
\leq&\, \begin{multlined}[t] \Bigl(\frac{\epsilon +2t\mu}{\underline{c}_3}+ \frac{c^2\|\bbb\|_{L^\infty(L^\infty)}^2}{\epsilon\overline{C}_3}\Bigr) \Phi(\Delta\ppsi_{tt})
+ \frac{1}{\epsilon} \|f\|_{L^2_t(L^2)}^2 
+\frac{1}{\epsilon} \|\aaa-\aaa_0\|_{L^\infty_t(L^\infty)}^2 \|\ppsi_{tt}\|_{L^2_t(L^2)}^2\\
+ 2\mu \|(\frakKone\Lconv\Delta\ppsi_t)(0)\|_{L^2}^2. \end{multlined}
\end{aligned}
\end{equation}
By first choosing $\epsilon>0$ sufficiently small and then assuming $\overline{C}_3$ to be sufficiently large {(which might necessitate a decrease of $T$, see the comment on largeness of $\overline{C}_3$ in Subsection~\ref{sec:Howto_K1psitt})}, we can achieve 
\[\frac{\epsilon +2t\mu}{\underline{c}_3}+ \frac{ c^2\|\bbb\|_{L^\infty(L^\infty)}^2}{\epsilon\overline{C}_3} \leq \frac{\ttbdelta}{2}\] and therefore end up with 
\begin{equation}
\begin{aligned}
&\frac12 \tau^a\Bigl[\|\frakKone\Lconv \nabla\ppsi_{tt}\|_{L^2}^2+c^2\|\frakKone\Lconv \Delta\ppsi_t\|_{L^2}^2\Bigr]_0^t 
+ \aaa_0 \int_0^t \Psi[\nabla\ppsi_{tt}](s)\ds
+ \frac{\ttbdelta}{2} \Phi[\Delta \ppsi_{tt}](t)\\
\leq&\, \begin{multlined}[t] C(t) \left(
\|f\|_{L^2(L^2)}^2 
+\|\aaa-\aaa_0\|_{L^\infty_t(L^\infty)}^2 \|\ppsi_{tt}\|_{L^2(L^2)}^2
+\|(\frakKone\Lconv\Delta\ppsi_t)(0)\|_{L^2}^2 \right. \\ \left.
+\tau^{2a}\|\frakKone\|_{{\mathcal{M}}(0,T)}^2 \|\Delta\ppsi_t(0)\|_{L^2}^2 \right), \end{multlined}
\end{aligned}
\end{equation}
where $C(t)>0$ might depend on $t$ but not on $\tau$.
Combining this bound with estimate \eqref{eqn:est-step2} where $C>0$ is the hidden constant and assuming 
\begin{equation}\label{eqn:cond_a-a0_CaseI}
C\,C(t)\|\aaa-\aaa_0\|_{L^\infty_t(L^\infty)}^2<\frac{\ttbdelta}{8},
\end{equation}
we arrive at the $\tau$- uniform estimate
\begin{equation}\label{eqn:enest_caseI}
\begin{aligned}
&
\begin{multlined}[t] \taua\Bigl[\|\frakKone\Lconv \nabla\ppsi_{tt}\|_{L^2}^2+c^2\|\frakKone\Lconv \Delta\ppsi_t\|_{L^2}^2\Bigr]_0^t + \int_0^t \int_\Omega\nabla\ppsi_{tt} \, \frakKone\Lconv \nabla \ppsi_{tt}\dxs\\
+ \|\ppsi_{tt}\|_{L^2_t(L^2)}^2 
+\int_0^t \Psi[\nabla\ppsi_{tt}](s)\ds
+ \sup_{t'\in(0,t)}\Phi[\Delta \ppsi_{tt}](t') \end{multlined} \\
\lesssim_T&\, \begin{multlined}[t] 
\|f\|_{L^2(L^2)}^2 
+\|(\frakKone\Lconv\Delta\ppsi_t)(0)\|_{L^2}^2
+\tau^{2a}\|\frakKone\|_{{\mathcal{M}}(0,t)}^2 \|\Delta\ppsi_t(0)|_{L^2}^2
+ C_{\frakKone}\|\ppsi_{tt}(0)\|_{L^2}^2 \\
 + \|\Delta \ppsi(0)\|_{L^2}^2+  \|\Delta \ppsi_t(0)\|_{L^2}^2.
\end{multlined}
\end{aligned}
\end{equation}
Recall that we assume $\ttbdelta$ to be fixed while $\tau\in(0,\overline{\tau}]$ might be arbitrarily small.

\medskip

\noindent
\textbf{Case II}:
Here the $\aaa-\aaa_0$ term in \eqref{rhs} cannot be estimated by means of $\Phi[\Delta \ppsi_{tt}]$ any more due to the factor $\frakKone\Lconv\Delta \ppsi_{tt}$. We therefore have to assume $\aaa-\aaa_0$ to vanish, that is, $\aaa$ to be constant; cf.\ \eqref{eqn:cond_ab_CaseII}. \\
\def\ftil{h}
\indent The further estimate of
\begin{align} \label{rhs_new}
	\text{rhs}= \int_0^t\int_\Omega (f+c^2\bbb\Delta\ppsi)(-\frakKone\Lconv \Delta \ppsi_{tt})\dxs
\end{align}
 becomes somewhat complicated since now $\Phi[\Delta \ppsi_{tt}]$ cannot dominate the $L^2(L^2)$ norm of the multiplier $\frakKone\Lconv\Delta \ppsi_{tt}$. We estimate $\ftil:=f+c^2\bbb\Delta\ppsi$  using the integration by parts and transposition identities involving kernels
\begin{equation}\label{integration_by_parts_transposition_identity}
\begin{aligned}
&\int_0^T y_t(T-t)w(t) \dt =\int_0^T y(T-t) w_t(t) \dt- w(T)y(0)+w(0)y(T),\ y, w \in W^{1,1}(0,T)\\
&\int_0^T (\frakKone*y)(T-t) w(t) \dt=\int_0^T (\frakKone*w)(t)y(T-t)\dt, \ y, w \in L^1(0,T),
\end{aligned}
\end{equation}
as well as the timeflip operator defined by $\overline{\ftil}^{t}(s)=\ftil(t-s)$, and the identity 
\[
(\overline{a}^{t}, b)_{L^2(0,t)} = (a*b)(t)= (b*a)(t).
\]
In this manner, we obtain the following identities:
\begin{equation}
\begin{aligned}
&\Bigl|\int_0^t\int_\Omega \ftil(s) \, (\frakKone\Lconv \Delta \ppsi_{tt})(s)\dxs\Bigr|\\
=&\, \Bigl|\int_0^t\int_\Omega \overline{\ftil}^t(t-s) \, ((\frakKone\Lconv \Delta \ppsi_t)_t)(s)-\frakKone(s)\Delta\ppsi_t(0))\dxs\Bigr|\\
=&\, \Bigl|\int_\Omega \Bigl( \Delta \ppsi_t \Lconv \frakKone\Lconv (\overline{\ftil}^t)_t\Bigr)(t)
+\left[\overline{\ftil}^t(t-s)\, (\frakKone\Lconv \Delta \ppsi_t)(s)\right]_{s=0}^t-\int_0^t\overline{\ftil}^t(t-s)\,\frakKone(s)\Delta\ppsi_t(0)\dxs\Bigr|
\end{aligned}
\end{equation}
and thus
\begin{equation}
	\begin{aligned}
		&\Bigl|\int_0^t\int_\Omega \ftil(s) \, (\frakKone\Lconv \Delta \ppsi_{tt})(s)\dxs\Bigr|\\
		=&\,\begin{multlined}[t] \Bigl|\int_0^t\int_\Omega \Delta \ppsi_t(s) \, (\frakKone\Lconv (\overline{\ftil}^t)_t)(t-s)\dxs\\
			+\int_\Omega\Bigl(
			\ftil(t)(\frakKone\Lconv \Delta \ppsi_t)(t)
			-\ftil(0)(\frakKone\Lconv \Delta \ppsi_t)(0)
			-\Delta \ppsi_t(0) \int_0^t\frakKone(s)\ftil(s)\ds
			\Bigr)\dx\Bigr|.
		\end{multlined}
	\end{aligned}
\end{equation}
From here we have
\begin{equation}\label{eqn:est_step1_caseII}
	\begin{aligned}
		&\Bigl|\int_0^t\int_\Omega \ftil(s) \, (\frakKone\Lconv \Delta \ppsi_{tt})(s)\dxs\Bigr|\\
		\leq&\, 
		\, \begin{multlined}[t] \|\Delta\ppsi_t\|_{L^p_t(L^2)} \|\frakKone\Lconv (\overline{\ftil}^t)_t\|_{L^{p'}_t(L^2)}
			+\|\ftil\|_{L^\infty_t(L^2)} \|\frakKone\Lconv \Delta\ppsi_t\|_{L^\infty_t(L^2)}\\
			+\|\ftil(0)\|_{L^2} \|(\frakKone\Lconv \Delta\ppsi_t)(0)\|_{L^2}
			+\|\Delta\ppsi_t(0)\|_{L^2}
			\|\ftil\|_{L^\infty_t(L^2)} \|\frakKone\|_{{\mathcal{M}}(0,t)}
		\end{multlined}
	\end{aligned}
\end{equation}
for $p\in[1,\infty]$ as in \eqref{eqn:coercivity_2_CaseII} and $p'=\frac{p}{p-1}$. Further, by Young's convolution inequality
\begin{equation}\label{eqn_K1Lq}
\begin{aligned}
&\|\frakKone\Lconv (\overline{\bbb\Delta\ppsi}^t)_t\|_{L^{p'}_t(L^2)} 
\leq\, \|\frakKone\|_{L^q(0,t)} \|\bbb_t\Delta\ppsi+\bbb\Delta\ppsi_t\|_{L^{\tilde{q}}_t(L^2)}\\
\leq&\, \|\frakKone\|_{L^q(0,t)} 
\left(\|\bbb_t\|_{L^{\tilde{q}}_t(L^\infty)}\|\Delta\ppsi\|_{L^{\infty}_t(L^2)}+\|\bbb\|_{L^{\infty}_t(L^\infty)}\Delta\ppsi_t\|_{L^{\tilde{q}}_t(L^2)}\right),
\end{aligned}
\end{equation}
where $\tilde{q}=\frac{p'q}{p'q+q-p'}$ so that $1+\frac{1}{p'}=\frac{1}{q}+\frac{1}{\tilde{q}}$ and $q\leq p'$ as in \eqref{eqn:coercivity_K1}.
Note that estimate \eqref{eqn_K1Lq} is the general kernel substitute for the Kato--Ponce inequality used in \cite{kaltenbacher2022time} to analyze the fJMGT equations.
In case $\frakKone=\delta_0$, \eqref{eqn_K1Lq} remains valid with $\|\frakKone\|_{L^q(0,t)}$ replaced by one and $\tilde{q}=p'$.

Estimating further by means of assumption \eqref{eqn:coercivity_2_CaseII} (which again might require a decrease of $T$ to achieve $\overline{C}_3$ to be large enough) yields the $\tau$-uniform estimate:
\begin{equation}\label{eqn:enest_caseII}
\begin{aligned}
&
\begin{multlined}[t] \tau^a\Bigl[\|\frakKone\Lconv \nabla\ppsi_{tt}\|_{L^2}^2+c^2\|\frakKone\Lconv \Delta\ppsi_t\|_{L^2}^2\Bigr]_0^t + \int_0^t \int_\Omega\nabla\ppsi_{tt} \, \frakKone\Lconv \nabla \ppsi_{tt}\dxs\\
+ \|\ppsi_{tt}\|_{L^2_t(L^2)}^2 
+\int_0^t \Psi[\nabla\ppsi_{tt}](s)\ds
+ \sup_{t'\in(0,t)}\Phi[\Delta \ppsi_{tt}](t') \end{multlined}\\
\lesssim_T&\, \begin{multlined}[t] 
\|f\|_{L^\infty(L^2)}^2 
+\|f_t\|_{L^{\tilde{q}}_t(L^2)}^2 
+\|\Delta\ppsi_t(0)\|_{L^2}^2
+\|\Delta\ppsi(0)\|_{L^2}^2
+C_{\frakKone}\|\ppsi_{tt}(0)\|_{L^2}^2,
\end{multlined}
\end{aligned}
\end{equation}
again assuming $\ttbdelta$ to be fixed while $\tau\in(0,\overline{\tau}]$ might be arbitrarily small. As a consequence, the natural solution space here is  
\begin{equation}\label{eqn:defX_testK1}
	\begin{aligned}
		\mathcal{X}^\tau= \left\{\vphantom{W} \right.\ppsi \in& W^{s,\mathtt{r}}(0,T; \Honetwo)\cap  W^{\sigma,\rho}(0,T; \Honezero):\\
		&\ \ppsi_{tt} \in L^2(0,T; \Ltwo), 
 \, \ppsi_{t},\, \Delta\ppsi\in X^\infty_{\frakKone}(0,T; \Ltwo)\vphantom{W} \left \}  \vphantom{W} \right.
	\end{aligned}
\end{equation}
and the $\tau$ uniform solution space here is
\begin{equation}\label{eqn:defX_testK1_uniform}
	\begin{aligned}
		\mathcal{X}= \{\ppsi \in W^{s,\mathtt{r}}(0,T; \Honetwo)&\cap W^{\sigma,\rho}(0,T; \Honezero):\\ &\, \ppsi_{tt} \in L^2(0,T; \Ltwo)\},
	\end{aligned}
\end{equation}
where the space $X^\infty_{\frakKone}(0,t')$ is defined in \eqref{def_Xfp}.

As previously mentioned, since we have used $(\frakKone\Lconv \ppsi_{tt})_t$ in place of $\frakKone\Lconv \ppsi_{ttt}$ in the derivation of both energy estimates \eqref{eqn:enest_caseI} and \eqref{eqn:enest_caseII} above, we have to add $\frakKone(s)\cdot\ppsi_{tt}(0)$ to the right-hand side. Consequently we would have to assume $\frakKone\in L^2(0,T)$ in Case I and even $\frakKone\in L^\infty(0,T)\cap W^{1,\tilde{q}}(0,T)$ in Case II. To avoid this, we impose the condition $\ppsi_2=0$ in the upcoming analysis.

\begin{proposition}\label{Prop:Wellp_testingK1_lin}
	Let $T>0$ and $\tau \in (0, \bar{\tau}]$ for some fixed $\bar{\tau}>0$. Let assumptions \eqref{reg_kernels}  and
\eqref{Athree_general}--
\eqref{eqn:PhiPsi_Sobolevnorms} on the kernels hold, as well as 
\[\ppsi_0,\, \ppsi_1\, \in \Honetwo, \quad \ppsi_2=0\]
and the following assumptions on the kernels, coefficients $\aaa$, $\bbb$, and the right-hand side:
\begin{enumerate}
\item[(I)] \eqref{eqn:coercivity_2_CaseI}, \eqref{eqn:cond_a-a0_CaseI}, 
$f\in L^2(0,T;\Ltwo)$, 
or
\item[(II)] \eqref{eqn:coercivity_2_CaseII}, \eqref{eqn:cond_ab_CaseII}, \eqref{eqn:cond_pq_CaseII}, 
$f\in L^\infty(0,T;\Ltwo)\cap W^{1,\tilde{q}}(0,T;\Ltwo)$
\end{enumerate}
with $\tilde{q}$ as in \eqref{eqn:cond_pq_CaseII}.
Then there exists a unique solution $\ppsi \in \calX^\tau\subseteq\calX$ of the initial boundary-value problem
	\begin{equation} \label{IBVP_TestingK1_lin}
	\left \{	\begin{aligned}
	&\begin{multlined}[t]	\taua  (\frakKone * \ppsi_{tt})_t+\aaa(x,t) \ppsi_{tt}
	-c^2 \bbb(x,t) \Delta \ppsi
	-   {\tau^a} c^2 \frakKone* \Delta \ppsi_{t} -\ttbdelta  \frakKtwo* \Delta \ppsi_{tt}=f(x,t),\end{multlined}\\
	&\ppsi\vert_{\partial \Omega}=0, \\
	&	(\ppsi, \ppsi_{t}, \ppsi_{tt})\vert_{t=0} =(\ppsi_0, \ppsi_1, \ppsi_2) .
	\end{aligned} \right. 
	\end{equation}
	The solution satisfies the following estimate:     
\begin{equation}
\begin{aligned}
\|\ppsi\|_{\calX}^2\lesssim_T &
&\begin{cases}
\|f\|_{L^2(L^2)}^2 
+\|\Delta\ppsi(0)\|_{L^2}^2
+\|\Delta\ppsi_t(0)\|_{L^2}^2
+\tau^{2a}\|\frakKone\|_{{\mathcal{M}}(0,T)}^2 \|\Delta\ppsi_t(0)\|_{L^2}^2
& (\textup{I}),
\\
\|f\|_{L^\infty(L^2)}^2 
+\|f_t\|_{L^{\tilde{q}}(L^2)}^2 +\|\Delta\ppsi(0)\|_{L^2}^2
+\|\Delta\ppsi_t(0)\|_{L^2}^2
&(\textup{II}).
\end{cases}
\end{aligned}
\end{equation}
\end{proposition}

\begin{proof}
The proof follows by a Faedo--Galerkin procedure, where Appendix~\ref{Appendix:Galerkin} can be used to establish well-posedness of the ODEs resulting from Galerkin semidiscretization. The testing strategies shown above and resulting in \eqref{eqn:enest_caseI}, \eqref{eqn:enest_caseII} are then applied to the Galerkin solutions in place of $\ppsi$. \\
\indent 
Stating the energy estimates in terms of Sobolev norms $\|\ppsi_{tt}\|$, $\|\nabla\ppsi\|_{W^{s,r}(0,t)}$, and $\|\Delta\ppsi\|_{W^{\sigma,\rho}(0,t)}$ (cf. \eqref{eqn:PhiPsi_Sobolevnorms}), we can rely on weak$-*$ lower continuity of these norms when taking weak limits. Moreover, we use Lemma~\ref{Lemma:Caputo_seq_compact} with $p=\infty$ to also take limits in the $\tau^a$ terms of \eqref{eqn:enest_caseI} and \eqref{eqn:enest_caseII}. This is useful for enabling the use of $\Delta\ppsi_t$ as a multiplier in the uniqueness proof. \\[1mm]
\noindent \emph{Uniqueness.}
Since $\ppsi\in\mathcal{X}^\tau$, with $\mathcal{X}^\tau$ defined in \eqref{eqn:defX_testK1}, we can test the time integrated homogeneous PDE with vanishing $f$ and initial data (not only its Galerkin semidiscretization) in the way described above (that is, testing with the time integrated versions $-\frakKone\Lconv\Delta\ppsi_{t}$, $\ppsi_{t}$ of $-\frakKone\Lconv\Delta\ppsi_{tt}$, $\ppsi_{tt}$) and analogously obtain from \eqref{eqn:enest_caseI}, \eqref{eqn:enest_caseII}  that its solution needs to vanish. 
\end{proof}

\subsection{Uniform well-posedeness with Kuznetsov--Blackstock and Westervelt--Blackstock nonlinearities in Case I}

To relate the previous analysis to the nonlinear equation, we again consider the fixed-point mapping $\mathcal{T}: \calB \ni \ppsi^* \mapsto \ppsi$, 
on the ball
\begin{equation}
\begin{aligned}
\calB =\left \{ \ppsi^* \in \calX:\right.&\, \, \|\ppsi^*\|_{\calX} \leq R,\ 
\left	(\ppsi^*, \ppsi^*_{t}, \ppsi^*_{tt})\vert_{t=0} =(\ppsi_0, \ppsi_1, 0) \}. \right.
\end{aligned}
\end{equation}
Here $\ppsi=\mathcal{T}\ppsi^*$ solves \eqref{pde_testingK1_lin} with
\begin{equation}\label{ab_testingK1}
\begin{aligned}
\aaa(\ppsi^*,\ppsi^*_t) &=\begin{cases} \,1+2k_1 \ppsi^*& (\textup{WB}), \\ \,1+2k_1 \ppsi^*_t& (\textup{KB}),\end{cases}
\qquad \bbb(\ppsi^*,\ppsi^*_t) = \begin{cases} \,1-2k_2 \ppsi^*& (\textup{WB}),\\ \,1-2k_2 \ppsi^*_t& (\textup{KB}),\end{cases}
\end{aligned}
\end{equation}
and 
\begin{equation}\label{f_testingK1}
	\begin{aligned}
f(x,t)&=\, -\calN(\ppsi^*_t, \nabla\ppsi^*, \nabla\ppsi^*_t)= 
\begin{cases}
-2 {\kthree} (\ppsi^*_t)^2& (\textup{WB}),\\
-2 \kthree \nabla \ppsi^*\cdot \nabla \ppsi_t^*& (\textup{KB} ),\end{cases}
\end{aligned}
\end{equation}
where (WB) stands for the Westervelt--Blackstock and (KB) for Kuznetsov--Blackstock nonlinearities.
\begin{theorem} \label{Thm:Wellp_testingK1_caseI}
	Let $T>0$ and $\tau \in (0, \bar{\tau}]$. 
Let assumptions 
\eqref{reg_kernels}  and
\eqref{Athree_general}--
\eqref{eqn:PhiPsi_Sobolevnorms},
\eqref{eqn:coercivity_2_CaseI}
on the kernels $\frakKone$ and $\frakKtwo$ hold with $s$, $\mathtt{r}$, $\sigma$, and $\rho$ so that the following continuous embedding holds:
\begin{equation}\label{embedding_testingK1_caseI}
\begin{aligned}
&W^{s,\mathtt{r}}(0,T; \Honetwo)\cap W^{\sigma,\rho}(0,T; \Hone)\\
\hookrightarrow&\,
\begin{cases}
L^\infty(0,T;L^\infty(\Omega))\cap W^{1,4}(0,T;L^4(\Omega))\cap L^2(0,T;\Honetwo)& (\textup{WB}),\\
W^{1,\infty}(0,T;L^\infty(\Omega))\cap H^1(0,T;W^{1,4}(\Omega))\cap L^2(0,T;\Honetwo)& (\textup{KB}).
\end{cases}
\end{aligned}
\end{equation} 
There exists a data size $r={r(T)}>0$, independent of $\tau$, such that if
	\begin{equation}\label{bound_init}
\|\Delta\ppsi(0)\|_{L^2}^2
+\|\Delta\ppsi_t(0)\|_{L^2}^2
+\tau^{2a}\|\frakKone\|_{{\mathcal{M}}(0,T)}^2 \|\Delta\ppsi_t(0)\|_{L^2}^2
\leq r^2,
	\end{equation}
	then there is a unique solution $\ppsi \in \calB$ of the nonlinear initial boundary-value problem 
	\begin{equation} \label{IBVP_testingK1}
	\left \{	\begin{aligned}
	&\begin{multlined}[t]	\taua  \frakKone * \ppsi_{ttt}+\aaa(\ppsi,\ppsi_t) \ppsi_{tt}
	-c^2 \bbb(\ppsi,\ppsi_t) \Delta \ppsi
	-   \taua c^2 \frakKone* \Delta \ppsi_t -\ttbdelta  \frakKtwo* \Delta \ppsi_{tt}\\ 
=-\calN(\ppsi_t, \nabla\ppsi, \nabla\ppsi_t),
\end{multlined}\\
	&\ppsi\vert_{\partial \Omega}=0, \\
	&	(\ppsi, \ppsi_t, \ppsi_{tt})\vert_{t=0} =(\ppsi_0, \ppsi_1, 0) .
	\end{aligned} \right. 
	\end{equation}
\end{theorem}
\begin{proof}
The proof goes analogously to the one of Theorem~\ref{Thm:WellP_GFEIII_BKuznetsov}, using the fact that $\ppsi:=\mathcal{T}\ppsi^*$ solves \eqref{IBVP_TestingK1_lin} with $\aaa$, $\bbb$, and $f$ as in \eqref{ab_testingK1} and \eqref{f_testingK1} to establish self-mapping and the fact that 
$\phi=\ppsi-v:=\mathcal{T}(\ppsi^*)-\mathcal{T}(v^*)$ solves \eqref{IBVP_TestingK1_lin} with 
\begin{equation}\label{ab_testingK1_diff}
\begin{aligned}
\aaa(\ppsi^*,\ppsi^*_t) &=\begin{cases} \,1+2k_1 \ppsi^*& (\textup{WB}),\\ \,1+2k_1 \ppsi^*_{t}& (\textup{KB}),\end{cases}
\qquad \bbb(\ppsi^*,\ppsi^*_t)= \begin{cases} \,1-2k_2 \ppsi^*& (\textup{WB}),\\ \,1-2k_2 \ppsi^*_{t}& (\textup{KB}),\end{cases}
\end{aligned}
\end{equation}
and, with $\phi^*=u^*-v^*$,
\begin{equation}\label{f_testingK1_diff}
	\begin{aligned}
=&\,\begin{cases}
-2\phi^*(k_1v_{tt}+k_2\Delta v)
-2\kthree \phi^*_t(\ppsi^*_{t}+v^*_{t})& (\textup{WB}),\\[2mm]
-2\phi^*_{t}(k_1v_{tt}+k_2\Delta v)
-2 \kthree \nabla \phi^*\cdot \nabla u^*_{t}
-2 \kthree \nabla v^*\cdot \nabla \phi^*_{t}\ & (\textup{KB}),\end{cases}
\end{aligned}
\end{equation}
for establishing contractivity. Indeed, the continuous embedding assumption \eqref{embedding_testingK1_caseI} together with smallness of data size $r$ as well as $R$ ensures the required bounds. 
\end{proof}
The fact that a Kuznetsov-type nonlinearity is enabled here under only $H^2$ regularity of the initial data is due to the relative strength of the damping term with factor $\delta$ in Case I, where kernel $\frakKtwo$ is more singular than $\frakKone$.

\subsection{Uniform existence with Westervelt--Blackstock nonlinearities in Case II}
{To treat Case II, we restrict our considerations to the Westervelt--Blackstock nonlinearities with $k_1=0$; that is, we analyze the following equation: 
\begin{equation}
	\taua  \frakKone * \ppsi_{ttt}+\ \ppsi_{tt}
	-c^2 (1-2k_2 \ppsi) \Delta \ppsi
	-   \taua c^2 \frakKone* \Delta \ppsi_t -\ttbdelta  \frakKtwo* \Delta \ppsi_{tt}+ 2k_3 \ppsit^2=0.
\end{equation}}
\begin{theorem} \label{Thm:Wellp_testingK1_caseII}
	Let $T>0$ and $\tau \in (0, \bar{\tau}]$. 
Let assumptions 
\eqref{reg_kernels}  and
\eqref{Athree_general}--
\eqref{eqn:PhiPsi_Sobolevnorms},
\eqref{eqn:coercivity_2_CaseII}, \eqref{eqn:cond_pq_CaseII}, 
on the kernels $\frakKone$ and $\frakKtwo$ and $\ppsi_0,\, \ppsi_1\, \in \Honetwo$, $\ppsi_2=0$ on 
the initial data  hold and 
\begin{equation}\label{embedding_testingK1_caseII}
W^{s,\mathtt{r}}(0,T; \Honetwo)\cap W^{\sigma,\rho}(0,T; \Hone)\hookrightarrow L^\infty(0,T;L^\infty(\Omega)).
\end{equation}
There exists a data size $r=r(T)>0$, independent of $\tau$, such that if \eqref{bound_init} holds,
then there is a unique solution $\ppsi \in \calB$ of the nonlinear initial boundary-value problem \eqref{IBVP_testingK1} in case \textup{(WB)} with $k_1=0$ in \eqref{ab_testingK1}.
\end{theorem}

\begin{proof}
Establishing a self-mapping property of $\mathcal{T}$ on a sufficiently small ball works as in the proof of Theorems~\ref{Thm:WellP_GFEIII_BKuznetsov}, \ref{Thm:Wellp_testingK1_caseI}, based on assumption of having the continuous embedding in \eqref{embedding_testingK1_caseII}, which implies continuity of the embedding
\[
\begin{aligned}
&W^{s,\mathtt{r}}(0,T; \Honetwo)\cap W^{\sigma,\rho}(0,T; \Hone)\\
\hookrightarrow&\,
L^\infty(0,T;L^\infty(\Omega))\cap W^{1,\tilde{q}}(0,T;L^\infty(\Omega))\cap  W^{1,\infty}(0,T;L^4(\Omega))\cap W^{1,\frac{2\tilde{q}}{2-\tilde{q}}}(0,T;L^\infty(\Omega)).
\end{aligned}
\]
However, a corresponding assumption leading to contractivity based on \eqref{ab_testingK1_diff}, \eqref{f_testingK1_diff}  would be unrealistically strong; see Remark~\ref{rem:nonlinGFE}. We therefore (similarly to \cite{splittingModelsNonlinearAcoustics, kaltenbacher2019jordan}) only prove existence of solutions based on a general version of Schauder's fixed-point theorem in locally convex topological spaces (see~\cite{Fan1952}) and for this purpose establish weak$-*$ continuity of $\mathcal{T}$ as follows. \\
\indent For any sequence $(\ppsi^*_n)_{n\in\mathbb{N}}\subseteq \calB$ that weakly$-*$ converges to $\ppsi^*\in \calB$ in $\calX$, we also have \[(\mathcal{T}(\ppsi^*_n))_{n\in\mathbb{N}}\subseteq \calB.\] Thus, by compactness of the embedding $\calX\to W^{1,\infty}(0,T;L^\infty(\Omega))$, there exists a subsequence $(\ppsi^*_{n_\ell})_{\ell\in\mathbb{N}}$ such that $1\pm k_{1/2}\ppsi^*_{n_\ell\, t}$ converges to $1\pm k_{1/2}\ppsi^*_t$ strongly in $L^\infty(0,T;L^\infty(\Omega))$ and $\mathcal{T}(\ppsi^*_{n_\ell})$ converges weakly* in $\calX$ to some $\ppsi\in \calB$, which by definition of $\calB$ satisfies the initial and homogeneous Dirichlet boundary conditions. It is readily checked that $\ppsi$ also solves the PDE defining $\mathcal{T}(\ppsi^*)$, which, by uniqueness in Proposition~\ref{Prop:Wellp_testingK1_lin}, implies $\ppsi=\mathcal{T}(\ppsi^*)$. A subsequence-subsequence argument yields weak$-*$ convergence in $\calX$ of the whole sequence $(\mathcal{T}(\ppsi^*_n))_{n\in\mathbb{N}}$ to $\mathcal{T}(\ppsi^*)$.
\end{proof}


As can be read off from Table~\ref{table:GFElaws_pq_etc}, embedding \eqref{embedding_testingK1_caseII} is satisfied for all Compte--Metzler laws. To establish embedding \eqref{embedding_testingK1_caseI} for GFE I, we use interpolation 
\[
\|w\|_{H^{\frac32+\frac{\alpha(1-\epsilon)}{4}}(0,T;H^{\frac{3+\epsilon}{2}}(\Omega))}\lesssim
\|w\|_{H^{\frac32}(0,T;\Honetwo)}^{\frac{1+\epsilon}{2}}
\|w\|_{H^{\frac{3+\alpha}{2}}(0,T;\Hone)}^{\frac{1-\epsilon}{2}}
\]
with $\alpha>0$ and $\epsilon\in(0,1)$. 
\begin{remark}[Uniqueness of solutions of \eqref{IBVP_testingK1} in Case II]\label{rem:nonlinGFE}
To prove contractivity of $\mathcal{T}$ in Case II with Westervelt--Blackstock nonlinearity when $k_1 \neq 0$ (that is, with $\aaa=1+2k_1 u$), we would need  
\begin{equation} 
\begin{aligned}
	2\phi^*(k_1v_{tt}+k_2\Delta v)-2(k_1+k_2)(\ppsi_{t}^*+v_{t}^*)\phi_{t}^*
\in \, L^\infty(0,T;L^2(\Omega))\cap W^{1,\tilde{q}}(0,T;L^2(\Omega)),
\end{aligned}	
\end{equation}
thus requiring an estimate on $v_{ttt}$ (and thus $\psittt$ in the existence proof). 
This is clearly beyond the scope of the available energy estimates. Time differentiation of the PDE and further testing might enable this at the cost of stronger conditions on the initial data.

Likewise, Kuznetsov-type nonlinearities ($\aaa=1+2k_1 \ut$) in Case II would require $\ppsi_{tt}\in L^{\tilde{q}}(0,T;L^\infty(\Omega))$ (as needed for $\bbb_t\in L^{\tilde{q}}(0,T;L^\infty(\Omega))$ in estimate \eqref{eqn_K1Lq}), which seems to be out of reach for most Compte--Metzler laws.  {The \ref{fJMGT} equation based on the GFE law, where $\frakKtwo=1$, allows for an alternative testing strategy that makes it possible to incorporate also Kuznetsov-type nonlinearities under stronger assumptions on the regularity of data; we refer to~\cite{nikolic2023} for details and the corresponding analysis. }\\
\indent Note that the results obtained for the \ref{fJMGTIII} equation here are weaker than in the previous Section~\ref{Sec:GFEIII}.
\end{remark}

\subsection{Weak singular limits} We next discuss the weak limiting behavior of these nonlocal equations in both cases. Let $\tau \in (0, \bar{\tau}]$. Under the assumptions of Theorems~\ref{Thm:Wellp_testingK1_caseI} and \ref{Thm:Wellp_testingK1_caseII} with uniformly bounded data
\begin{equation}
\|\Delta\ppsi^\tau(0)\|_{L^2}^2
+\|\Delta\ppsi^\tau_t(0)\|_{L^2}^2
+\tau^{2a}\|\frakKone\|_{{\mathcal{M}}(0,T)}^2 \|\Delta\ppsi^\tau_t(0)\|_{L^2}^2 \leq r^2
\end{equation}
{and  $r$ independent of $\tau$ as in Theorems~\ref{Thm:Wellp_testingK1_caseI} and \ref{Thm:Wellp_testingK1_caseII}},
consider
\begin{equation} 
	\left \{	\begin{aligned}
&\begin{multlined}[t]	\taua  \frakKone * \ppsi_{ttt}^\tau+\aaa(\ppsi^\tau, \ppsi^\tau_t) \ppsi_{tt}^\tau
-c^2 \bbb(\ppsi^\tau, \ppsi^\tau_t) \Delta \ppsi^\tau
-   \taua c^2 \frakKone* \Delta \ppsi_{t}^\tau \\ \hspace*{4cm}-\ttbdelta  \frakKtwo* \Delta \ppsi_{tt}^\tau + 
\calN(\ppsi_t, \nabla\ppsi, \nabla\ppsi_t)
=0 \quad \text{in } \Omega \times (0,T), \end{multlined}\\
&\ppsi^\tau\vert_{\partial \Omega}=0, \\
&	
{
(\ppsi^\tau, \ppsi_{t}^\tau, \ppsi_{tt}^\tau)\vert_{t=0} =(\ppsi^\tau_0, \ppsi^\tau_1, 
0
) .
}
\end{aligned} \right. 
\end{equation}
According to Theorems~\ref{Thm:Wellp_testingK1_caseI}, \ref{Thm:Wellp_testingK1_caseII} a solution of this problem exists  in $ \calB$ (although it might not be unique in Case II). Therefore, we have the following uniform bounds with respect to the relaxation time:
\begin{equation} \label{bounds_tau_testingK1}
\left \{	\begin{aligned}
&\ppsi^\tau \ \text{is bounded in } W^{s,\mathtt{r}}(0,T; \Honetwo)\cap W^{\sigma,\rho}(0,T; \Hone), \\
&\ppsi_{tt}^\tau  \ \text{is bounded in } L^2(0,T; \Ltwo). 
\end{aligned} \right.
\end{equation}
hence existence of a subsequence, not relabeled, such that 
\begin{equation} \label{weak_limits_1_tau_testingK1}
\begin{alignedat}{4} 
\ppsi^\tau&\relbar\joinrel\rightharpoonup \ppsi && \quad\text{ weakly-$\star$} &&\text{ in } &&W^{s,\mathtt{r}}(0,T; \Honetwo)\cap W^{\sigma,\rho}(0,T; \Hone),  \\
\ppsi_{tt}^\tau &\relbar\joinrel\rightharpoonup \ppsi_{tt} && \quad \text{ weakly}  &&\text{ in } &&L^2(0,T; \Ltwo).
\end{alignedat} 
\end{equation} 
Assuming compactness of the embedding 
\begin{equation}\label{compactness_tau_testingK1}
\calX\hookrightarrow  \hookrightarrow C([0,T]; H^{\beta_0}(\Omega))\cap C^1([0,T]; H^{\beta_1}(\Omega)), \quad
\beta_0,\, \beta_1\,\geq0 
\end{equation}
we have additionally
\begin{equation} \label{strong_Climits_tau_testingK1}
\begin{alignedat}{4} 
\ppsi^\tau&\longrightarrow \ppsi && \quad \text{ strongly}  &&\text{ in } &&
C([0,T]; H^{\beta_0}(\Omega)),\\
\ppsi_{t}^\tau &\longrightarrow \ppsi_{t} && \quad \text{ strongly}  &&\text{ in } &&
C([0,T]; H^{\beta_1}(\Omega)).
\end{alignedat} 
\end{equation} 
We should prove that $u$ solves the limiting problem. Let $v \in C_0^\infty([0,T]; C_0^\infty(\Omega))$. We have with $\bar{\ppsi}=\ppsi-\ppsi^\tau$:
\begin{equation}
\begin{aligned}
&\begin{multlined}[t] \intTO \aaa(\ppsi, \ppsi_t) \ppsi_{tt} v \dxs-c^2 \intTO \bbb(\ppsi, \ppsi_t) \Delta \ppsi v \dxs +\rt^b \delta \intTO  \frakKtwo* \nabla \ppsi_{tt} \cdot \nabla v \dxs\\
+\intTO \calN(\ppsi_t, \nabla \ppsi, \nabla \ppsi_{t}) v \dxs \end{multlined}\\
=&	\begin{multlined}[t]  \intTO  \aaa(\ppsi,\ppsi_{t})\bar{\ppsi}_{tt} v \dxs- c^2 \intTO \bbb(\ppsi,\ppsi_{t}) \Delta \bar{\ppsi} v \dxs+\rt^b \delta \intTO  \frakKtwo* \nabla \bar{\ppsi}_{tt} \cdot \nabla v \dxs \\
- \intTO  \taua \frakKone * (\ppsi_{ttt}^\tau - c^2 \Delta \ppsi_{t}^\tau) v \dxs \\
-\intTO (\aaa(\ppsi^\tau,\ppsi_{t}^\tau)-\aaa(\ppsi,\ppsi_{t})) \ppsi_{tt}^\tau v \dxs + c^2\intTO (\bbb(\ppsi^\tau,\ppsi_{t}^\tau)-\bbb(\ppsi, \ppsi_{t}))\Delta \ppsi^\tau v \dxs\\ -\intTO (\calN(\ppsi^\tau_t, \nabla \ppsi^\tau,  \nabla \ppsi_{t}^\tau) -\calN(\ppsi_t, \nabla \ppsi,  \nabla \ppsi_{t}))v \dxs.  \end{multlined}
\end{aligned}
\end{equation}
We wish to prove that the right-hand side tends to zero as $\tau \searrow 0$. To this end, we rely on the established weak and strong convergence in \eqref{weak_limits_1_tau_testingK1} and \eqref{strong_Climits_tau_testingK1}, respectively. We first discuss the terms involving the kernels and treat them by means of transposition and integration by parts (see \eqref{integration_by_parts_transposition_identity}), 
which as compared to the proof of Theorem~\ref{Thm:WeakLim_GFEIII_BWest} is required due to the limited regularity of $\bar{\ppsi}_{tt}$: 
\begin{equation}
\begin{aligned}
&\intTO  (\frakKtwo* \nabla \bar{\ppsi}_{tt})(s) \cdot \nabla v(s) \dxs
=-\intTO  (\frakKtwo* \bar{\ppsi}_{tt})(s) \cdot \overline{\Delta v}^T(T-s) \dxs\\
&=-\intTO  \bar{\ppsi}_{tt}(s) \cdot (\frakKtwo* \overline{\Delta v}^T)(T-s) \dxs
\ \rightarrow  0 \quad \text{as } \tau \searrow 0,
\end{aligned}
\end{equation}
and for $w^\tau:= \ppsi_{tt}^\tau - c^2 \Delta \ppsi^\tau$
\begin{equation}
\begin{aligned}
&\taua \intTO  (\frakKone * w^\tau_t)(s) v(s) \dxs
=\taua \intTO  (\frakKone * w^\tau_t)(s) \overline{v}^T (T-s)\dxs\\
&=\taua \intTO  w^\tau_t(s) (\frakKone * \overline{v}^T) (T-s)\dxs\\
&=\taua \Bigl(\intTO  w^\tau(s) (\frakKone * \overline{v}^T)_t (T-s)\dxs
-\int_\Omega w^\tau(0) (\frakKone * \overline{v}^T)_t (T)\dx\Bigr)
\ \rightarrow  0 \quad \text{as } \tau \searrow 0.
\end{aligned}
\end{equation}
Here we have used the fact that due to $v \in C_0^\infty([0,T]; C_0^\infty(\Omega))$ we have 
\[(\frakKone * \overline{v}^T)_t (t)= (\frakKone * \overline{v}^T_t) (t)+ \frakKone(t) \overline{v}^T(0)=-\int_0^t\frakKone(s)v_t(T-s)\ds\] which vanishes for $t=0$, even in case $\frakKone=\delta_0$.

By relying on the equivalence of the norms $\|{\aaa(\ppsi,\ppsi_{t}) v}\|_{L^2}$, $\|v\|_{L^2}$, and $\|\bbb(\ppsi,\ppsi_t) v\|_{L^2}$, we can conclude that 
\begin{equation}
\begin{aligned}
\intTO  \aaa(\ppsi,\ppsi_{t})\bar{\ppsi}_{tt} v \dxs+ c^2 \intTO \bbb(\ppsi,\ppsi_{t}) \Delta \bar{\ppsi} v \dxs  \rightarrow  0 \quad \text{as } \tau \searrow 0.
\end{aligned}
\end{equation}
We have
\begin{equation}
\begin{aligned}
&-\intTO (\aaa(\ppsi^\tau,\ppsi_{t}^\tau)-\aaa(\ppsi,\ppsi_{t})) \ppsi^\tau_{tt} v \dxs + c^2\intTO (\bbb(\ppsi^\tau,\ppsi_{t}^\tau)-\bbb(\ppsi,\ppsi_{t}))\Delta \ppsi^\tau v \dxs \\
=& \, -2\intTO \bar{w}(k_1 \ppsi_{tt}^\tau + k_2 \Delta \ppsi^\tau) v \dxs 	
\end{aligned}
\end{equation}
with $\bar{w}=\bar{\ppsi}$ (WB) or $\bar{w}=\bar{\ppsi}_t$ (KB), 
which tends to zero as well thanks to \eqref{bounds_tau_testingK1} and \eqref{strong_Climits_tau_testingK1}. Finally, with the Westervelt--Blackstock type nonlinearity (cf.\ \eqref{ab_testingK1_diff},  \eqref{f_testingK1_diff}):
\begin{equation}
\begin{aligned}
\intTO (\calN(\ppsi^\tau_t) -\calN(\ppsi_t))v \dxs 
=&\, 2k_3 \intTO \bar{\ppsi}_t(\ppsi_t+\ppsi^\tau_t)  v \dxs \\
=&\,-2k_3\intTO \bar{\ppsi}((\ppsi_t+\ppsi^\tau_t)  v)_t \dxs
\end{aligned}
\end{equation}
and with the Kuznetsov--Blackstock nonlinearity:
\begin{equation}
\begin{aligned}
&\intTO (\calN(\nabla \ppsi^\tau, \ppsi_{t}^\tau) -\calN(\nabla \ppsi,  \nabla \ppsi_{t}))v \dxs
=&\, 2 \kthree \intTO\Bigl( \nabla \bar{\ppsi}\cdot \nabla \ppsi^\tau_t+\nabla \ppsi\cdot \nabla \bar{\ppsi}_t\Bigr)v \dxs.
\end{aligned}
\end{equation}
We have
\[
\begin{aligned}
&\nabla\bar{\ppsi}\to0\text{ in }L^2(0,T; \Ltwo), \qquad
\|v\nabla \ppsi^\tau_t\|_{L^2(L^2)}
\leq\|v\|_{L^\infty(L^\infty)} \|\nabla \ppsi^\tau_t\|_{L^2(;L^2)}
\end{aligned}
\]
and 
\[
\begin{aligned}
	&\bar{\ppsi}_t\rightharpoonup0\ \text{weakly in }L^2(0,T; \Ltwo), \quad
	\|\nabla\cdot(v\nabla \ppsi)\|_{L^2(L^2)}
	\leq \begin{multlined}[t]\|v\|_{L^\infty(L^\infty)} \|\Delta \ppsi\|_{L^2(L^2)}\\
	+\|\nabla v\|_{L^\infty(L^4)} \|\nabla \ppsi\|_{L^2(L^4)}. \end{multlined}
\end{aligned}
\]
In both cases these terms tend to zero on account of again \eqref{bounds_tau_testingK1}, \eqref{weak_limits_1_tau_testingK1}, \eqref{strong_Climits_tau_testingK1}. The attainment of initial conditions $(\ppsi_1, \ppsi_2)$ follows by \eqref{strong_Climits_tau_testingK1}, analogously to \eqref{limits_tau_initial_BW}. 
With a subsequence-subsequence argument and using uniqueness for the limiting equation,
similarly to Theorem~\ref{Thm:WeakLim_GFEIII_BWest}, this leads to the following result.
\begin{theorem} \label{Thm:WeakLim_testingK1}
	Let the assumptions of Theorem~\ref{Thm:Wellp_testingK1_caseI} or \ref{Thm:Wellp_testingK1_caseII} 
with 
	\begin{equation}
\|\Delta\ppsi^\tau(0)\|_{L^2}^2+\|\Delta\ppsi^\tau_t(0)\|_{L^2}^2
+\tau^{2a}\|\frakKone\|_{{\mathcal{M}}(0,T)}^2 \|\Delta\ppsi^\tau_t(0)\|_{L^2}^2 
\leq r^2,
	\end{equation}
for $\tau \in (0, \bar{\tau}]$,
as well as embedding \eqref{compactness_tau_testingK1} hold. Then any family $\{\ppsi^\tau\}_{\tau \in (0, \bar{\tau}]}$ of solutions to
	\eqref{IBVP_testingK1}
	converges weakly in the sense of \eqref{weak_limits_1_tau_testingK1} to the solution $\ppsi \in\calX$ of 
	\begin{equation} \label{IBVP_tau0_testingK1}
	\left \{	\begin{aligned}
	&\begin{multlined}[t]	\aaa(\ppsi,\ppsi_t) \ppsi_{tt}
	-c^2 \bbb(\ppsi,\ppsi_t) \Delta \ppsi
	 -\ttbdelta  \frakKtwo* \Delta \ppsi_{tt}
+\calN(\ppsi_t, \nabla\ppsi, \nabla\ppsi_t) =0
\end{multlined}\\
	&\ppsi\vert_{\partial \Omega}=0, \\
	&	(\ppsi, \ppsi_t)\vert_{t=0} =(\ppsi_0, \ppsi_1) 
	\end{aligned} \right. 
	\end{equation}
with $\aaa$ and $\bbb$ as in \eqref{ab_testingK1} and $\calN$ as in  \eqref{f_testingK1} (and restricted to the Westervelt--Blackstock case with {$k_1=0$} under the conditions of Theorem~\ref{Thm:Wellp_testingK1_caseII}).
\end{theorem}
\begin{remark}[Weak limits for solutions of the fJMGT equations]\label{rem:taulimitGFE}
The condition on the compactness of the embedding \eqref{compactness_tau_testingK1} that is left to be verified,  for equations based on the Compte--Metzler fractional laws holds by interpolation
\[
\begin{aligned}
H^\sigma(0,T;H^1(\Omega))\cap H^s(0,T;H^2(\Omega))
&\subseteq H^{\theta_i\sigma+(1-\theta_i)s}(0,T;H^{2-\theta_i}(\Omega))\\
&\hookrightarrow C^i([0,T];H^{\beta_i}(\Omega)), \quad i\in\{0,1\}
\end{aligned}
\]
for $\theta_i > \frac{i+1/2-s}{\sigma-s}$.
This yields $\beta_0\in(0,2)$ and $\beta_1\in(0,2-\frac{3-2s}{2(\sigma-s)})$; more precisely,
$\beta_1\in(0,2)$ for the laws GFE I, GFE III and $\beta_1\in(0,1+\alpha)$ for GFE.
\end{remark}
\section*{Acknowledgments} 
The work of BK was supported by the Austrian Science Fund FWF under the grants P36318 and DOC 78.
\renewcommand\appendixname{Appendix}
\begin{appendices}
\section{Analysis of the semi-discrete problems} \label{Appendix:Galerkin}
We present in this appendix the proof of the unique solvability of the semi-discrete problem considered in {Propositions~\ref{Prop:Wellp_GFEIII_BWest_lin} and~\ref{Prop:Wellp_testingK1_lin}}. Let $\{\phi_i\}_{i \geq 1}$ be a basis of $V=\Honetwo$ consisting of the eigenfunctions of the Dirichlet--Laplace operator. Let $V_n=\text{span}\{\phi_1, \ldots, \phi_n\} \subset V$ and set
\begin{equation}
	\begin{aligned}
		\ppsi^{(n)}(t) = \sum_{i=1}^n \xin (t) \phi_i.
	\end{aligned}	
\end{equation}
We choose the approximate initial data as
\begin{equation}
	\ppsi^{(n)}_0=  \sum_{i=1}^n \xi_i^{(0, n)}(t) \phi_i,\quad \ppsi^{(n)}_1= \sum_{i=1}^n \xi_i^{(1, n)}(t) \phi_i, \quad  \ppsi^{(n)}_2 \sum_{i=1}^n \xi_i^{(2, n)}(t) \phi_i \in V_n,
\end{equation}
({with $\ppsi^{(n)}_2=0$ in Proposition~\ref{Prop:Wellp_testingK1_lin}})  such that
\begin{equation} \label{convergence_approx_initial_data}
	\begin{aligned}
		\ppsi^{(n)}_0\rightarrow \ppsi_0 \  \text{in} \ \Honetwo, \  	\ppsi^{(n)}_1 \rightarrow \ppsi_1 \ \text{in} \ \Honetwo, \ \text{and } \, 	\ppsi^{(n)}_2 \rightarrow \ppsi_2 \  \text{in} \ \Honezero, \ n \rightarrow \infty.
	\end{aligned}
\end{equation}
For each $n \in \N$, the system of Galerkin equations is given by
\begin{equation}
	\begin{aligned}
		\begin{multlined}[t]\taua	\sum_{i=1}^n (\frakKone * \xittt)(t) (\phi_i, \phi_j)_{L^2}+ \sum_{i=1}^n \xitt (\aaa(t)\phi_i, \phi_j)_{L^2}+c^2  \sum_{i=1}^n \xin (\bbb(t) \Delta \phi_i, \phi_j)_{L^2} \\
			+\tau^a c^2 \sum_{i=1}^n (\frakKone * \xit)(t) (\nabla \phi_i, \nabla \phi_j)_{L^2} +\ttbdelta \sum_{i=1}^n (\frakKtwo * \xitt)(t) (\nabla \phi_i, \nabla \phi_j)_{L^2} \\
			= (f(t), \phi_j)_{L^2}
		\end{multlined}	
	\end{aligned}		
\end{equation}
for a.e.\ $t \in (0,T)$ and all $j \in \{1, \ldots, n\}$. With $\bxi = [\xin_1 \ \ldots \  \xin_n]^T$, we can write this system in matrix form 
\begin{equation}
	\left \{	\begin{aligned}
		& \taua M	\frakKone * \bxittt + M_\aaa \bxitt+ K_\bbb\bxi+\tau^a c^2 K  \frakKone* \bxit + \ttbdelta K \frakKtwo *\bxitt = \boldsymbol{f}, \\[1mm]
		& (\bxin, \bxit, \bxitt)\vert_{t=0} = (\boldsymbol{\xi_0}, \boldsymbol{\xi_1}, \boldsymbol{\xi_2}),
	\end{aligned} \right.
\end{equation}
where $(\boldsymbol{\xi_0}, \boldsymbol{\xi_1}, \boldsymbol{\xi_2})=([\xi_1^{(0,n)} \, \ldots \, \xi_n^{(0, n)}]^T,\, [\xi_1^{(1,n)} \, \ldots \, \xi_n^{(1, n)}]^T,\, [\xi_1^{(2,n)} \, \ldots \, \xi_n^{(2, n)}]^T)$. Above, $M$ and $K$ are the standard mass and stiffness matrices, respectively. $M_\aaa$ and $K_\bbb$ are the corresponding matrices with wieghted entries:
\[
M_{\aaa, ij}=(\aaa \phi_i, \phi_j)_{L^2}, \quad K_{\bbb, ij}= (\bbb \nabla \phi_i,\nabla \phi_j)_{L^2}.
\] 
Let $\bmu= \frakKone*\bxittt$ be the new unknown. Then with $\tfrakKone$ defined by \eqref{Athree_GFEIII} we have
\begin{equation}
	\begin{aligned}
		&\bxitt=\, \tfrakKone* \bmu + \boldsymbol{\xi_2}, \\
		& \bxit= 1*\tfrakKone*\bmu+ \boldsymbol{\xi_2}t+\boldsymbol{\xi_1}, \\
		&\bxi=\,1*1*\tfrakKone*\bmu+\boldsymbol{\xi_2} \frac{t^2}{2}+\boldsymbol{\xi_1}t +\boldsymbol{\xi_0}.
	\end{aligned}
\end{equation}
The system can then be equivalently rewritten as a system of Volterra equations:
\begin{equation}
	\begin{aligned}
		\begin{multlined}[t] \taua	\bmu+ M^{-1} M_\aaa(\tfrakKone* \bmu + \boldsymbol{\xi_2})+ M^{-1}K_\bbb (1*1*\tilde{\frakKone}*\bmu+\boldsymbol{\xi_2} \frac{t^2}{2}+\boldsymbol{\xi_1}t +\boldsymbol{\xi_0})\\
			+\tau^a c^2 M^{-1}K  \frakKone* (1*\tilde{\frakKone}*\bmu+ \boldsymbol{\xi_2}t+\boldsymbol{\xi_1}) + \ttbdelta M^{-1}K \frakKtwo *(\tilde{\frakKone}* \bmu + \boldsymbol{\xi_2}) = M^{-1}\boldsymbol{f}
		\end{multlined}
	\end{aligned}
\end{equation}
or
\begin{equation}
	\begin{aligned}
		\begin{multlined}[t] \taua	\bmu+ M^{-1} M_\aaa \tilde{\frakKone}* \bmu + M^{-1}K_\bbb 1*1*\tilde{\frakKone}*\bmu
			+\tau^a c^2 M^{-1}K  \frakKone* 1*\tilde{\frakKone}*\bmu\\+ \ttbdelta M^{-1}K \frakKtwo *\tilde{\frakKone}* \bmu  =  \boldsymbol{\tilde{f}}
		\end{multlined}
	\end{aligned}
\end{equation}
with the right-hand side
\begin{equation}
	\begin{aligned}
		\boldsymbol{\tilde{f}}=&\, 	\begin{multlined}[t]  M^{-1}\boldsymbol{f}-M^{-1} M_\aaa \boldsymbol{\xi_2}-M^{-1}K_\bbb(\boldsymbol{\xi_2} \frac{t^2}{2}+\boldsymbol{\xi_1}t +\boldsymbol{\xi_0})-\tau^a c^2 M^{-1}K  \frakKone*(\boldsymbol{\xi_2}t+\boldsymbol{\xi_1}) \\-\ttbdelta M^{-1}K \frakKtwo *\boldsymbol{\xi_2} 	\end{multlined}\\
		\in&\,  L^\infty(0, T). 
	\end{aligned}
\end{equation}
By \cite[Ch.\ 2, Theorem 3.5]{gripenberg1990volterra}, the system has a unique solution $\bfmu \in L^\infty(0,T)$. We then consider the problem
\begin{equation}
	\left \{	\begin{aligned}
		&	\frakKone*\bxittt =\bmu, \\
		& (\bfxi, \bxit, \bxitt)\vert_{t=0} =(\bxi_0, \bxi_1, \bxi_2)
	\end{aligned} \right.
\end{equation}
by rewriting it equivalently as
\begin{equation}
	\left \{	\begin{aligned}
		&	\bxitt = \tfrakKone* \bmu +\bxi_2 \in L^\infty(0,T), \\
		& (\bfxi, \bxit)\vert_{t=0} =(\bxi_0, \bxi_1),
	\end{aligned} \right.
\end{equation}
which has a unique solution $\bxi \in W^{2, \infty}(0,T)$. In this way we obtain existence of a unique approximate solution $\ppsi^{(n)} \in W^{2, \infty}(0,T; V_n)$. 
\section{Proof of Theorem~\ref{Thm:WellP_GFEIII_BKuznetsov}} \label{Appendix:WellPproofKB}
\noindent We present here the proof of the $\tau$-uniform well-posedness stated in Theorem~\ref{Thm:WellP_GFEIII_BKuznetsov}.
\begin{proof}
	 Let $\ppsi^* \in \calBBK$ with $\calBBK$ defined in \eqref{def_BBK}. Since this implies that $\ppsi^* \in \calXBK$, the smoothness assumptions on $\aaa$ and $\bbb$ in \eqref{smoothness_a_b} are fulfilled and the smallness assumption on $\aaa$ given in \eqref{smallness_a} follows by reducing $R>0$ (independently of $\tau$). 
	The non-degeneracy condition on $\aaa$ is fulfilled for small enough $R$ as well. 
	Furthermore, we have
	\begin{equation}
		\begin{aligned}
			&\|\calN(\nabla \ppsi^*, \nabla \ppsi^*_t)\|_{W^{1,1}(L^2)}\\
			\lesssim&\, \begin{multlined}[t] \| \nabla \ppsi^*\|_{L^\infty(L^4)}\| \nabla \ppsi_t^*\|_{L^1(L^4)} + \| \nabla \ppsit^*\|_{L^1(L^4)}\| \nabla \ppsit^*\|_{L^\infty(L^4)}+ \|\nabla \ppsi_t^*\|_{L^\infty(L^4)} \|\nabla \ppsitt^*\|_{L^1(L^4)} 
			\end{multlined}\\
			\leq& C(\Omega, T) R^2,
		\end{aligned}
	\end{equation}
	where in the last line we have relied on the embedding $\Hone \hookrightarrow L^4(\Om)$. By employing the bound for the linear problem, we obtain 
	\begin{equation}
		\begin{aligned}
			\|\ppsi\|_{\calXBK} 
			\leq&\,\begin{multlined}[t]
				C_1 \exp \left(C_2T(1+\|\bbb\|_{H^1(H^2)})\right)
				 (\|\ppsi_0\|^2_{H^3}+\|\ppsi_1\|^2_{H^3} + \tau^a  \| \ppsi_2\|^2_{H^2}+ \|f\|_{W^{1,1}(L^2)}^2) .
			\end{multlined}
		\end{aligned}
	\end{equation}
	Since
	\begin{equation}
		\begin{aligned}
			\|\bbb\|_{H^1(H^2)} \lesssim \|\ppsi^*_{t}\|_{L^2(H^2)} +\|\ppsi^*_{tt}\|_{L^2(H^2)}  \lesssim_T R,
		\end{aligned}
	\end{equation}
	we have
	\begin{equation}
		\begin{aligned}
			\|\ppsi\|_{\calXBK} 
			\leq&\, C_1 \exp \left(C_2T(1+TR)\right)(r^2+CR^4).
		\end{aligned}
	\end{equation}
	Therefore, $\ppsi \in \calBBK$ for sufficiently small data size $r$ and radius $R$. \\
	\indent Let $\calT \ppsi^{*} =\ppsi$ and $\calT v^*=v$. Set $\phi=\ppsi-v$ and $\phi^*= \ppsi^*-v^*$. Then the difference $\phi$ solves 
	\begin{equation}
		\begin{aligned}
			&\begin{multlined}[t]	\taua  \frakKone * \phi_{ttt}+\aaa(\ppsi_t^*) \phi_{tt}
				-c^2 \bbb(\ppsi_t^*) \Delta \phi
				-   {\tau^a} c^2 \frakKone* \Delta \phi_t -\ttbdelta \frakKtwo* \Delta \phi_{tt}\end{multlined}\\
			=&\, -2k_1 \phi^*_t v_{tt} -2k_2 \phi_t^*\Delta v-2 \kthree \nabla \phi^*\cdot \nabla u^*_{t}
			-2 \kthree \nabla v^*\cdot \nabla \phi^*_{t} := f
		\end{aligned}
	\end{equation}
	with homogeneous boundary and initial data.  Note that as before we cannot prove contractivity with respect to the $\|\cdot\|_{\calXBK}$ norm by exploiting the linear bound in $\calXBK$ (see \eqref{est_2_GFEIII_BK}), as the right-hand side of this equation does not belong to $W^{1,1}(0,T; H^1(\Omega))$. Instead, we test this equation with $\phi_{tt}$ and prove contractivity with respect to a lower-order norm. To this end, we rely on the following identity:
	\begin{equation}
		\begin{aligned}
			&-c^2\inttO \bbb(\ppsi_t^*)\Delta \phi \phi_{tt} \dxs \\
			=&\,\begin{multlined}[t] c^2\intO (1-2k_2 \ppsi_t^*(t))\Delta \phi(t) \phi_{t}(t) \dx-c^2\inttO (1-2k_2 \ppsi_t^*)\Delta \phi_t \phi_{t} \dxs\\+2k_2c^2\inttO  \ppsi_{tt}^*\Delta \phi \phi_{t} \dxs,
			\end{multlined}
		\end{aligned}
	\end{equation}
from which we have
	\begin{equation}
	\begin{aligned}
		&-c^2\inttO \bbb(\ppsi_t^*)\Delta \phi \phi_{tt} \dxs \\
		=&\, \begin{multlined}[t] 	 -c^2\intO (1-2k_2 \ppsi_t^*(t))\nabla \phi(t) \cdot \nabla \phi_{t}(t) \dx+2k_2 c^2\intO \nabla \ppsi_t^*(t) \cdot \nabla \phi(t) \phi_{t}(t) \dx\\
			+c^2\inttO (1-2k_2 \ppsi_t^*)\nabla \phi_t \cdot \nabla \phi_{t} \dxs-		2k_2c^2\inttO \nabla \ppsi_t^* \cdot \nabla \phi_t   \phi_{t} \dxs \\
			-2k_2c^2\inttO  \ppsi_{tt}^*\nabla  \phi  \cdot \phi_{t} \dxs  -2k_2c^2\inttO  \nabla \ppsi_{tt}^*\cdot \nabla  \phi \phi_{t} \dxs. 
		\end{multlined}	
	\end{aligned}
\end{equation}
	Similarly to the proof of uniqueness for the linear problem in Section~\ref{Sec:GFEIII}, we then have  
	\begin{equation}\label{est_contractivity_GFEIII_BlackstockKuznetsov}
		\begin{aligned}
			&\int_0^t \nLtwo{\sqrt{\aaa(\ppsi_t^*)}\phi_{tt}}^2\ds +  \|\sqrt{\bbb(\ppsi^*_t)}\nabla \phi_t\|^2_{L^\infty(L^2)}\\
			\lesssim&\, \begin{multlined}[t] \|f\|^2_{L^2(L^2)}+(1+\|\ppsi_t^*\|_{L^\infty(L^\infty)})\|\nabla \phi\|_{L^\infty(L^2)}\|\nabla \phi_t\|_{L^\infty(L^2)}\\+\|\nabla \ppsi^*_t\|_{L^\infty(L^\infty)}\|\nabla \phi_t\|_{L^\infty(L^2)}\|\phi_t\|_{L^\infty(L^2)}\\
				+\|\ppsi_t^*\|_{L^\infty(L^\infty)}\|\nabla \phi_t\|^2_{L^2(L^2)}+\|\nabla \ppsi_t^*\|_{L^\infty(L^\infty)}\|\nabla \phi_t\|_{L^2(L^2)}\|\phi_t\|_{L^2(L^2)}\\
				+\|\ppsi_{tt}^*\|_{L^2(L^\infty)} \|\nabla \phi\|_{L^2(L^2)}\|\phi_t\|_{L^\infty(L^2)}+\|\nabla \ppsi_{tt}^*\|_{L^2(L^4)}\|\nabla \phi\|_{L^\infty(L^2)}\|\phi_t\|_{L^2(L^4)}.
			\end{multlined}.
		\end{aligned}
	\end{equation}
	We note that
	\begin{equation}
		\begin{aligned}
			\|f\|^2_{L^2(L^2)}\lesssim&\, \begin{multlined}[t]	\|\phi_t^*\|_{L^\infty(L^4)}\|v_{tt}\|_{L^2(L^4)}+\|\phi_t^*\|_{L^\infty(L^2)}\|\Delta v\|_{L^\infty(L^\infty)}\\
				+\|\nabla \phi^*\|_{L^\infty(L^2)}\|\nabla u_t^*\|_{L^2(L^\infty)}+\|\nabla v^*\|_{L^\infty(L^\infty)}\|\nabla \phi_t^*\|_{L^2(L^2)}. 
			\end{multlined}
		\end{aligned}
	\end{equation}
	Therefore,  from \eqref{est_contractivity_GFEIII_BlackstockKuznetsov} by employing Young's and Gronwall's inequalities, and noting that $\ppsi^* \in \calBBK$, we conclude that
	\begin{equation}
		\begin{aligned}
			\int_0^t \nLtwo{\sqrt{\aaa(\ppsi_t^*)}\phi_{tt}}^2\ds +  \|\sqrt{\bbb(\ppsi^*_t)}\nabla \phi_t\|^2_{L^\infty(L^2)}
			\lesssim&\, R \left(\int_0^t \nLtwo{\phi_{tt}}^2\ds +  \|\nabla \phi_t\|^2_{L^\infty(L^2)}\right).
		\end{aligned}
	\end{equation}
	Thus, one  can guarantee strict contractivity of the mapping $\calT$ in the norm of the space $W^{1,\infty}(0,T; \Honezero) \cap H^1(0,T; L^2(\Omega))$  by reducing the radius $R$. 
\end{proof}
\end{appendices}

\bibliography{references}{}
\bibliographystyle{siam} 
\end{document}